\documentclass[a4paper,12pt]{amsart}
\usepackage{version}
\usepackage{amssymb}
\usepackage{amsfonts}
\usepackage{graphicx}
\usepackage{tikz}
\usetikzlibrary{arrows}

\usepackage{xcolor}
\usepackage[pagebackref]{hyperref}
\hypersetup{
   colorlinks,
    linkcolor={red!60!black},
    citecolor={blue!60!black},
    urlcolor={blue!90!black}
}

\usepackage{epstopdf}
\usepackage[english]{babel}
\usepackage{float}
\usepackage{cite}
\usepackage{tikz,pgf}
\usetikzlibrary{patterns,spy,angles}
\usepackage{tikz-cd}

\usepackage{geometry}
\newtheorem{theorem}{Theorem}[section]

\newtheorem{lemma}[theorem]{Lemma}
\newtheorem{corollary}[theorem]{Corollary}
\newtheorem{proposition}[theorem]{Proposition}

\newtheorem{claim}[theorem]{Claim}

\theoremstyle{definition}
\newtheorem{definition}[theorem]{Definition}

\newtheorem{notation}[theorem]{Notation}

\theoremstyle{remark}

\numberwithin{equation}{section}

\renewcommand\bigskip{\medskip}

\def\to{\rightarrow}

\def\cF{\mathcal{F}}
\def\N{\mathbb N}

\def\cT{\mathcal{T}}

\def\Q{\mathbb Q}

\def\R{\mathbb R}
\def\C{\mathbb C}
\def\RP{\mathbb{RP}^1}
\def\Z{\mathbb Z}

\def\cP{\mathcal{P}}

\def\cF{\mathcal{F}}
\def\dlp{d_{\rm LP}}
\def\cE{\mathcal{E}}

\def\cN{\mathcal{N}}
\def\mub{\bar{\mu}}

\def\-1{^{-1}}

\def\essinf{{{\rm ess}\inf}}

\def\bmu{\bar{\mu}}
\def\bnu{\bar{\nu}}
\newcommand{\ifi}{\mathtt{a}}

\newcommand{\io}{\mathtt{i}}
\newcommand{\ko}{\mathtt{k}}
\newcommand{\jo}{\mathtt{j}}

\newcommand{\jfi}{\mathtt{b}}

\DeclareMathOperator{\diam}{diam}
\DeclareMathOperator{\dimh}{\dim_H}
\usepackage{tabularx}
\usepackage{booktabs}

\begin{document}

\title{Resonance between planar self-affine measures}
\author{Aleksi Pyörälä}
\address
        {Department of Mathematics and Statistics \\ 
        P.O.\ Box 35 (MaD) \\ 
        FI-40014 University of Jyväskylä \\ 
         Finland}
\email{aleksi.pyorala@gmail.com}
\thanks{The research of this project was conducted as part of the author's doctoral studies at University of Oulu, and has been partly supported by the Research Council of Finland via the project \emph{GeoQuantAM: Geometric and Quantitative Analysis on Metric spaces}, grant no. 354241. I thank Ville Suomala and Meng Wu for their comments on an early version of the paper, Ariel Rapaport and the anonymous referee for many useful comments and suggestions that led to improved presentation of the paper, and Antti Käenmäki for pointing out an error in the proof of Proposition \ref{claim-ergodicproduct}.}
\subjclass[2020]{Primary 28A80; Secondary 37A10}
\keywords{Self-affine measures, Hausdorff dimension, convolution of measures, resonance}
\begin{abstract}
    We show that if $\lbrace \varphi_i\rbrace_{i\in \Gamma}$ and $\lbrace \psi_j\rbrace_{j\in\Lambda}$ are self-affine iterated function systems on the plane that satisfy strong separation, domination and irreducibility, then for any associated self-affine measures $\mu$ and $\nu$, the inequality $$\dim_{\rm H}(\mu*\nu) < \min \lbrace 2, \dim_{\rm H} \mu + \dim_{\rm H} \nu \rbrace$$ implies that there is algebraic resonance between the eigenvalues of the linear parts of $\varphi_i$ and $\psi_j$. This extends to planar non-conformal setting the existing analogous results for self-conformal measures on the line. 
\end{abstract}
\maketitle
\section{Introduction}\label{section1}
In the 1960s, Furstenberg conjectured that if $X,Y\subseteq [0,1]$ are closed sets invariant under multiplication by integers $m$ and $n$, respectively, then for any $s\neq 0$, the inequality
\begin{equation}\label{eq-dimensiondrop}
\dim(X+sY) < \min \lbrace 1,\dim X+\dim Y\rbrace
\end{equation}
implies that $\frac{\log m}{\log n} \in \Q$. Here and in the following, $\dim$ denotes the (lower) Hausdorff dimension for both sets and measures. This conjecture was one of several that aimed to capture the idea that if $\frac{\log m}{\log n}\not\in\Q$, then expansions in base $m$ and $n$ should have no common structure: Indeed, the right-hand side of \eqref{eq-dimensiondrop} is always an upper bound for $\dim(X+sY)$, while the strict inequality \eqref{eq-dimensiondrop} implies that many of the fibers $\lbrace (x,y)\in X\times sY:\ x+y=z\rbrace$ are large, which means that $X$ and $Y$ should have arithmetically similar structure in many places. The phenomenon \eqref{eq-dimensiondrop} is usually referred to as resonance: $X$ and $Y$ are said to \emph{resonate} if they satisfy \eqref{eq-dimensiondrop} for some $s$, and otherwise they are said to \emph{dissonate}.

It is also natural to ask if a similar phenomenon holds in a more general setting: For dynamical systems $X$ and $Y$, does \eqref{eq-dimensiondrop} imply some kind of algebraic or arithmetic similarity between the sets or the dynamics? The first result in this direction is due to Moreira \cite{Moreira1998} from 1998, who proved that for two self-conformal sets on the line, \eqref{eq-dimensiondrop} cannot hold in the presence of an irrationality assumption if one of the sets is totally non-linear. Recall that a set $K\subseteq \R$ is called self-conformal if 
\begin{equation}\label{eq-selfsimilar}
K = \bigcup_{i=1}^m f_i(K)
\end{equation}
for some $C^{1+\varepsilon}$-contractions $f_i$. 

In 2009, Peres and Shmerkin \cite{PeresShmerkin2009} established analogous results for sums of self-similar sets on the line: The dimension of the sum is maximal unless the contraction ratios of the defining similarities form an \emph{arithmetic set}. A set $A\subseteq \R$ is called arithmetic if $A\subseteq \alpha\N$ for some $\alpha\in \R$. Recall that a set is self-similar if it satisfies \eqref{eq-selfsimilar} with $f_i$ being similarities. An analogous result for convolutions of Cantor measures on the line was obtained shortly after by Nazarov, Peres and Shmerkin \cite{NazarovPeresShmerkin2012}: Indeed, by replacing sum with convolution in \eqref{eq-dimensiondrop}, one can formulate the concept of resonance for measures. 

A major breakthrough on the topic of resonance between dynamical systems was achieved by Hochman and Shmerkin in 2012 \cite{HochmanShmerkin2012}, who introduced a powerful method called the \emph{local entropy averages} to attack problems regarding projections (and therefore sums) of dynamically defined fractals. Hochman and Shmerkin managed to both prove the original conjecture of Furstenberg, and extend to the setting of measures the existing results on the sums of self-similar and self-conformal sets on the line. 

Namely, they proved that if $\lbrace f_i\rbrace_{i\in\Gamma}$ and $\lbrace g_j\rbrace_{j\in\Lambda}$ are families of $C^{1+\varepsilon}$-contractions on $\R$ that satisfy the open set condition, then for any associated self-conformal measures $\mu$ and $\nu$, 
\begin{equation}\label{eq-selfconformalresonance}
\dim(\mu*\nu) = \min \lbrace 1,\dim \mu+\dim\nu\rbrace
\end{equation}
unless the asymptotic contraction ratios of $\phi_i$ and $\psi_j$ form an arithmetic set. Recall that $\mu$ is self-conformal if 
\begin{equation}\label{eq-selfconformalmeasure}
\mu = \sum_{i\in\Gamma} p_i \cdot \mu\circ f_i^{-1}
\end{equation}
for a probability vector $(p_i)_{i\in\Gamma}$ and $f_i$ as above. Due to well-known variational principles, the result for self-conformal measures implies the result for sets as well. Recently, an analogue of \eqref{eq-selfconformalresonance} was proven by Bruce and Jin \cite{BruceJin2022} for a rich class of measures on homogeneous self-similar sets without requiring any separation conditions, in particular measures satisfying \eqref{eq-selfconformalmeasure} with $f_i$ being similarities. For self-conformal measures, \eqref{eq-selfconformalresonance} was recently verified without requiring any separation conditions by Bárány, Käenmäki, Wu and the author \cite{BaranyKaenmakiPyoralaWu2023}.

Perhaps surprisingly, almost nothing seems to be known of this phenomenon in higher dimensions. While there certainly exists literature on bounding the size of sums or convolutions from below, such as the famous inverse theorem of Hochman \cite{Hochman2015, Hochman2014}, it primarily focuses on showing that, for very general $X$ and $Y$, the sum (or convolution) $X+Y$ is \emph{strictly larger} than $X$, unless $X$ and $Y$ have a very special structure. See also \cite{RossiShmerkin2020, Fraser2022} and the references therein for progress in related phenomena. However, the existing results do not aim to capture the spirit of the phenomenon predicted by Furstenberg, that ``geometric resonance'' of dynamically defined sets should imply a kind of ``algebraic resonance'' between the dynamics. 

The purpose of the present work is to provide an extension of this principle to the planar setting. However, in formulating an extension beyond the line, one has to be careful: The direct extension, that 
\begin{equation}\label{eq-dimensiondropplane}
\dim(X+Y) < \min\lbrace 2, \dim X+ \dim Y\rbrace
\end{equation}
unless there is algebraic resonance between the dynamics of $X$ and $Y$, breaks down easily. Indeed, one can of course isometrically embed any sets $X$ and $Y$ on the line to the plane, and their sum will always have dimension at most $1$. It is also not difficult to construct examples of $X$ and $Y$ with dimension strictly greater than one by taking product sets. Thus, in order to expect \eqref{eq-dimensiondropplane} to imply algebraic resonance, one has to assume that $X$ and $Y$ are ``spread out'' in sufficiently many directions, in some sense. 

In this paper, we consider the size of $\mu*\nu$ when $\mu$ and $\nu$ are self-affine measures on the plane, that is, they satisfy \eqref{eq-selfconformalmeasure} with $f_i$ being invertible affine contractions on $\R^2$. Let $\RP$ denote the collection of one-dimensional subspaces of $\R^2$. For a $2 \times 2$-matrix $A$, let $|\lambda_1(A)| \leq |\lambda_2(A)|$ denote its eigenvalues. Let $A$ also denote the action induced by $A$ on $\RP$. We say that an iterated function system (an IFS for short) $\Phi= \lbrace f_i(x) = A_i x + a_i \rbrace_{i\in\Gamma}$ of affine contractions on $\R^2$ satisfies
\begin{itemize}
    \item[1)] the \emph{strong separation condition} if there exists a bounded open set $V \neq \emptyset$ such that for every $i\neq j\in \Gamma$, $f_i({\rm cl}(V)) \subseteq V$ and $f_i({\rm cl}(V)) \cap f_j({\rm cl}(V)) = \emptyset$,
    \item[2)] \emph{hyperbolicity} if there exists $n\in\N$ and $(i_1,\ldots, i_n) \in\Gamma^n$ such that $|\lambda_1(A_{i_1}\ldots A_{i_n})|<|\lambda_2(A_{i_1}\ldots A_{i_n})|$,
    \item[3)]  \emph{total irreducibility} if for every finite set $\Theta\subset\RP$, there exists $i\in\Gamma$ such that $A_i \Theta \neq \Theta$, and
    \item[4)] the \emph{domination condition} if there exists a multicone $\mathcal{C} \subsetneq \RP$, i.e. a finite union of closed cones, such that $A_i \mathcal{C} \subseteq {\rm int}(\mathcal{C})$ for each $i\in\Gamma$. 
\end{itemize}
Under the domination condition, $|\lambda_1(A_i)|<|\lambda_2(A_i)|$ for every $i\in\Gamma$ by \cite[Corollary 2.4]{BaranyKaenmakiMorris2020}, and total irreducibility is equivalent to the a-priori weaker condition of \emph{irreducibility}, that for every $\theta\in\RP$ there exists $i\in\Gamma$ such that $A_i\theta\neq \theta$; see \cite[Lemma 2.10]{BaranyKaenmakiYu2021}. 

\begin{theorem}\label{theorem-main}
Let $\Phi = \lbrace \varphi_i(x) = A_i x + a_i \rbrace_{i \in \Gamma}$ and $\Psi = \lbrace \psi_j(x) = B_j x + b_j \rbrace_{j \in \Lambda}$ be systems of affine contractions on $\R^2$ that satisfy the irreducibility, domination and strong separation conditions. Let $\mu$ and $\nu$ be fully supported self-affine measures associated to $\Phi$ and $\Psi$ with $\dim\mu\geq\dim\nu$. If
$$
\dim (\mu*\nu) < \min \lbrace 2,\dim \mu + \dim \nu \rbrace,
$$
then $\dim\mu>1>\dim\nu$ and $$
\lbrace \log|\lambda_1(A_i)|:\ i \in \Gamma\rbrace \cup \lbrace \log|\lambda_2(B_j)|:\ j \in \Lambda\rbrace
$$ is an arithmetic set.
\end{theorem}
For measures which do not satisfy the domination condition, we have a partial result which holds under a much weaker separation assumption: We say that $\Phi$ satisfies the \emph{exponential separation condition} if there exists a constant $c>0$ and a left-invariant metric $d$ on the group of invertible affine maps on the plane such that for any $n$ and any pair of words $(i_1,\ldots, i_n) \neq (j_1,\ldots, j_n)\in \Gamma^n$, we have $d(\varphi_{i_1}\circ\cdots\circ \varphi_{i_n}, \varphi_{j_1}\circ\cdots \circ \varphi_{j_n})> c^n$. 
\begin{theorem}\label{theorem-main-2}
    Let $\Phi = \lbrace \varphi_i(x) = A_i x + a_i \rbrace_{i \in \Gamma}$ and $\Psi = \lbrace \psi_j(x) = B_j x + b_j \rbrace_{j \in \Lambda}$ be systems of affine contractions on $\R^2$ that satisfy the hyperbolicity, total irreducibility and exponential separation conditions. Let $\mu$ and $\nu$ be fully supported self-affine measures associated to $\Phi$ and $\Psi$. Then
    $$
\dim(\mu*\nu) \geq \min \lbrace 1,\dim\mu\rbrace + \min\lbrace 1,\dim\nu\rbrace.
    $$
\end{theorem}
In particular, if $\dim(\mu*\nu) < \min\lbrace 2,\dim\mu+\dim\nu\rbrace$, then $\dim\mu>1>\dim\nu$ or $\dim\nu>1>\dim\mu$. In the proof of Theorem \ref{theorem-main-2} we do not confront the exponential separation condition directly; We only use it to gain access to a result of Bárány, Hochman and Rapaport \cite{BaranyHochmanRapaport2019}. We refer the reader to the papers \cite{BaranyHochmanRapaport2019, HochmanRapaport2021} for more discussion on the exponential separation condition.

Let us now comment on the assumptions of Theorem \ref{theorem-main}. The assumption of strong separation is classical in the study of iterated function systems, since it makes it possible to view the attractor as a dynamical system, giving access to a multitude of tools from ergodic theory. However, during recent years, much attention has been directed towards establishing existing results without assuming any separation conditions, and we expect it can be removed from our result as well.

The assumption of hyperbolicity ensures that the systems $\Phi$ and $\Psi$ are ``strictly self-affine''. This is crucial in our approach, since strictly self-affine measures have a very special tangential structure that we will heavily use. This structure is formulated in Proposition \ref{prop-parempijono} which is the main technical contribution of this paper, and we believe it can find applications outside this work as well. Without the hyperbolicity assumption, the contractions are similarities up to a change of basis, meaning that the self-affine measures are essentially self-similar. Of course, it would be interesting to find an analogous result for resonance between planar self-similar measures. 

The assumption of total irreducibility is our way to ensure that the measures $\mu$ and $\nu$ are ``spread out'' in sufficiently many directions, which is something one has to assume as explained in the preceding discussion. Without this assumption, it is easy to construct examples for which the conclusion of the theorem does not hold, by e.g. constructing measures on Bedford-McMullen carpets. However, we do not know if it is enough to assume total irreducibility for just one of the systems $\Phi$ and $\Psi$. 

If it happens that $\dim \mu \geq \dim \nu \geq 1$ or $1 \geq \dim \mu \geq \dim \nu$, then the strong separation (even exponential separation), hyperbolicity and total irreducibility are enough to ensure that $\mu*\nu$ has the maximal dimension, by Theorem \ref{theorem-main-2}. The case $\dim \mu > 1 > \dim \nu$ is more delicate, since now $\mu$ might be large enough to ``absorb'' some of $\nu$ in the convolution. However, if the arithmetic conclusion of Theorem \ref{theorem-main} does not hold, then we can use the domination condition to show that the measures $\mu$ and $\nu$ are ``spaced'' in such a manner that this kind of absorption cannot happen. However, it is likely that the requirement for domination is just a by-product of our method.

Finally, we remark that Theorems \ref{theorem-main} and \ref{theorem-main-2} combined with the variational principle of Bárány, Käenmäki and Rossi \cite[Proposition 2.4]{BaranyKaenmakiRossi2021} yield an analogous result for sets:
\begin{corollary}\label{corollary}
Let $\Phi = \lbrace \varphi_i(x) = A_i x + a_i \rbrace_{i \in \Gamma}$ and $\Psi = \lbrace \psi_j(x) = B_j x + b_j \rbrace_{j \in \Lambda}$ be systems of affine contractions on $\R^2$ that satisfy the domination, irreduciblity, and strong separation conditions. Let $X$ and $Y$ denote the attractors of $\Phi$ and $\Psi$, and suppose that $\dim X \geq \dim Y$. If
$$
\dim (X+Y) < \min \lbrace 2,\dim X  + \dim Y\rbrace,
$$
then $\dim X > 1 > \dim Y$ and there exists $n\in\N$ such that   
$$
\lbrace \log|\lambda_1(A_{i_1}\ldots A_{i_n})|:\ (i_1,\ldots,i_n) \in \Gamma^n\rbrace \cup \lbrace \log|\lambda_2(B_{j_1}\ldots B_{j_n})|:\ (j_1,\ldots,j_n) \in \Lambda^n\rbrace
$$
is an arithmetic set.
\end{corollary}

\begin{proof}
Suppose that $\dim(X+Y) < \min\lbrace 2,\dim X + \dim Y\rbrace$ and let $\varepsilon>0$ be small. Using \cite[Proposition 2.4]{BaranyKaenmakiRossi2021}, choose $n\in\N$ and fully supported self-affine measures $\mu$ and $\nu$ associated to the systems $\lbrace \varphi_{i_1}\circ \cdots \circ \varphi_{i_n}\rbrace_{(i_1,\ldots, i_n)\in\Gamma^n}$ and $\lbrace \psi_{j_1}\circ \cdots \circ \psi_{j_n}\rbrace_{(j_1,\ldots, j_n)\in\Lambda^n}$, such that $\dim \mu \geq \dim X - \varepsilon$ and $\dim \nu \geq \dim Y - \varepsilon$. Then, since $\mu*\nu$ is supported on $X+Y$, for small enough $\varepsilon$ we have
$$
\dim (\mu*\nu) \leq \dim(X+Y) < \min \lbrace 2,\dim  X+ \dim Y \rbrace - 3\varepsilon \leq \min \lbrace 2, \dim \mu + \dim\nu \rbrace - \varepsilon.
$$
The conclusion regarding eigenvalues now follows from Theorem \ref{theorem-main}. In addition, we have $\dim \mu > 1 > \dim \nu$, which implies that $\dim X > 1 \geq \dim Y$ if $\varepsilon$ was chosen small enough. If it were that $\dim Y = 1$, then $\dim \mu > 1 \geq \dim \nu \geq1-\varepsilon$, and Theorem \ref{theorem-main-2} would assert that $\dim(\mu*\nu) > 2-\varepsilon$ which is a contradiction. Thus $\dim Y < 1$ and the proof is complete.
\end{proof}

\subsection{On the proofs of Theorems \ref{theorem-main} and \ref{theorem-main-2}}
We begin with the proof of Theorem \ref{theorem-main-2}. Our proof is based on the local entropy averages of Hochman-Shmerkin \cite{HochmanShmerkin2012}: instead of proving directly that $\dim(\mu*\nu)$ is large, we will show that the average of the finite-scale entropies of $\mu_k *\nu_k$ over many $k\in\N$ is large, where $\mu_k$ and $\nu_k$ are \emph{magnifications} (at scale $2^{-k}$) of $\mu$ and $\nu$ along properly chosen filtrations of their supports. 

Indeed, magnifying both $\mu$ and $\nu$ along the ``cylinder ellipses'' $\varphi_{i_1} \circ \cdots \circ \varphi_{i_n}(B(0,1))$ and $\psi_{j_1} \circ\cdots \circ \psi_{j_n}(B(0,1))$, on the limit they both resemble orthogonal projections of the original measures, by the hyperbolicity assumption. Relying on the exponential separation condition and the strong projection theorem for self-affine measures \cite[Theorem 1.3]{BaranyHochmanRapaport2019} due to Bárány, Hochman and Rapaport, these magnifications will have entropy close to $\min\lbrace 1,\dim\mu\rbrace$ and $\min\lbrace 1,\dim\nu\rbrace$, respectively. Since the irreducibility assumption ensures that the ellipses on which the magnifications are supported on have major semi-axes pointing in different directions, their convolution has a product-like structure on the plane and thus has entropy close to $\min \lbrace 1,\dim\mu\rbrace + \min\lbrace 1,\dim\nu\rbrace$. This essentially concludes the proof of Theorem \ref{theorem-main-2}.

It is a formal consequence of Theorem \ref{theorem-main-2} that for self-affine measures $\mu$ and $\nu$ as in the statement with $\dim\mu\geq \dim \nu$, the inequality $\dim(\mu*\nu) < \min\lbrace 2,\dim\mu + \dim\nu\rbrace$ is only possible when $\dim\mu > 1 >\dim \nu$. This is the first statement of Theorem \ref{theorem-main}. Given measures $\mu$ and $\nu$ with $\dim\mu>1>\dim\nu$, we show that in the case where $\lbrace \log|\lambda_1(A_i)|:\ i \in \Gamma\rbrace \cup \lbrace \log|\lambda_2(B_j)|:\ j \in \Lambda\rbrace
$ is \emph{not} an arithmetic set, we will in fact have $\dim(\mu*\nu) = \min\lbrace 2,\dim\mu+\dim\nu\rbrace$, contradicting the assumption of strict inequality.
 
As in the proof of Theorem \ref{theorem-main-2}, we rely on the local entropy averages. However, a consequence of the assumption $\dim \mu >1 >\dim\nu$ is that the magnifications of $\mu$ along the cylinder ellipses can no longer store a sufficient amount of the dimension of $\mu$ for us to be able to reach the desired lower bound for the dimension of $\mu*\nu$. Instead, we will magnify $\mu$ along dyadic squares and $\nu$ along the cylinder ellipses. 

For any such magnifications $\mu_k$ and $\nu_k$ of $\mu$ and $\nu$, and any orthogonal projection $\pi$, applying the chain rule of entropy yields that the entropy of $\mu_k*\nu_k$ is equal to
\begin{equation*}
\text{entropy of}\ \pi\mu_k * \pi\nu_k\ +\ \text{conditional entropy of}\ \mu_k*\nu_k\ \text{w.r.t.}\ \pi 
\end{equation*}
which in turn is bounded from below by
\begin{equation}\label{eq-aim1}
\text{entropy of}\ \pi\mu_k * \pi\nu_k\ +\ \text{entropy of}\ \mu_k - \text{entropy of}\ \pi\mu_k,
\end{equation}
using again the chain rule and the fact that convolution never decreases entropy. This observation is put to use through our key geometric ingredient: We show that for nearly every $k$, the magnification $\mu_k$ has a \emph{fiber structure} in the sense that $\pi\mu_k$ is (close to) a \emph{slice measure} of $\mu$, for a properly chosen $\pi$. The precise form of this structure is stated in Proposition \ref{prop-parempijono}. Similar fiber structures have been previously observed for self-affine sets by Käenmäki, Koivusalo and Rossi \cite{KaenmakiKoivusaloRossi2017}, for self-affine measures on Bedford-McMullen carpets by Ferguson, Fraser and Sahlsten \cite{FergusonFraserSahlsten2015} and for self-affine measures with an additional projection condition by Kempton \cite{Kempton2015}. Combining this with the dimension conservation phenomenon that follows for planar self-affine measures from the Ledrappier-Young formula of Bárány \cite{Barany2015} and the general fact that when averaged over many $k$, the entropy of $\mu_k$ is close to the dimension of $\mu$, we see that upon averaging, \eqref{eq-aim1} is in fact close to
\begin{equation}\label{eq-aim2}
    \text{entropy of}\ \pi\mu_k * \pi\nu_k\ +\ 1,
\end{equation}
recalling that we are assuming $\dim \mu > 1$. 

Thus, if we manage to show that for most $k$, for the measures $\pi\mu_k$ and $\pi\nu_k$ \emph{on the line} we have 
\begin{equation}\label{eq-aim3}
\text{entropy of}\ \pi\mu_k*\pi\nu_k \geq \min \lbrace 1,\text{entropy of}\ \pi\mu_k\ +\ \text{entropy of}\ \pi\nu_k\rbrace - o(1),
\end{equation}
where $o(1)$ denotes a quantity that vanishes as the scale of the entropy decreases, then we have that the entropy of $\mu_k*\nu_k$ when averaged over many $k$ is at least
\begin{align*}
&\min \lbrace 2,\text{entropy of}\ \pi\mu_k\ +\ 1\ +\ \text{entropy of}\ \pi\nu_k\rbrace-o(1)\\
=\ &\min\lbrace 2, \dim \mu + \dim \nu \rbrace-o(1)
\end{align*}
by \eqref{eq-aim2}, another application of the dimension conservation of \cite{Barany2015} and the projection theorem of \cite{BaranyHochmanRapaport2019}. This would conclude the proof.

Proving \eqref{eq-aim3} is the part where the assumption on the eigenvalues of $A_i$ and $B_j$ steps in, and where the domination condition is most heavily utilized. According to the classical projection theorem of Marstrand (combined with some known exact-dimensionality results), \eqref{eq-aim3} would hold if we were allowed to scale either of the measures $\pi\mu_k$ or $\pi\nu_k$ by a random real number. Because of this, if we were able to show that the sequence $(\pi\mu_k \times \pi\nu_k)_{k\in\N}$ equidistributes for a distribution on the space of measures which is jointly invariant under scaling in horizontal and vertical directions, \eqref{eq-aim3} would follow from an application of Fubini and Marstrand's theorem. 

The first step in establishing the equidistribution of $(\pi\mu_k \times \pi\nu_k)_{k\in\N}$ is the observation that since $\pi\mu_k$ is (close to) a slice of $\mu$ and $\pi\nu_k$ is an orthogonal projection of $\nu$, they both can be expressed as images of points in certain underlying continuous-time dynamical systems $\mathcal{Z}_\Phi = (Z_\Phi, \lbrace T_t\rbrace_{t\in\R}, \lambda_\Phi)$ and $\mathcal{Z}_\Psi = (Z_\Psi, \lbrace S_t\rbrace_{t\in\R}, \lambda_\Psi)$ defined in Section \ref{section6}. As in \cite{BruceJin2022, HochmanShmerkin2012}, this allows us to describe the sequence $(\pi\mu_k  \times \pi\nu_k)_{k\in\N}$ as an image of an orbit of a point in the product system $\mathcal{Z}_\Phi\times\mathcal{Z}_\Psi$, and so the study of equidistribution of $(\pi\mu_k  \times \pi\nu_k)_{k\in\N}$ in the space of probability measures is reduced to the study of equidistribution of orbits in $\mathcal{Z}_\Phi\times\mathcal{Z}_\Psi$. In our setting, the study of equidistribution properties of $\mathcal{Z}_\Phi\times\mathcal{Z}_\Psi$ requires methods which are very different from those of \cite{BruceJin2022, HochmanShmerkin2012}, where the systems $\Phi$ and $\Psi$ were assumed to be self-similar, and which, to our knowledge, have not appeared before in a fractal-geometric setting. 

In the proof of Proposition \ref{claim-ergodicproduct}, we discover a relation between the eigenvalues of the systems $\mathcal{Z}_\Phi$ and $\mathcal{Z}_\Psi$ and the eigenvalues of the linear parts of the functions in $\Phi$ and $\Psi$. Using this relation we find out that under the assumption that $\lbrace \log|\lambda_1(A_i)|:\ i \in \Gamma\rbrace \cup \lbrace \log|\lambda_2(B_j)|:\ j \in \Lambda\rbrace
$ is not an arithmetic set, the flows $\mathcal{Z}_\Phi$ and $\mathcal{Z}_\Psi$ have no common eigenvalues and so, applying methods from ergodic theory, we show that typical orbits in $\mathcal{Z}_\Phi\times\mathcal{Z}_\Psi$ will equidistribute for the product measure $\lambda_\Phi \times \lambda_\Psi$. Consequently, the sequence $(\pi\mu_k \times \pi\nu_k)_{k\in\N}$ can be expressed as an image of an orbit which equidistributes for $\lambda_\Phi \times \lambda_\Psi$, and so the sequence $(\pi\mu_k \times \pi\nu_k)_{k\in\N}$ will equidistribute for a distribution with the desired joint invariance under scaling in vertical and horizonal directions. 

\subsection{Structure}
In Section \ref{section2} we introduce our setting more rigorously, and collect some general known results and short lemmas on self-affine measures, dynamical systems and random matrix products. Section \ref{section3} is devoted to translating the local entropy averages machinery of \cite{HochmanShmerkin2012} to our setting, while in Section \ref{section4} we state our main technical results, and use these to conclude the proof of Theorem \ref{theorem-main}. Section \ref{section5} is perhaps the most technical one, devoted to the proof of the main geometric result, Proposition \ref{prop-parempijono}. The arguments here were inspired by the work of Kempton \cite{Kempton2015}. Finally, in Section \ref{section6} we investigate the dynamics of the sequences of magnifications of $\mu$ and $\nu$, and prove the required lower bounds for their average entropies. 
\clearpage

\begin{table}[H]
\caption{Notation}
\begin{tabularx}{\textwidth}{@{}p{0.4\textwidth}X@{}}
\toprule
  $\Gamma, \Lambda$ & Finite alphabets \\
  $\io,\jo,\ko,\ldots $ & Infinite words of $\Gamma^\N$, $\Lambda^\N$\\
  $\ifi, \jfi,\ldots$ & Finite words \\
  $\Phi = \lbrace \varphi_i\rbrace_{i\in\Gamma} , \Psi =\lbrace \psi_j \rbrace_{j\in\Lambda}$ & Systems of affine invertible contractions \\
  $\bar{\mu},\bar{\nu}$ & Bernoulli measures on $\Gamma^\N$ and $\Lambda^\N$ \\
  $\Pi$ & The natural projections $\Gamma^\N \to\R^2$ and $\Lambda^\N\to\R^2$ \\
  $\mu, \nu$   & Projections of $\bar{\mu}$, $\bar{\nu}$ through $\Pi$  \\
  $\mu_D$ & Normalized restricition on $D$ \\
  $\mu^D$ & Measure $\mu_D$ linearly rescaled onto $[-1,1)^d$ \\
  $\pi_\theta$ & Orthogonal projection onto the line $\theta$ \\ 
  $\pi^i$ & Projection onto the $i$th coordinate \\
  $\mu_{\io,\theta}$ & A ``slice measure'' of $\mu$; see \eqref{eq-muitheta} \\
  $S_t, T_x$ & Scaling by $2^t$ and translation by $-x$, respectively\\
  $R_\theta$ & A rotation taking $\theta$ onto the $y$-axis \\
  $Y_{\io,\theta,r_1,r_2}$ & A rectangle with side lengths $2^{-r_2}\leq 2^{-r_1}$ \\
  $H_{Y_{\io,\theta,r_1,r_2}}$ & An affine map rescaling $Y_{\io,\theta,r_1,r_2}$ onto $R_\theta^{-1}[-1,1]^2$ \\
  $A = UDV^{-1}$ & The singular value decomposition \\
  $L_{\io,\theta,k}, F_{\theta,\io}^k$ & Non-singular linear maps $\R^2\to\R^2$\\
  $Q_{\io,\theta,k}$ & $Y_{\sigma^{\ell_k}\io, \theta, kN + \log \alpha_1 (\io|_{\ell_k}), kN + \log\alpha_2(\io|_{\ell_k})}$ \\
  $E_{\io,\theta,k}$ & The largest ellipse contained in $Q_{\io,\theta,k}$ \\
  $\theta(A)$ & The direction of the longer semi-axis of  $A(B(0,1))$ \\
  $\theta(\io), \theta^-(\io), \theta^*(\io)$ & Limit orientations; see Lemma \ref{lemma-suuntasuppeneekaikkialla} \\
  $\Vert A \Vert$ & Operator norm of $A$ \\
  $A|_\theta$ & The restriction of $A$ onto the line $\theta$ \\
  $\alpha_1(A) \leq \alpha_2(A)$ & The singular values of $A$ \\
  $|\lambda_1(A)|\leq |\lambda_2(A)|$ & The eigenvalues of $A$ \\
  $i_k(\io),\ i_k(\jo)$ & Stopping times; $\Vert A_{\io|_{i_k}} \Vert \approx \Vert B_{\io|_{i_k}}\Vert \approx 2^{-kN}$ \\
  $\ell_k = \ell_k(\io)$ & An increasing sequence; see \eqref{eq-lk}\\
  $\mu_{\io|_{i_k}},\ \nu_{\jo|_{i_k}},\ \mu^{D_{kN}(\Pi(\io))}$ & Magnifications of $\mu$ and $\nu$; see Notation \ref{notation-magnifications}\\
  $\rho(n, (\io,\theta))$ & The reflection done by $A_{\io|_n}^{-1}$ on the line $\theta$\\
  $\mathcal{Z}_\Phi=(Z_\Phi, \lambda_\Phi, \mathcal{T}_s)$ & A flow related to $(\pi_2 \mu^{D_{kN}(\Pi(\io))})_{k\in\N}$; see Section \ref{section6} \\
  $\mathcal{Z}_\Psi=(Z_\Psi, \lambda_\Psi, \mathcal{T}_s)$ & A flow related to $(\nu_{\jo|_{i_k}})_{k\in\N}$; see Section \ref{section6} \\
  $\mathcal{Z}_\Phi', \mathcal{Z}_\Psi'= (Z_\Psi', \lambda_\Psi', \mathcal{T}_s)$ & Projections of $\mathcal{Z}_\Phi$ and $\mathcal{Z}_\Psi$ through $\pi^{1,2,4}$ \\
  $F: \mathcal{Z}_\Phi \to \mathcal{P}(\R^2)$ & Coding of $\pi^2\mu^{D_{kN}(\Pi(\io))}$ via $\mathcal{Z}_\Phi$; see \eqref{eq-definitionF} \\
  $G: \mathcal{Z}_\Psi \to \mathcal{P}(\R^2)$ & Coding of $\pi^2 \nu_{\jo|_{i_k}}$ via $\mathcal{Z}_\Psi$; see \eqref{eq-definitionG}\\
  $F', G'$ & Functions $F$ and $G$ without reflection \\
\bottomrule
\end{tabularx}
\end{table}
\newpage

\section{Preliminaries}\label{section2}

In this paper, a measure refers to a Radon measure on a metrizable topological space. The notation $\cP(X)$ stands for probability measures on the space $X$. For a measure $\mu$ on $X$ and a subset $Y\subseteq X$, $\mu|_Y$ denotes the restriction of $\mu$ onto $Y$, $\mu_Y := \mu(Y)^{-1} \mu|_Y$ the normalized restriction when $\mu(Y) >0$. For a measurable function $f$, let $f\mu:= \mu\circ f^{-1}$ denote the push-forward. The space of probability measures is always equipped with the weak-$^*$ topology which we metrize using the L{\'e}vy-Prokhorov metric $d_{\rm LP}$, 
$$
\dlp(\mu, \nu) = \inf \lbrace \varepsilon>0:\ \mu(A) \leq\nu(A^\varepsilon)+\varepsilon,\ \nu(A)\leq\mu(A^\varepsilon)+\varepsilon\ \text{for all Borel}\ A\rbrace,
$$
where $A^\varepsilon$ denotes the open $\varepsilon$-neighbourhood of $A$. We measure the size of measures most often with the lower Hausdorff dimension, 
$$
\dim \mu = \dimh \mu = \inf \lbrace \dimh(E):\ \mu(E) > 0\rbrace,
$$ 
and occasionally with lower and upper local dimensions, 
$$
\underline{\dim}_{\rm loc}\mu(x) = \liminf_{r\to 0} \frac{\log \mu(B(x,r))}{\log r}, \qquad \overline{\dim}_{\rm loc}\mu(x) = \limsup_{r\to 0} \frac{\log \mu(B(x,r))}{\log r},
$$
where $B(x,r)$ denotes the closed ball centered at $x$ and of radius $r$. The measure $\mu$ is called exact dimensional if there exists $c\geq 0$ such that $\underline{\dim}_{\rm loc}\mu(x)=\overline{\dim}_{\rm loc}\mu(x) = c$ almost everywhere. The following connection between Hausdorff and lower local dimension is well-known, see e.g. \cite[Theorem 1.2]{FanLauRao2002}: For any measure $\mu$, 
$$
\dim \mu = \essinf_{x\sim \mu} \underline{\dim}_{\rm loc}\mu(x).
$$ 
In particular, if $\mu$ is exact dimensional, then $\dim\mu = \underline{\dim}_{\rm loc}\mu(x) = \overline{\dim}_{\rm loc}\mu(x)$ for $\mu$-almost every $x$.

\subsection{Symbolic dynamics}

Let $\Gamma$ be a finite set with $\#\Gamma \geq 2$, and let $\Phi = \lbrace \varphi_i \rbrace_{i\in \Gamma}$ be an iterated function system of contractions on $\R^d$. It is well-known that there exists a unique non-empty compact set $K$, called the \emph{attractor} of $\Phi$, such that 
$$
K = \bigcup_{i\in\Gamma}\varphi_i(K).
$$
We refer to Falconer's book \cite{Falconer1990} for standard properties of iterated function systems. If the functions $\varphi_i$ are of the form $\varphi_i(x) = A_i x + a_i$, where $A_i: \R^d\to\R^d$ are invertible linear maps with $\Vert A_i \Vert < 1$ and $a_i\in\R^d$, then the IFS $\Phi$ is called \emph{self-affine} and its attractor a self-affine set. 

Write $\Gamma^* = \bigcup_{n}\Gamma^n$ for the set of finite words composed of elements of $\Gamma$. For a finite word $\ifi = i_0 i_1 \ldots i_n\in \Gamma^*$ we write $\varphi_\ifi = \varphi_{i_0}\circ\cdots \circ \varphi_{i_n}$. Let $|\ifi|$ denote the number of elements in $\ifi$. For finite words $\ifi$ and $\jfi$, let $\ifi\jfi \in \Gamma^{|\ifi|+|\jfi|}$ denote their concatenation.

For a word $\io\in\Gamma^\N \cup \Gamma^*$ and and integer $k \leq |\io|$, let $\io|_k\in\Gamma^k$ denote its projection to the first $k$ coordinates. When $k=0$, set $\io|_k=\emptyset$. For $\io,\jo\in\Gamma^\N$, let $\io\wedge\jo := \io|_k$, where $k$ is the largest integer for which $\io|_k = \jo|_k$. Define a distance $d$ on $\Gamma^\N$ by 
$$
d(\io,\jo) = 2^{-|\io\wedge\jo|}
$$
for every $\io,\jo\in\Gamma^\N$. For a finite word $\ifi \in \Gamma^*$, write $[\ifi]$ for the cylinder set $\lbrace \io\in\Gamma^\N:\ \io|_{|\ifi|} = \ifi\rbrace$. It is not difficult to see that the cylinder sets are closed and open in the topology generated by $d$.

It is sometimes convenient to consider the two-sided sequence space $\Gamma^\Z$. For $\io = \ldots i_{-2}i_{-1};i_0i_1i_2\ldots \in\Gamma^\Z$ and $m\leq n \in \Z$, write $\io|_m^n = i_{m}i_{m+1}\ldots i_n$. The metric $d$ extends to $\Gamma^\Z$ by replacing $\io|_k$ by $\io|_{-k}^k$ in the definition of $\io\wedge\jo$. The cylinder sets of $\Gamma^\Z$ are given by $[\io]_m^n := \lbrace \jo \in \Gamma^\Z:\ \jo|_m^n = \io|_m^n \rbrace$. There is a natural surjection $\Gamma^\Z \to \Gamma^\N$ given by the restriction to the ``positive coordinates'', $\ldots i_{-1};i_0i_1\ldots =: \io \mapsto \io^+ := i_0 i_1\ldots$. Similarly, we define the projection to the ``negative coordinates'' by $\io^- := i_{-1} i_{-2}\ldots \in \Gamma^\N$.

We let $\sigma$ denote the left-shift on both $\Gamma^\N$ and $\Gamma^\Z$, given by $\sigma(i_0i_1\ldots) = i_1i_2\ldots$ and $\sigma(i_{-1};i_0i_1\ldots) = \ldots i_{0};i_1i_2\ldots$. The tuples $(\Gamma^\N, \sigma)$ and $(\Gamma^\Z,\sigma)$ are referred to as the one-sided and two-sided shift spaces, respectively. For any $\sigma$-invariant probability measure $\nu$ on $\Gamma^\Z$, there is a unique $\sigma$-invariant probability measure on $\Gamma^\Z$ which we also denote by $\nu$, given by $\nu([\io]_m^n) := \nu([\io|_m^n])$ for each $\io\in\Gamma^\Z$, $m\leq n\in\Z$. This is referred to as the natural extension of $\nu$. 

\subsection{Linear algebra and random matrix products}
Let $A$ be a real-valued $2\times 2$-matrix. Recall that the eigenvalues of $A$ are denoted by $\lambda_1(A)$ and $\lambda_2(A)$ with $|\lambda_1(A)|\leq|\lambda_2(A)|$. A matrix $A$ maps the unit ball onto an ellipse with major semi-axes of length $\alpha_1(A) \leq \alpha_2(A)$, where $\alpha_1(A), \alpha_2(A)$ are the singular values of $A$. Let $\theta(A) \in \RP$ denote the line parallel to the longer semi-axis of this ellipse. The conditions of total irreducibility, hyperbolicity and domination defined in Section \ref{section1} for affine iterated function systems extend in the obvious way for a tuple of matrices $\lbrace A_i\rbrace_{i\in\Gamma}$.  

For a finite word $\ifi = i_0i_1\ldots i_n\in\Gamma^*$, write $A_\ifi = A_{i_0}A_{i_1} \ldots A_{i_n}$ and $A_\ifi^{-1} = (A_\ifi)^{-1} = A_{i_n}^{-1} \ldots A_{i_1}^{-1} A_{i_0}^{-1}$. We define a distance $d$ on $\RP$, given by the smaller angle between lines.

\begin{lemma}\label{lemma-suuntasuppeneeprelminary}
    Let $\lbrace A_i \rbrace_{i\in\Gamma}$ be a collection of $2\times 2$-matrices satisfying the total irreducibility and hyperbolicity conditions, and let $\bmu$ be a fully supported Bernoulli measure on $\Gamma^\N$. For $\bmu$-almost every $\io\in\Gamma^\N$, there exists $\theta(\io)\in\RP$ such that any accumulation point of $$\lbrace \Vert A_{\io|_n}\Vert^{-1} A_{\io|_n}:\ n\in\N\rbrace$$ in the strong operator topology is a rank-$1$ matrix with range $\theta(\io)$. If $\lbrace A_i \rbrace_{i\in\Gamma}$ additionally satisfies the domination condition, then this is true for every $\io\in\Gamma^\N$.
\end{lemma}
\begin{proof}
    The statement for almost every $\io$ is included in \cite[Theorem III.3.1]{BougerolLacroix1985}. If $\lbrace A_i \rbrace_{i\in\Gamma}$ satisfies the domination condition, it is proved in \cite[Theorem B]{BochiGourmelon2009} that $\frac{\alpha_1(A_{\io|_n})}{\alpha_2(A_{\io|_n})}\to 0$ as $n\to \infty$, uniformly for every $\io\in\Gamma^\N$, whence any accumulation point of $\lbrace \Vert A_{\io|_n}\Vert^{-1} A_{\io|_n}:\ n\in\N\rbrace$ is a rank-$1$ matrix. For a proof that all of these accumulation points have the same range $\theta(\io)$, we refer to \cite[Lemma 2.1]{Rossi2021}.
\end{proof}

\begin{lemma}\label{lemma-suuntasuppeneekaikkialla}
Let $\lbrace A_i \rbrace_{i\in\Gamma}$ be a collection of $2\times 2$-matrices satisfying the total irreducibility and hyperbolicity conditions, and let $\bmu$ be a fully supported Bernoulli measure on $\Gamma^\N$. For $\bmu$-almost every $\io\in\Gamma^\N$, the following limits exist:
\begin{align*}
\theta(\io) &= \lim_{n \to \infty} \theta(A_{\io|_n}),\\
\theta^-(\io) &:= \lim_{n\to\infty} \theta(A_{i_0}^{-1} A_{i_1}^{-1}\ldots A_{i_n}^{-1}), \\
\theta^*(\io) &:= \lim_{n \to\infty} \theta(A_{i_0}^{*}A_{i_1}^{*}\ldots A_{i_n}^*).
\end{align*}
If $\lbrace A_i \rbrace_{i\in\Gamma}$ additionally satisfies the domination condition, then the above limits exist for every $\io\in\Gamma^\N$, the convergences are uniform, and the functions $\io \mapsto \theta(\io)$, $\io\mapsto \theta^-(\io)$ and $\io\mapsto \theta^*(\io)$ are Hölder continuous. 
\end{lemma}
\begin{proof}
    It is easy to see that the family $\lbrace A_i \rbrace_{i\in\Gamma}$ being totally irreducible, hyperbolic or dominated implies the same for the families $\lbrace A_i^{-1} \rbrace_{i\in\Gamma}$ and $\lbrace A_i^* \rbrace_{i\in\Gamma}$, whence the existence of the limits follows immediately from Lemma \ref{lemma-suuntasuppeneeprelminary}. The other statements are proved in \cite[Lemma 2.1]{Rossi2021}.
\end{proof}

\begin{lemma}\label{cor-almostcontraction}
    Let $\lbrace A_i \rbrace_{i\in\Gamma}$ be a collection of $2\times 2$-matrices satisfying the total irreduciblity and hyperbolicity conditions. Then for $\bmu$-almost every $\io\in\Gamma^\N$ and for every $\theta_1,\theta_2 \in \RP\setminus \lbrace \theta(\io)\rbrace$ and $\theta_3,\theta_4\in\RP\setminus\lbrace\theta(\io)^\perp\rbrace$,
    $$
\lim_{k\to\infty} d(A_{\io|_k}^{-1} \theta_1,\ A_{\io|_k}^{-1} \theta_2) = \lim_{k\to\infty} d(A_{\io|_k}^*\theta_3,\ A_{\io|_k}^*\theta_4) = 0.
    $$
    If $\lbrace A_i \rbrace_{i\in\Gamma}$ additionally satisfies the domination condition, then this holds for every $\io\in\Gamma^\N$.
\end{lemma}

\begin{proof}
    By \cite[Lemma III.4.2]{BougerolLacroix1985}, 
    $$
    \frac{d(A_{\io|_k}^{-1} \theta_1,\ A_{\io|_k}^{-1}\theta_2)}{d(\theta_1,\theta_2)} \leq  \frac{\alpha_1(A_{\io|_k})^{-1} \alpha_2(A_{\io|_k})^{-1}}{\Vert A_{\io|_k}^{-1}e_{\theta_1}\Vert \Vert A_{\io|_k}^{-1}e_{\theta_2}\Vert},
    $$
    where $e_{\theta_1}$ and $e_{\theta_2}$ are unit vectors of $\theta_1$ and $\theta_2$. For $\theta\in\RP$, let $\pi_\theta: \R^2\to \theta$ denote the orthogonal projection onto $\theta$. From the singular value decomposition, it is not difficult to see that for any invertible $2\times2$-matrix $A$ and any $x\in\R^2$, we have $\Vert A^{-1} x \Vert \geq \alpha_1(A)^{-1} \Vert \pi_{\theta(A)^{\perp}} x\Vert$, cf. \cite[Proof of Proposition III.3.2]{BougerolLacroix1985}. Inserting this into the above estimate we obtain 
    $$
    \frac{d(A_{\io|_k}^{-1} \theta_1,\ A_{\io|_k}^{-1}\theta_2)}{d(\theta_1,\theta_2)} \leq \frac{\alpha_1(A_{\io|_k})}{\alpha_2(A_{\io|_k})} \frac{1}{\Vert \pi_{\theta(A_{\io|_k})^\perp} e_{\theta_1}\Vert \Vert \pi_{\theta(A_{\io|_k})^\perp} e_{\theta_2}\Vert}.
    $$
    Since $\theta_1,\theta_2 \in \RP\setminus\lbrace \theta(\io)\rbrace$, it follows from Lemma 
\ref{lemma-suuntasuppeneekaikkialla} that the vectors $\pi_{\theta(A_{\io|_k})^\perp} e_{\theta_1}$ and $\pi_{\theta(A_{\io|_k})^\perp} e_{\theta_2}$ remain uniformly bounded away from $0$ for large $k$, for $\mu$-almost every (or every, if the domination condition holds) $\io\in\Gamma^\N$. Since $\lim_{k\to\infty} \frac{\alpha_1(A_{\io|_k})}{\alpha_2(A_{\io|_k})} = 0$ for $\bmu$-almost every $\io$ (or every, if domination is assumed) by Lemma \ref{lemma-suuntasuppeneeprelminary}, it follows that $\lim_{k\to\infty} d(A_{\io|_k}^{-1} \theta_1,\ A_{\io|_k}^{-1}\theta_2) = 0$. That $\lim_{k\to\infty} d(A_{\io|_k}^*\theta_3,\ A_{\io|_k}^*\theta_4) = 0$ is proved analogously, using the observation that $\Vert A^* x \Vert \geq \alpha_2(A) \Vert \pi_{\theta(A)} x \Vert$ for any $2\times2$-matrix $A$ and $x\in \R^2$. 
\end{proof}

\begin{lemma}\label{lemma-differentangles}
Let $\lbrace A_i \rbrace_{i\in\Gamma}$ be a collection of $2\times 2$-matrices satisfying the total irreduciblity and hyperbolicity onditions. Then for any $\theta\in\RP$ and for $\bmu$-almost every $\io\in\Gamma^\N$, we have $d(\theta(\io), \theta)>0$, $d(\theta^-(\io),\theta)>0$ and $d(\theta^*(\io),\theta) > 0$.
\end{lemma}
\begin{proof}
This is proved in \cite[Theorem III.3.1]{BougerolLacroix1985}.
\end{proof}

For a matrix $A$, let $v_i(A)$ denote the eigenspace associated to $\lambda_i(A)$ for $i=1,2$. When $A$ is \emph{hyperbolic}, that is, $|\lambda_1(A)| < |\lambda_2(A)|$, we have $v_i(A) \in \RP$ for $i=1,2$. The following observation follows immediately from the eigenvalue decomposition.

\begin{lemma}\label{lemma-A^n}
For a hyperbolic matrix $A$, $\lim_{n \to \infty} \theta(A^n) = v_2(A)$.
\end{lemma}

\subsection{The Furstenberg measures}

Let $\Phi = \lbrace \varphi_i(x) = A_i x + a_i\rbrace_{i\in\Gamma}$ be an affine iterated function system satisfying the total irreducibility and hyperbolicity conditions. It is due to Furstenberg \cite{Furstenberg1963} that for any probability vector $(p_i)_{i\in\Gamma}$, there exist unique and non-atomic probability measures $\mu_F$ and $\mu_F^*$ on $\RP$, called the \emph{Furstenberg measures}, that satisfy $\mu_F = \sum_{i\in\Gamma} p_i \mu_F\circ A_i$ and $\mu_F^* = \sum_{i\in\Gamma} p_i \mu_F \circ (A_i^*)^{-1}$. Proof of the existence of these measures, and an extensive discussion on their properties, can also be found in the book of Bougerol and Lacroix \cite[Theorem II.4.1]{BougerolLacroix1985}. If $\bmu := p^\N$ denotes the Bernoulli measure on $\Gamma^\N$ with marginal $(p_i)_{i\in\Gamma}$, the product measures $\bmu\times \mu_F$ and $\bmu\times \mu_F^*$ are invariant and ergodic under the maps
\begin{align*}
M: (\io,\theta) \mapsto (\sigma\io, A_{i_0}^{-1}\theta), \\
M_*: (\io,\theta) \mapsto (\sigma\io, A_{i_0}^* \theta),
\end{align*}
respectively; see for example \cite[Section 1.4]{BenoistQuint2016}. When $\Phi$ satisfies the domination condition, the measures $\bmu\times \mu_F$ and $\bmu\times\mu_F^*$ are images of the Bernoulli measure $p^\Z$ on $\Gamma^\Z$ through the factor maps $\io \mapsto (\io^+, \theta^-(\io^-))$ and $\io \mapsto (\io^+, \theta^*(\io^-))$, which easily implies the existence, invariance and ergodicity of $\bmu\times \mu_F$ and $\bmu\times\mu_F^*$. It follows from Birkhoff's ergodic theorem that the measures $\mu_F$ and $\mu_F^*$ are given by
\begin{align*}
\mu_F &= \lim_{n\to\infty} \frac{1}{n}\sum_{k=1}^n \delta_{A_{i_n}^{-1} \ldots A_{i_0}^{-1}\theta}, \\
\mu_F^* &=\lim_{n\to\infty} \frac{1}{n}\sum_{k=1}^n \delta_{A_{i_n}^{*} \ldots A_{i_0}^{*}\theta}
\end{align*}
for almost every $(\io,\theta)\in\Gamma^\N\times\RP$. In particular, when $\Phi$ satisfies the domination condition, it follows from Lemmas \ref{lemma-suuntasuppeneeprelminary} and \ref{lemma-suuntasuppeneekaikkialla} that $\mu_F$ and $\mu_F^*$ are supported on the closed sets $\theta^-(\Gamma^\N)$ and $\theta^*(\Gamma^\N) = \theta(\Gamma^\N)^\perp$, respectively.

\subsection{Self-affine measures}

Let $\Phi$ be a self-affine iterated function system, and let $K$ denote its attractor. Let $\Pi: \Gamma^\N \to K$ denote the surjection 
$$
\io \mapsto \lim_{k\to\infty}\varphi_{\io|_k}(0)
$$
which we call the natural projection. Fix a probability vector $p = (p_i)_{i \in \Gamma}$ and let $\bar{\mu} = p^\N$ denote the Bernoulli measure on $\Gamma^\N$ with marginal $p$. The measure $\mu := \Pi \mub = \mub \circ \Pi^{-1}$ is called the \emph{self-affine measure} associated to $(\Phi,p)$, and is well-known to be a Radon measure supported on $K$ that satisfies 
$$
\mu = \sum_{i\in\Gamma} p_i \cdot \varphi_i \mu.
$$
The following strong projection theorem for self-affine measures is due to by Bárány, Hochman and Rapaport \cite{BaranyHochmanRapaport2019}. For $\theta\in\RP$, let $\pi_\theta :\R^2\to\theta$ denote the orthogonal projection onto the line $\theta$.
\begin{theorem}[Theorems 1.3 and 7.1 of \cite{BaranyHochmanRapaport2019}]\label{theorem-selfaffineprojections}
Let $\mu$ be a self-affine measure associated to a totally irreducible, hyperbolic system of affine contractions $\Phi$ with the exponential separation condition. Then for $\mu_F^*$-almost every $\theta\in\RP$, 
\begin{equation}\label{eq-projectionresult}
\dim \pi_\theta \mu = \min \lbrace 1, \dim \mu \rbrace.
\end{equation}
Moreover, if $\Phi$ satisfies the strong separation condition, then \eqref{eq-projectionresult} holds for every $\theta\in\RP$.
\end{theorem}
The analogous result for self-similar measures was proven earlier by Hochman and Shmerkin in \cite{HochmanShmerkin2012}. 

It follows from the Ledrappier-Young formula due to Bárány \cite{Barany2015} that planar self-affine measures satisfy dimension conservation in directions typical for the Furstenberg measure $\mu_F$. The Ledrappier-Young formula was shortly after generalized to higher dimensions by Bárány and Käenmäki \cite{BaranyKaenmaki2017}. The application to dimension conservation was recently generalized by Feng \cite{Feng2023} to higher dimensions and for more general self-affine measures. Let $\mu = \int \mu_x^\theta\,d\pi_\theta\mu(x)$ denote the disintegration of $\mu$ with respect to the orthogonal projection $\pi_\theta$.

\begin{theorem}[Corollary of Theorem 2.7 of \cite{Barany2015}]\label{thm-dimensionconservation}
Let $\mu$ be a self-affine measure associated to an affine system of contractions $\Phi$ with the strong separation and domination conditions, and $\mu_F$ the associated Furstenberg measure. Then for $\mu_F$-a.e. $\theta$ and $\pi_\theta\mu$-a.e. $x$, the measure $\mu_x^\theta$ is exact dimensional and
$$
\dim \mu = \dim \pi_{\theta^\perp}\mu + \dim \mu^\theta_x.
$$
\end{theorem}

\subsection{Flows and eigenvalues}

In the following, let $X$ be a complete separable metric space and let $\mu$ be a Borel probability measure on $X$. Let $\lbrace T_t: X\to X\rbrace_{t\in \R}$ be a family of measurable functions with the property that $T_0 = {\rm Id}$ and $T_s \circ T_t = T_{s+t}$ for each $s,t\in \R$. Recall that a real number $c \in \R$ is an \emph{eigenvalue} of the flow $(X, \lbrace T_t\rbrace_{t\in\R}, \mu)$ if there exists a measurable function $f: X \to \C$ such that $f\not\equiv 0$ and $f \circ T_t(x) = e(c t)f(x)$ for $\mu$-almost every $x$ and every $t\in \R$, where we write $e(x) := \exp (2\pi i x)$. Such a function $f$ is called an \emph{eigenfunction} for $c$. Recall that the flow $(X, \lbrace T_t\rbrace_{t\in\R}, \mu)$ is ergodic if and only if $0$ is a simple eigenvalue of $(X, \lbrace T_t\rbrace_{t\in\R}, \mu)$, that is, every invariant function is constant almost everywhere. We now record some standard properties of eigenvalues.

\begin{lemma}\label{preliminary-eigenvalueergodic}
Suppose that $T_t: X\to X$ is continuous for each $t\in\R$, and that $(X,\lbrace T_t\rbrace_{t\in\R},\mu)$ is ergodic. For any $c\in \R$, the discrete-time dynamical system $(X, T_c, \mu)$ is ergodic if and only if no non-zero integer multiple of $c^{-1}$ is an eigenvalue of $(X,\lbrace T_t\rbrace_{t\in\R},\mu)$. 
\end{lemma}
\begin{proof}
    This is proved in \cite[Lemma 3.11]{Hochman2012} in a slightly different language. For the convenience of the reader, we repeat the short proof here. 
    
    Let $\mu = \int \mu_x \,d\mu(x)$ be the ergodic decomposition of $\mu$ with respect to the map $T_{c}$. From $\lbrace T_t\rbrace_{t\in\R}$-invariance of $\mu$ it follows that $\mu = \int_0^{c} T_t \mu\,dt = \int\int_0^{c} T_t\mu_{x} \,dt\,d\mu(x)$. Since the function $x\mapsto \int_0^{c} T_t\mu_{x}\,dt$ is $\lbrace T_t\rbrace_{t\in\R}$-invariant, it is constant almost everywhere and so we have $\mu = \int_0^{c}T_t\mu_{x}\,dt$ for $\mu$-almost every $x$. Fix now such an $x$, let $\nu := \mu_x$ and note that the function $F: \R/c\Z \to \mathcal{P}(X)$ given by $F(t) = T_t \nu$ is well-defined since $T_{c}\nu = \nu$. Let $\Lambda = \lbrace r\geq 0:\ F(r) = F(0)\rbrace$ denote the set of periods of $F$. Since $F$ is Lebesgue measurable (even continuous, see e.g. \cite[Proposition 3.1]{EdekoGerlachKuhner2019}), either $\Lambda = \R/c\Z$ or $\Lambda = \lbrace kc/n: 0\leq k \leq n\rbrace$ for some $n\in\N$. In the former case, $\mu = \int_0^{c} F(t)\,dt = F(0) = \nu$ and so $\mu$ is $T_c$-ergodic. In the latter case, there exists a largest $n\in\N$ such that $T_{c/n}\nu = \nu$ and so $\mu = \int_0^{c/n} T_t\nu\,dt$. In particular, $n/c$ is an eigenvalue of $(X,\lbrace T_t\rbrace_{t\in\R},\mu)$ for the eigenfunction $f: X \to \C$ defined by $f(x) = e(t(x) n/c)$ for $\mu$-almost every $x$, where $t(x)\in \R/cn^{-1}\Z$ is the unique number such that $\lim_{k\to\infty} \frac{1}{k}\sum_{\ell=1}^k \delta_{T_{c}^\ell x} = T_{t(x)}\nu$.
\end{proof}

\begin{lemma}\label{preliminary-countablymanyeigenvalues}
The flow $(X, \lbrace T_t\rbrace_{t\in\R}, \mu)$ has at most countably many eigenvalues.
\end{lemma}

\begin{proof}
It is not difficult to see that eigenfunctions for different eigenvalues are orthogonal. Since the space $L^2(X)$ is separable when $X$ is complete and separable, any collection of orthogonal functions has to be countable. 
\end{proof}

The following characterisation of ergodicity of a product of ergodic systems plays a key role in the proof of Theorem \ref{theorem-main}.

\begin{proposition}\label{preliminary-ergodicproduct}
Let $(X, \lbrace T_t\rbrace_{t\in\R},\mu)$ and $(Y,\lbrace S_t\rbrace_{t\in\R},\nu)$ be ergodic flows over complete separable metric measure spaces $(X,\mu)$ and $(Y,\nu)$. Then every eigenvalue of $(X\times Y,\lbrace T_t \times S_t\rbrace_{t\in\R}, \mu\times \nu)$ is a sum of eigenvalues of $(X,\mu)$ and $(Y,\nu)$. In particular, the product flow $(X\times Y,\lbrace T_t \times S_t\rbrace_{t\in\R}, \mu\times \nu)$ is ergodic if and only if the flows have no common eigenvalues other than $0$.
\end{proposition}

A detailed proof for the product of discrete-time dynamical systems can be found for example in \cite[Theorem 6.6.6]{Hochmanlecturenotes}. While we believe that the result for the product of ergodic flows is also well-kwown, we were unable to find a reference in the literature and so provide a short proof relying on the spectral theorem for unitary operators. The form of the spectral theorem we require is recorded below.

\begin{theorem}[Spectral theorem for unitary operators]\label{theorem-spectraltheorem}
    Let $(U_t)_{t\in\R}$ be a strongly continuous one-parameter unitary group on a Hilbert space $H$. Then there exists a finite measure $\eta$ on $\R\times\N$ and a unitary map $R: L^2(\R\times\N,\eta)\to H$ such that 
    $$
U_t = R L_t R^{-1},
    $$
    for every $t\in\R$, where the operator $L_t: L^2(\R\times\N,\eta) \to L^2(\R\times\N,\eta)$ is defined by $L_t(a(x,n))(x, n) = e(tx) a(x,n)$ for every $a\in L^2(\R\times\N,\eta)$ and $\eta$-almost every $(x,n)$. 
\end{theorem}
\begin{proof}
    By Stone's theorem for one-parameter unitary groups \cite[Theorem 10.15]{Hall2013}, there exists a self-adjoint operator $A$ such that $U_t = e(t A) = \sum_{k=0}^\infty (2\pi i t A)^k/k!$ for every $t\in\R$. Applying the spectral theorem for self-adjoint operators \cite[Theorem 2.5.1]{Davies1995} we find a finite measure $\eta$ on $\R\times\N$ and a unitary map $R: L^2(\R\times\N,\eta)\to H$ such that $e(tA) = R e(th) R^{-1}$, where $h\in L^2(\R\times \N,\eta)$ is the map $h(x,n) = x$. Setting $L_t = e(th)$ for every $t\in\R$, we have $U_t = R L_t R^{-1}$ and $L_t(a)(x,n) = e(th(x,n)) a(x,n) = e(tx) a(x,n)$ for every $a\in L^2(\R\times\N,\eta)$ and $\eta$-almost every $(x,n)$, as was required. 
\end{proof}

\begin{proof}[Proof of Proposition \ref{preliminary-ergodicproduct}]
    Let $\alpha$ be an eigenvalue of $\lbrace T_t \times S_t\rbrace_{t\in\R}$ for an eigenfunction $f$. Let $\tilde{T}_t$ and $\tilde{S}_t$ denote the linear Koopman operators on the Hilbert spaces $L^2(X,\mu)$ and $L^2(Y,\nu)$ defined by $\tilde{T}_tg = g\circ T_t$ and $\tilde{S}_t g = g\circ S_t$. Since $T_t$ and $S_t$ are measure-preserving, these operators are easily seen to be unitary. Moreover, the groups $(\tilde{T}_t)_{t\in\R}$ and $(\tilde{S}_t)_{t\in\R}$ are strongly continuous, see for example \cite[Proposition 3.1]{EdekoGerlachKuhner2019}. Let $\lbrace \varphi_i\rbrace_{i\in\N}$ and $\lbrace \psi_j\rbrace_{j\in\N}$ be orthonormal bases for $L^2(X,\mu)$ and $L^2(Y,\nu)$. Then $\lbrace \phi_i \psi_j\rbrace_{i,j\in\N}$ forms an orthonormal basis for $L^2(X\times Y, \mu\times\nu)$, so we can expand $f= \sum_{i,j}a_{i,j}\varphi_i\psi_j$ for some $a_{i,j}\in\C$. Decomposing $\tilde{T}_t = R_T L_t^T R_T^{-1}$ and $\tilde{S}_t = R_S L_t^S R_S^{-1}$ using Theorem \ref{theorem-spectraltheorem}, we may further decompose 
    \begin{equation}\label{eq-tensorprod}
\tilde{T}_t \otimes \tilde{S}_t = (R_{T} \otimes R_{S}) (L_{t}^T\otimes L_{t}^S) (R_{T}^{-1}\otimes R_{S}^{-1}).
    \end{equation}
    Using \eqref{eq-tensorprod} and the identity $(\tilde{T}_t \otimes \tilde{S}_t)f = e(\alpha t) f$, we have
    \begin{align*}
    (R_{T}^{-1}\otimes R_{S}^{-1}) e(\alpha t)f &= L_{t}^T\otimes L_{t}^S(R_{T}^{-1}\otimes R_{S}^{-1}) f \\
    &= \sum_{i,j} a_{i,j} e(t(x+y)) R_T^{-1}\otimes R_S^{-1} \varphi_i\psi_j.
    \end{align*}
    By uniqueness of coordinate representation, for any $(i,j)$ such that $a_{i,j}\neq 0$ we have $R_T^{-1}\varphi_iR_S^{-1}\psi_j = e(t(x+y-\alpha)) R_T^{-1}\varphi_iR_S^{-1}\psi_j$, in particular, $R_T^{-1}\varphi_i(x,n) R_S^{-1}\psi_j(y,m)=0$ whenever $x+y\neq \alpha$. Fixing any such $(i,j)$, there exists $c\in\R$ such that $R_T^{-1}\varphi_i(x,n) = 0$ whenever $x\neq \alpha+c$ and $R_S^{-1}\psi_j(y,m)=0$ whenever $y\neq -c$. It easily follows that $\alpha+c$ is an eigenvalue of $\lbrace T_t\rbrace_{t\in\R}$ for the eigenfunction $\varphi_i$, and $-c$ is an eigenvalue of $\lbrace S_t\rbrace_{t\in\R}$ for the eigenfunction $\psi_j$:
    $$
    \varphi_i\circ T_t = R_T L_t^T R_T^{-1}\varphi_i \equiv R_T e((\alpha+c)t) R_T^{-1}\varphi_i = e((\alpha+c)t)\varphi_i
    $$
    and similarly, $\psi_j\circ S_t \equiv e(-ct)\psi_j$. In particular, $\alpha = (\alpha+c)-c$ is a sum of eigenvalues of $\lbrace T_t\rbrace_{t\in\R}$ and $\lbrace S_t\rbrace_{t\in\R}$. 
    
    To prove the second claim, apply the above with $f$ being an invariant function and $\alpha = 0$. Then we find that $c$ is an eigenvalue of both $\lbrace T_t\rbrace_{t\in\R}$ and $\lbrace S_t\rbrace_{t\in\R}$ for the eigenfunctions $\varphi_i$ and $\overline{\psi_j}$, respectively. Since the flows are assumed to have no common eigenvalues other than $0$, we conclude that $c=0$. But this implies that $\varphi_i$ and $\overline{\psi_j}$ are invariant functions for $\lbrace T_t\rbrace_{t\in\R}$ and $\lbrace S_t\rbrace_{t\in\R}$ and therefore constant almost everywhere. Since this is true for any pair $(i,j)$ such that $a_{i,j}\neq 0$ in the representation $f= \sum_{i,j}a_{i,j}\varphi_i\psi_j$, we conclude that also $f$ is constant almost everywhere.
\end{proof}

Let $(\Gamma^\N, \sigma)$ be a shift space. Given a function $f: \Gamma^\N \to (0,+\infty)$, we may build a semi-flow from $(\Gamma^\N, \sigma)$ by ``flowing up'' from a point $\io \in\Gamma^\N$ until we reach the time $f(\io)$, then switch to the point $\sigma \io$ and continue flowing until the time $f(\sigma \io)$, and so on. Formally, we let $Z = (\Gamma^\N\times [0,\infty))/\sim$ eqipped with the quotinent topology, where $\sim$ denotes the equivalence relation generated by $(\io, f(\io)) \sim (\sigma\io, 0)$ for every $\io\in\Gamma^\N$. Denoting each equivalence class $[(\io,t)]$ by the unique representative $(\io,t)$ with $0\leq t< f(\io)$, we may write
$$
Z = \lbrace (\io, t):\ \io \in \Gamma^\N, 0 \leq t \leq f(\io)\rbrace.
$$
Let $T_s: (\io, t) \mapsto (\io, t+s)$ for every $s\geq 0$. If there is a $\sigma$-invariant measure $\mu$ on $\Gamma^\N$ which measures $f$ and $\int f\,d\mu < \infty$, a natural $\lbrace T_s\rbrace_{s\geq 0}$-invariant probability measure on $Z$ is given by $\lambda := (\mu \times \mathcal{L})_Z$. The tripet $(\Gamma^\N, \lbrace T_s\rbrace_{s\geq 0}, \lambda)$ is called the \emph{suspension} of $(\Gamma^\N, \sigma, \mu)$ under the roof function $f$. It is easy to see that if $(\Gamma^\N, \sigma, \mu)$ is ergodic, then so is $(Z, \lbrace T_s\rbrace_{s\geq 0}, \lambda)$: For if $\varphi$ were a non-constant invariant function for $(Z, \lbrace T_s\rbrace_{s\geq 0},\lambda)$, then $\io \mapsto \varphi(\io,0)$ would be a non-constant invariant function for $(\Gamma^\N, \sigma, \mu)$. Suspension flows over the invertible system $(\Gamma^\Z, \sigma)$ are defined analogously by setting $Z= (\Gamma^\Z \times \R)/\sim$, where $\sim$ is the equivalence relation generated by $(\io,t) \sim (\sigma \io, t-f(\io))$ for every $(\io,t)\in\Gamma^\N\times\R$, and in this case we let $T_s: (\io,t)\mapsto (\io,t+s)$ for every $s\in\R$. 

A special property of regular enough suspension flows is that to any eigenvalue corresponds a \emph{continuous} eigenfunction.

\begin{proposition}[Proposition 6.2 of \cite{ParryPollicott1990}]\label{preliminary-weakmixing}
Let $(Z, \lbrace T_s\rbrace_{s\geq 0}, \lambda)$ be the suspension of a shift space $(\Gamma^\N, \sigma, \mu)$ under a locally Hölder continuous roof function, where $\mu$ is the equilibrium state for a locally Hölder continuous potential on $\Gamma^\N$. Then a number $\alpha \in \R$ is an eigenvalue of $(Z, \lbrace T_s\rbrace_{s\geq 0}, \lambda)$ for a continuous eigenfunction if and only if it is an eigenvalue for a measurable eigenfunction. 
\end{proposition}

It is well-known that Bernoulli measures on $\Gamma^\N$ are equilibrium states for locally constant potentials. 
\begin{lemma}\label{lemma-samatominaisarvot}
    Let $f: \Gamma^\Z \to (0,+\infty)$ be a continuous function such that $f(\io) = f(\jo)$ whenever $\io^+ = \jo^+$, and let $f^+: \Gamma^\N \to [0,+\infty)$ denote the function given by $f^+(\io) = f(\jo)$ for every $\io\in\Gamma^\N$ and $\jo\in\Gamma^\Z$ such that $\jo^+ =\io$.

    Let $\mu$ be a shift-invariant ergodic measure on $\Gamma^\Z$ which measures $f$, and write $\mu^+$ for the projection onto the positive coordinates. Let $(Z, \lbrace T_s\rbrace_{s\geq 0}, \lambda) =: \mathcal{Z}$ denote the suspension of $(\Gamma^\Z,\sigma,\mu)$ over $f$, and $(Z^+, \lbrace T_s\rbrace_{s\geq 0}, \lambda^+)=: \mathcal{Z}^+$ the suspension of $(\Gamma^\N, \sigma, \mu^+)$ over $f^+$. 

    Then for any eigenvalue $\alpha\neq 0$ of $\mathcal{Z}$, there exists $n\in\Z$ such that $n\alpha$ is an eigenvalue of $\mathcal{Z}^+$. 
\end{lemma}

\begin{proof}
    Let $\alpha\neq 0$ be such that $n\alpha$ is not an eigenvalue of $\mathcal{Z}^+$ for any $n\in\Z$. Then by Lemma \ref{preliminary-eigenvalueergodic}, the discrete-time system $(Z^+, T_{1/\alpha}, \lambda^+)$ is ergodic. (By possibly replacing $\alpha$ by $-\alpha$, we may assume $\alpha>0$.) Our aim is to show that also the system $(Z, T_{1/\alpha},\lambda)$ is ergodic. Let $U\subseteq Z$ be an open set of the form $[\io]_m^n \times I$ for some $\io\in\Gamma^\Z$, $m\leq n \in \Z$ and an interval $I\subseteq \R$. Since any continuous function on $Z$ can be approximated arbitrarily well (in $L^1$) by simple functions on this kind of sets, in order to show that $\lambda$ is ergodic under $T_{1/\alpha}$, it suffices to show that
    $$
    \lim_{n\to\infty} \frac{1}{n}\#\lbrace 0\leq k\leq n:\ T_{k/\alpha}(\io, t)\in U\rbrace = \lambda(U)
    $$
    for a.e. $(\io,t)$. 
    
    Let $\ell\in\N$ be large enough so that for every $(\io,t)\in T_{-\ell/\alpha}U$, also $(\jo,t)\in T_{-\ell/\alpha} U$ whenever $\jo^+ =\io^+$. Write $(T_{-\ell/\alpha}U)^+$ for the projection of $T_{-\ell/\alpha}U$ onto $Z^+$, and note that $T_{-\ell/\alpha}U$ equals the embedding of $(T_{-\ell/\alpha}U)^+$ to $Z$.
    
    Let $A^+$ be the set of full $\lambda^+$-measure such that for each $(\jo, t)\in A^+$, we have
    $$
    \lim_{n\to\infty} \frac{1}{n}\#\lbrace 0\leq k\leq n:\ T_{k/\alpha}(\jo, t)\in (T_{-\ell/\alpha}U)^+\rbrace = \lambda^+((T_{-\ell/\alpha}U)^+) = \lambda(T_{-\ell/\alpha}U) = \lambda(U).
    $$
    by Birkhoff's ergodic theorem. The second-to-last equality follows from the choice of $\ell$, and the last follows from $T_t$-invariance of $\lambda$. 
    Now, if $A$ is the embedding of $A^+$ to $Z$, then $T_{\ell/\alpha} A$ has full $\lambda$-measure, and for each $(\io,t)\in T_{\ell/\alpha} A$, 
    \begin{align*}
    &\lim_{n\to\infty} \frac{1}{n}\#\lbrace 0\leq k\leq n:\ T_{k/\alpha}(\io, t)\in U\rbrace\\
    =\ & \lim_{n\to\infty} \frac{1}{n}\#\lbrace 0\leq k\leq n:\ T_{(k-\ell)/\alpha}(\io, t)\in T_{-\ell/\alpha}U\rbrace \\
    =\ &\lambda(U).
    \end{align*}
    Therefore $\lambda$ is ergodic under $T_{1/\alpha}$ and by Lemma \ref{preliminary-eigenvalueergodic}, $\alpha$ is not an eigenvalue of $\mathcal{Z}$. 
\end{proof}

\subsection{Shannon entropy}

For a probability measure $\mu$ on $\R^d$ and any measurable partitions $\cE$ and $\cF$ of $\R^d$, we write $H(\mu, \cE) = -\sum_{E \in \cE} \mu(E)\log\mu(E)$ for the \emph{(Shannon) entropy} of $\mu$ with respect to $\cE$, and $H(\mu, \cE|\cF) = -\sum_{F\in\cF} \mu(F) H(\mu_F, \cE)$ for the \emph{conditional entropy} of $\mu$ with respect to $\cE$, given $\cF$. Write $\cE \vee \cF = \lbrace E\cap F:\ E\in \cE, F\in\cF\rbrace$ for the join of $\cE$ and $\cF$. In the following, we record some elementary properties of entropy. We refer to \cite[Section 2]{CoverThomas2006} and \cite[Sections 4.2 and 4.3]{Walters1982} for more detailed discussions on the topic.

\begin{lemma}[Chain rule, Theorem 4.3 of \cite{Walters1982}]\label{lemma-chainrule}
Let $\mu$ be a probability measure on $\R^2$, and let $\cE$, $\cF$ be partitions of $\R^2$. Then
$$
H(\mu, \cE\vee \cF) = H(\mu, \cE | \cF) + H(\mu, \cF).
$$
\end{lemma}

\begin{lemma}[Concavity and almost-convexity]\label{lemma-concavityofentropy}
For any partitions $\cE$ and $\cF$ of $\R^2$, the function $\mu\mapsto H(\mu,\cE|\cF)$ is concave. Moreover, the function $\mu\mapsto H(\mu,\cE)$ is almost convex in the sense that for any probability measures $\mu_1, \ldots, \mu_k$ and a probability vector $(p_1, \ldots, p_k)$, 
$$
H\left(\sum_{i=1}^k p_i \mu_i, \cE\right) \leq \sum_{i=1}^k p_i H(\mu_i, \cE) - \sum_{i=1}^k p_i \log p_i.
$$
\end{lemma}
\begin{proof}
    The concavity of the function $\mu\mapsto H(\mu,\cE|\cF)$ follows essentially from the concavity of $x\mapsto - x\log x$ on $[0,1]$; For details, we refer to \cite[Equations (2.29) and (2.43) and Theorems 2.7.2 and 2.7.3]{CoverThomas2006}. The second claim is immediate from the definition of entropy: 
    \begin{align*}
H\left(\sum_{i=1}^k p_i \mu_i, \cE\right) &= -\sum_{i=1}^k p_i \sum_{E\in\cE} \mu_i(E)\log\left(\sum_{i=1}^k p_i\mu_i(E)\right) \\
& \leq-\sum_{i=1}^k p_i\sum_{E\in\cE} \mu_i(E)(\log p_i + \log \mu_i(E)) \\
&= \sum_{i=1}^k p_i H(\mu_i, \cE) - \sum_{i=1}^k p_i \log p_i.
    \end{align*}
\end{proof}

\begin{lemma}\label{lemma-almostcontinuity}
Let $\mu$ be a probability meausure on $\R^d$, and let $\cE$ and $\cF$ be partitions of $\R^d$ such that each element of $\cE$ intersects at most $k$ elements of $\cF$ and vice versa. Then 
$$
|H(\mu, \cE) - H(\mu, \cF)| \leq \log k.
$$
\end{lemma}
\begin{proof}
    We first note that $H(\mu, \cE\vee\cF) \geq H(\mu,\cE)$ by \cite[Theorem 4.3]{Walters1982}. Lemma \ref{lemma-chainrule} now asserts that $H(\mu, \cE) - H(\mu, \cF) \leq H(\mu, \cE\vee F) - H(\mu, \cF) =  H(\mu, \cE |\cF)$ and since for each $F\in \cF$, $\mu_F$ is supported on at most $k$ atoms of $\cE$, it follows from \cite[Corollary 4.21]{Walters1982} that $H(\mu, \cE|\cF) \leq \log k$. Similarly, $H(\mu,\cF) - H(\mu,\cE)\leq \log k$. 
\end{proof}

A consequence of concavity is that taking a convolution can decrease entropy at most by an additive constant.

\begin{lemma}\label{lemma-entropyofconvolution}
Let $\cE$ and $\cF$ be partitions of $\R^2$ such that any translate of $\cE\vee \cF$ intersects at most $k$ elements of $\cE\vee\cF$ and any translate of $\cF$ intersects at most $k$ elements of $\cF$. Then for any probability measures $\mu$ and $\nu$,
$$
H(\mu*\nu, \cE|\cF) \geq H(\mu, \cE|\cF) -2\log k.
$$
\end{lemma}
\begin{proof}
It follows from Lemma \ref{lemma-concavityofentropy} and Jensen's inequality that $H(\mu*\nu,\cE|\cF) \geq \int H(\mu*\delta_x, \cE|\cF)\,d\nu(x)$. By Lemmas \ref{lemma-chainrule} and \ref{lemma-almostcontinuity}, $H(\mu*\delta_x, \cE|\cF) \geq H(\mu, \cE|\cF) - 2\log k$ for every $x\in\R^2$.
\end{proof}

Let $\mathcal{D}_n = \mathcal{D}_n(\R^d)$ denote the partition of $\R^d$ into dyadic cubes of side-length $2^{-n}$ which we call the ``level-$n$'' dyadic partition. For $x\in\R^d$, write $D_n(x)$ for the element of $\mathcal{D}_n$ that contains $x$. For entropy with respect to the dyadic partition we use the short-hand notation $H_n(\mu) = H(\mu, \mathcal{D}_n) = -\sum_{D\in\mathcal{D}_n} \mu(D)\log \mu(D)$. For $\theta\in\RP$, let $\mathcal{D}_n(\theta)$ denote the level-$n$ dyadic partition of the line $\theta$. The following is a simple application of the chain rule; Let $R_\theta$ denote the ``shortest'' rotation which takes $\theta\in\RP$ onto the $y$-axis, with $R_{x\text{-axis}}$ given by the clockwise rotation.

\begin{lemma}\label{lemma-chainruleapplication}
Let $\mu$ be a probability measure on $\R^2$, and let $\theta\in\RP$. Then, denoting $H_n(\mu|\pi_\theta) := H(\mu, \mathcal{D}_n(\R^2) | \pi_\theta^{-1} \mathcal{D}_n(\theta))$, we have
$$
|H_n(\mu) - (H_n(\pi_\theta\mu) + H_n(\mu|\pi_\theta))| \leq \log 9
$$
for every $n\in\N$. 
\end{lemma}

\begin{proof}
    It follows from Lemma \ref{lemma-chainrule} that $H(\mu, \mathcal{D}_n(\R^2)\vee \pi_\theta^{-1} \mathcal{D}_n(\theta))= H(\pi_\theta\mu, \mathcal{D}_n(\theta)) + H_n(\mu|\pi_\theta)$. It is not difficult to see that each element of $\mathcal{D}_n(\R^2)$ intersects at most two elements of $\mathcal{D}_n(\R^2)\vee \pi_\theta^{-1} \mathcal{D}_n(\theta)$ and vice versa. On the other hand, each element of $\mathcal{D}_n(\theta)$ intersects at most three elements of $\mathcal{D}_n(\R^2)$ and vice versa, whence it follows from Lemma \ref{lemma-almostcontinuity} that $|H_n(\mu) - (H_n(\pi_\theta\mu) + H_n(\mu|\pi_\theta))| \leq \log 9$.
\end{proof}

\begin{lemma}\label{lemma-continuityofentropy}
Let $\mu$ and $\nu$ be probability measures on $[0,1]$, and suppose that $\mu$ is non-atomic. Then for every $r > 0$ and $\varepsilon>0$, there exists $N_0\in\N$ such that for any interval $I$ with $\mu(I)\geq r$, we have 
$$
\frac{1}{N} H_N(\mu_I*\nu) \geq \dim(\mu * \nu)- \varepsilon
$$
for every $N\geq N_0$.
\end{lemma}

\begin{proof}
Let $K_r = \lbrace (a,b)\in \R^2:\ \mu([a,b])\geq r \rbrace$. We will show that for every $r>0$ such that $K_r$ is nonempty, $\frac{1}{N}H_N(\mu_{[a,b]} * \nu)$ is continuous in $(a,b)\in K_r$ (in the subspace topology) and the continuity is uniform in $N$.

Let $r, \varepsilon > 0$ be given, and let $\delta > 0$ be small with respect to $\varepsilon$ and $r$. Fix $(a_0,b_0) \in K_r$ and let $I_\delta$ be the largest interval contained in $[a,b]$ for every $(a,b)\in B((a_0,b_0), \delta)$. If $\delta$ is small enough, we have $\frac{\mu(I_\delta)}{\mu([a,b])} \in [1-\varepsilon,1+\varepsilon]$ for every $(a,b)\in B((a_0,b_0),\delta)$, by non-atomicity of $\mu$ and the assumption $\mu([a_0,b_0])\geq r$. 

Now, for every $(a,b), (a',b') \in B((a_0,b_0), \delta)$, applying bilinearity of convolution and Lemma \ref{lemma-concavityofentropy} in the first and second-to-last inequalities, we have
\begin{align*}
&H_N(\mu_{[a,b]}*\nu) \\
\leq\ &\frac{\mu(I_\delta)}{\mu([a,b])} H_N(\mu_{I_\delta}*\nu) + \left(1-\frac{\mu(I_\delta)}{\mu([a,b])}\right) H_N(\mu_{[a,b]\setminus I_\delta}*\nu) + \varepsilon \\
\leq\ &(1+\varepsilon) \frac{\mu(I_\delta)}{\mu([a', b'])} H_N (\mu_{I_\delta}*\nu) + 2N\varepsilon \\
\leq\ &(1+\varepsilon) H_N(\mu_{[a',b']}*\nu) + 2N\varepsilon \\
\leq\ &H_N(\mu_{[a', b']}*\nu) + 3N\varepsilon.
\end{align*}
Thus $(a,b) \mapsto \frac{1}{N} H_N(\mu_{[a,b]} *\nu)$ is continuous in $K_r$, uniformly in $N$. On the other hand, for every $(a,b)\in K_r$, 
$$
\dim (\mu*\nu) \leq \dim (\mu_{[a,b]}*\nu) \leq \liminf_{N\to\infty} \frac{1}{N} H_N(\mu_{[a,b]}*\nu)
$$ 
by \cite[Theorem 1.3]{FanLauRao2002}. By uniform continuity, this convergence is uniform in $K_r$, which is what we wanted to prove.
\end{proof}

\subsection{Magnifying measures}

For $x\in\R^d$ and $r\geq 0$, we let $T_x:\ y\mapsto y-x$ denote the translation taking $x$ to the origin, and $S_r: x \mapsto 2^r x$ the exponential ``magnification'' operation. We let $S_r^*$ denote the action of $S_r$ on measures equipped with normalisation and restriction, that is,
$$
S_r^*\mu(A) = \mu(B(0,2^{-r}))^{-1} \mu(2^{-r}A \cap B(0,2^{-r}))
$$
for every Borel set $A\subseteq B(0,1)$ and measure $\mu$ whose support contains the origin.

There is a natural way in which measures on $\R^d$ give rise to measures on $\cP(\R^d)$. Namely, consider the sequence $(S^*_r T_x\mu)_{r\geq 0}$, called the \emph{scenery} of $\mu$ at $x$. The statistical properties of this sequence are described by the accumulation points of the sequence $\left( \frac{1}{t}\int_0^t \delta_{S_r^* T_x\mu}\,dr \right)_{t\geq 1}$, called the \emph{scenery flow} of $\mu$ at $x$. The accumulation points of the scenery flow in the weak-$^*$ topology are measures on $\cP(\R^d)$, and are called \emph{tangent distributions} of $\mu$ at $x$. The measure $\mu$ is called \emph{uniformly scaling} if the scenery flow converges almost everywhere to a unique tangent distribution $P$. It is then said that $\mu$ generates $P$.

A remarkable result of Hochman \cite{Hochmanpreprint} is that tangent distributions at almost every point are \emph{fractal distributions}, objects which enjoy strong spatial invariance properties. We give the definition here for completeness, although we will not use it directly.
\begin{definition}
An $S_r^*$-invariant measure $P$ on $\cP(\R^d)$ is called a \emph{fractal distribution} if for any measurable $A$, $P(A)=1$ if and only if for every $r>0$, $P$-almost every $\nu$ satisfies 
$$
S_r^* T_x \nu \in A
$$
for $\nu$-almost every $x$ with $B(x,e^{-r})\subseteq B(0,1)$.
\end{definition}
\begin{theorem}[Theorem 1.7 of \cite{Hochmanpreprint}]\label{thm-fractaldistribution}
    Let $\mu$ be a Radon measure on $\R^d$. Then for $\mu$-almost every $x$, every tangent distribution at $x$ is a fractal distribution. 
\end{theorem}

The following is the only property of fractal distributions that we require directly.

\begin{lemma}\label{lemma-nolines0}
    Let $P$ be a fractal distribution. Then for $P$-a.e. $\nu$, any line $L$ with $\nu(L)>0$ must contain the origin. 
\end{lemma}

\begin{proof}
    We first show that $P$-a.e. atomic measure is the point mass at the origin. This is almost immediate from the results of \cite{Hochmanpreprint}. 

    For a contradiction, let $A = \lbrace \nu:\ \nu(\lbrace x\rbrace) > 0,\ x \neq 0\rbrace$ and suppose that $P(A) >0$. Let $P'$ be an ergodic component of $P$ with $P'(A)>0$, and let $\eta\in A$ be a uniformly scaling measure generating $P'$. Indeed, by \cite[Theorem 1.6]{Hochmanpreprint}, $P'$-almost every measure is uniformly scaling. Let $x$ be such that $\eta(\lbrace x \rbrace)>0$. Since $\eta|_{\lbrace x\rbrace}\ll \eta$, also $\eta|_{\lbrace x\rbrace}$ generates $P'$ by an application of the Lebesgue-Besicovitch differentiation theorem, whence $P'$ is supported on the point mass at the origin. In particular, $P'(A) = 0$, a contradiction.  
    
    Now, to prove the statement of the lemma, suppose that there exists a set $B$ with $P(B)>0$ such that for every $\nu\in B$, there exists a line $L_\nu$ with $\nu(L_\nu)>0$ and $0\not\in L_\nu$. Let $P'$ be an ergodic component of $P$ with $P'(B)>0$, and let $\eta\in B$ be a uniformly scaling measure generating $P'$. Now, since $\eta(L_\eta) >0$, also the measure $\eta|_{L_\eta}$ generates $P'$. Let $L$ denote the line $L_\eta$ translated so that it contains the origin. Clearly, all tangent measures of $\eta|_{L_\eta}$ are supported on $L$, whence $P'$ is supported on measures which are supported on $L$. Since any line not containing the origin intersects $L$ in at most one point, such a line has $P'$-almost surely zero measure by the above. Thus $P'(B)  =0$ which is a contradiction. 
\end{proof}

\subsection{Conditional measures on lines}

For a measure $\mu$ on $\R^2$ and $\theta\in\RP$, let $\mu = \int \mu_{x}^\theta \,d\pi_\theta\mu(x)$ denote the disintegration of $\mu$ with respect to $\pi_\theta$. It is well-known that for almost every $x\in\theta$, the measure $\mu_x^\theta$ is supported on the line $x+\theta^\perp$, and that $\mu_x^\theta$ is the limit of normalized restrictions of $\mu$ on thinner and thinner tubes centered at $x+\theta^\perp$. It will be useful for us to know that these tubes can replaced by preimages of sets of relatively large measure. 

\begin{lemma}\label{prop-generalslices}
Let $\mu$ be a measure on $\R^2$, let $\delta>0$ and $\theta\in\RP$. For every $r>0$, let 
$$
\mathcal{I}(x,r,\delta) = \lbrace I \subseteq B(x, r):\ \frac{\pi_\theta\mu(I)}{\pi_\theta\mu(B(x,r))}\geq \delta \rbrace.
$$
Then for $\pi_\theta \mu$-almost every $x$, we have
$$
\lim_{r \to 0} \sup_{I\in \mathcal{I}(x,r,\delta)} d_{\rm LP}\left( \mu_{\pi_\theta^{-1}(I)},\ \mu_x^\theta\right) = 0.
$$ 
\end{lemma}

\begin{proof}
For every $n\in\N$, let $E_n\subseteq \theta$ be the compact set given by Lusin's theorem with $\pi_\theta\mu(E_n) \geq 1-1/n$, on which the function $y \mapsto \mu_{y}^\theta$ is continuous. Since $\pi_\theta \mu(\bigcup_{n\in\N} E_{n}) = 1$, it suffices to prove the statement for almost every $x\in E_n$, for every $n\in\N$.  
Now, for almost every $x\in E_n$, if $B := B(x,r)$ and $I\in \mathcal{I}(x,r,\delta)$,
\begin{align*}
\frac{\pi_\theta\mu(I\cap E_n)}{\pi_\theta\mu(I)} &= 1 - \frac{\pi_\theta\mu(I\setminus E_n)}{\pi_\theta\mu(I)} \\
&\geq 1 -\frac{\pi_\theta\mu(B\setminus E_n)}{\pi_\theta\mu(B)} \frac{\pi_\theta\mu(B)}{\pi_\theta\mu(I)}\\
&\geq 1 - \frac{1}{\delta} \frac{\pi_\theta\mu(B\setminus E_n)}{\pi_\theta\mu(B)} \\
&= 1- o(1)/\delta
\end{align*}
by an application of the Lebesgue-Besicovitch differentiation theorem. Here $o(1)$ denotes a quantity that vanishes as $r\to 0$.

Therefore, for any Borel set $A\subseteq \R^2$, $I\in \mathcal{I}(x,r,\delta)$ and $\varepsilon>0$,
\begin{align*}
   \mu_{\pi_\theta^{-1} I}(A^\varepsilon) &= \frac{1}{\pi_\theta\mu( I)} \mu(\pi_\theta^{-1}I\cap A^\varepsilon)\\ &= \frac{1}{\pi_\theta\mu( I)} \int_{I} \mu_y^\theta(A^\varepsilon)\,d\pi_\theta\mu(y) \\
    &\geq \frac{1-o(1)/\delta}{\pi_\theta\mu(I\cap E_n)} \int_{ I\cap E_n} \mu_y^\theta(A^\varepsilon)\,d\pi_\theta\mu(y)  \\
    &\geq (1-o(1)/\delta) (\mu_{x}^\theta(A)-\varepsilon)
\end{align*}
if $r$ is small enough, by continuity of $y\mapsto \mu_y^\theta$. Similarly, 
$$
\mu_{\pi_\theta^{-1} I}(A) \leq \frac{1}{\pi_\theta\mu(I\cap E_n)} \int_{I\cap E_\varepsilon} \mu_y^\theta(A^\varepsilon)\,d\pi_\theta\mu(y) + o(1)/\delta \leq \mu_x^\theta(A^\varepsilon) + o(1)/\delta + \varepsilon,
$$
so $d_{\rm LP}(\mu_{\pi^{-1} I},\ \mu_x^\theta) \leq 2\varepsilon$ for small enough $r$. Taking $\varepsilon\to0$ completes the proof.
\end{proof}

\section{On local entropy averages}\label{section3}

In this section, we recall the local entropy averages of \cite{HochmanShmerkin2012} and introduce the different notions of magnifactions of measures that we use. Let $\mu$ and $\nu$ be self-affine measures associated to iterated function systems $\Phi =\lbrace \varphi_i(x) = A_i x + a_i\rbrace_{i\in\Gamma}$ and $\Psi = \lbrace \psi_j(x) = B_j x + b_j\rbrace_{j\in\Lambda}$, and denote by $\bmu$, $\bnu$ the associated Bernoulli measures. In this section, we impose no conditions on $\Phi$ and $\Psi$ other than that $\Vert A_i\Vert < 1$ and $\Vert B_j \Vert < 1$ for every $i,j$.

Let $\beta>0$. For every $\io\in\Gamma^\N$ and $\jo \in \Lambda^\N$, we define the ``stopping times'' 
\begin{align*}
i_k&=i_k(\io,\beta) = \min \lbrace n\in \N:\ \Vert A_{\io|_n} \Vert \leq 2^{-k\beta} \rbrace, \\
i_k&=i_k(\jo,\beta) = \min \lbrace n\in \N:\ \Vert B_{\jo|_n} \Vert \leq 2^{-k\beta} \rbrace.
\end{align*}
Although the stopping times on $\Gamma^\N$ and $\Lambda^\N$ are denoted by the same letter $i_k$, the choice of domain will always be clear from the context: for example, $\io|_{i_k} := \io|_{i_k(\io,\beta)}$ and $\jo|_{i_k} := \jo|_{i_k(\jo,\beta)}$. Note also that we omit the dependence on $\beta$ from the notation: In practice, $\beta$ will always be a large real number fixed beforehand. Below, we collect the different notions of scale-$k$ magnifications of $\mu$ and $\nu$. Note that they also depend on $\beta$.

\begin{notation}\label{notation-magnifications}
Let $\beta>0$. For each $\io \in \Gamma^\N$, $\jo\in \Lambda^\N$ and $k\in \N$, let
\begin{alignat*}{2}
&\nu_{\jo|_{i_k}} &&:= S_{k\beta}^* T_{\Pi(\jo)} \psi_{\jo|_{i_k}} \nu, \\
&\mu_{\io|_{i_k}} &&:= S_{k\beta}^* T_{\Pi(\io)} \varphi_{\io|_{i_k}} \nu, \\
&\mu^{D_{k\beta}(\Pi(\io))} &&:= S_{k\beta - \lfloor k\beta\rfloor} F\mu_{D_{\lfloor k\beta\rfloor}(\Pi(\io))},
\end{alignat*}
where $F$ denotes the unique homothety taking $D_{\lfloor k\beta\rfloor}(\Pi(\io))$, the level-$\lfloor k\beta\rfloor$ dyadic square that contains $\Pi(\io)$, onto $[-1,1)^2$.
\end{notation}

Note that $\mu_{\io|_{i_k}}$ and $\nu_{\jo|_{i_k}}$ are measures supported on ellipses that contain the origin and whose major semi-axes have length comparable to $1$ and are oriented in the directions $\theta(A_{\io|_{i_k}})$ and $\theta(B_{\jo|_{i_k}})$, respectively. On the other hand, $\mu^{D_{k\beta}(\Pi(\io))}$ is a measure supported on $[-2^{k\beta - \lfloor k\beta\rfloor},2^{k\beta - \lfloor k\beta\rfloor}]^2$. We require the following form of the local entropy averages of \cite{HochmanShmerkin2012}. For (non-integer) $\beta>0$, we write $H_\beta := H_{\lfloor \beta\rfloor}$.

\begin{theorem}\label{thm-localentropyaverages}
Let $\mu$ and $\nu$ be self-affine measures associated to $\Phi=\lbrace \varphi_i(x) = A_i x + a_i\rbrace_{i\in\Gamma}$ and $\Psi = \lbrace \psi_j(x) = B_j x + b_j\rbrace_{i\in\Lambda}$. For any $\varepsilon>0$, there exists $N_0\in\N$ such that if either
\begin{enumerate}
    \item[i)] $\liminf_{n\to\infty} \frac{1}{n} \sum_{k=0}^{n-1} \frac{1}{\beta} H_{\beta}(\mu^{D_{k\beta}(\Pi(\io))} * \nu_{\jo|_{i_k(\jo,\beta)}}) \geq \alpha$ or 
    \item[ii)] $\liminf_{n\to\infty} \frac{1}{n} \sum_{k=0}^{n-1} \frac{1}{\beta} H_{\beta}(\mu_{\io|_{i_k(\io,\beta)}} * \nu_{\jo|_{i_k(\jo,\beta)}}) \geq \alpha$
\end{enumerate}
holds for $\bmu\times \bnu$-a.e. $(\io,\jo)$ and some $\beta = \beta(\io,\jo, \varepsilon) \geq N_0$, then $\dim(\mu*\nu) \geq \alpha-\varepsilon$.
\end{theorem}
In applications, we will only apply i) with $\beta\in\N$, while in ii) we require the statement also for non-integer $\beta$. 

The proof is essentially in \cite{HochmanShmerkin2012}, but for the convenience of the reader we provide a sketch for the proof of the statement under the assumption that i) holds for almost every $(\io,\jo)$. The proof under the assumption ii) goes through similarly. We begin by recalling some terminology of \cite{HochmanShmerkin2012}. If $X$ is a finite set, $\rho>0$ and $d$ is the metric on $X^\N$ given by $d(\mathtt x, \mathtt y) = 2^{-\rho|\mathtt{x} \wedge \mathtt{x}'|}$, then $(X,d)$ (or just $X$, implicitly equipped with $d$) is called a $\rho$\emph{-tree}. Following \cite{HochmanShmerkin2012}, for $\rho$-trees $X^\N$ and $Y^\N$ we say that
\begin{itemize}
    \item[i)] a map $g: X^\N \to Y^\N$ is a \emph{tree morphism} if, for every $n\in \N$ and length-$n$ cylinder $[x_1\ldots x_n]\subseteq X^\N$ there exists a length-$n$ cylinder $[y_1 \ldots y_n]\subseteq Y^\N$ such that $g([x_1\ldots x_n])\subseteq [y_1\ldots y_n]$,
    \item[ii)] a map $h: X^\N \to \R^d$ is \emph{($c$-)faithful} if there exists a constant $c\geq1$ such that for any $[x_1 \ldots x_n] \subseteq X^\N$, no point in $f([x_1 \ldots x_n])$ is covered by more than $c$ of the sets $f([x_1 \ldots x_n x])$ for $x\in X$, and $f([x_1\ldots x_n])$ contains a ball of radius $(c^{-1} \rho)^n$ and is contained in a ball of radius $(c \rho)^n$.   
\end{itemize}
For a measure $\eta$ supported on a $\rho$-tree $X^\N$, we write
\begin{equation*}
    H_{k\rho}(\eta) := -\sum_{\mathtt x \in X^k} \eta([\mathtt x])\log \eta([\mathtt x]) = -\int \log\eta(B(\mathtt x, 2^{\rho k}))\,d\eta(\mathtt x),\qquad k\in\N.
\end{equation*}
\begin{proof}[Proof of Theorem \ref{thm-localentropyaverages}]
Let $\varepsilon>0$, and let $N_0\in\N$ be large with respect to $\varepsilon$. Assume first that there exists $\beta = \beta(\varepsilon) \geq N_0$ such that the condition i) holds for $\bmu\times\bnu$-almost every $(\io,\jo)$. The case where $\beta$ is allowed to depend on $(\io,\jo)$ is discussed near the end of the proof. As in \cite{HochmanShmerkin2012}, the idea is to lift $\mu \times \nu$ to a measure $\eta$ on a tree 
$$
\Sigma := \mathcal A^\N \times \mathcal B^\N
$$
with $\mathcal A$ and $\mathcal B$ being finite sets, and then find a finite collection of symbols $Y_\beta$, a tree morphism $g_\beta: \Sigma \to Y_\beta^\N$ and a faithful map $h_\beta: Y_\beta^\N\to \R^2$ such that $(h_\beta\circ g_\beta)\eta = \mu*\nu$. 

We begin by defining $\mathcal A$ and $\mathcal B$. The idea is to define $\mathcal A$ in such a way that $\mathcal A^k$ codes the dyadic covering of $[-1,1)^2$ via squares of diameter comparable to $2^{-\lfloor k\beta\rfloor}$, and $\mathcal B$ in such a way that $\mathcal B^k$ is the symbolic coding of the support of $\nu$ via cylinders of diameter comparable to $2^{-\lfloor k\beta\rfloor}$. We choose the following approach to formalise this. Recall the definition of $i_k = i_k(\jo,\beta)$. 

Let $\mathcal A \subset\lbrace 1,2,3,4\rbrace^* := \bigcup_{k\in\N}\lbrace 1,2,3,4\rbrace^k$ and $\mathcal B\subset \Lambda^*$ be the minimal finite sets such that 
\begin{enumerate}
\item for any $\ko\in \lbrace 1,2,3,4\rbrace^\N$ and $k\in\N$, $\ko|_{\lfloor k\beta\rfloor}\in\mathcal A^k$, 
\item for any $\jo\in\Lambda^\N$ and $k\in\N$, $\jo|_{i_k} \in \mathcal B^k$.
\end{enumerate}
Such $\mathcal A$ and $\mathcal B$ exist since the functions $\ko \mapsto |\ko|_{\lfloor (k+1)\beta\rfloor}| - |\ko|_{\lfloor k\beta\rfloor}|$ and $\jo\mapsto |\jo|_{i_{k+1}}| - |\jo|_{i_k}|$ are bounded uniformly in $k$, from below and above, by constant multiplicatives of $\log(\beta)$. On $\Sigma$, define the metric $d$ by 
\begin{equation*}
    d((\mathtt a, \mathtt b), (\mathtt a', \mathtt b')) = 2^{-\beta \min \lbrace |\mathtt a \wedge \mathtt a'|, |\mathtt b\wedge \mathtt b'|\rbrace}
\end{equation*}
so that $\Sigma$ is a $\beta$-tree (to match the definition of a $\rho$-tree we may identify $\Sigma$ with $(\mathcal A\times \mathcal B)^\N$). Denote by $\mathcal A_{\rm adm}^\N\subset\mathcal A^\N$ the set of \emph{admissible words} defined by 
\begin{equation*}
    \mathcal A^\N_{\rm adm}= \lbrace \mathtt a\in \mathcal A^\N:\ \text{there exists}\ \ko\in\lbrace1,2,3,4\rbrace^\N\ \text{s.t.}\ \ko|_{\lfloor k\beta\rfloor} = \mathtt a|_k\ \text{for every}\ k\in\N\rbrace.
\end{equation*}
The set $\mathcal A^\N_{\rm adm}$ is just the dyadic coding of $[-1,1)^2$ along the scales $2^{-\lfloor k\beta\rfloor}$, $k\in\N$.

Let $\lbrace D_{x_1},\ldots, D_{x_m}\rbrace$ be the (possibly overlapping) collection of dyadic squares indexed by $\mathcal A$, and let $F_{x_k}$ denote the homothety which sends $[-1,1]$ onto the closure of $D_{x_k}$. Note that $\diam D_{x_k}$ is comparable to $2^{-\lfloor\beta\rfloor}$ for every $k=1,\ldots, m$. Define the following projections:
\begin{alignat*}{3}
    &\Pi_\mathcal A&&: \mathcal A^\N \to [-1,1]^2,\ \ &&\mathtt a \mapsto \lim_{n\to\infty} F_{\mathtt a|_n}(0)\\
    &\Pi_\mathcal B&&: \mathcal B^\N \to \R^2, &&\mathtt b \mapsto \lim_{n \to \infty} \psi_{\mathtt b|_n}(0) \\
    &\Pi_\Gamma&&: \Gamma^\N \to \R^2, &&\io \mapsto \lim_{k \to \infty} \varphi_{\io|_k}(0) \\
    &\tilde{\Pi}&&: \Sigma \to \R^2\times\R^2, &&(\mathtt a, \mathtt b) \mapsto (\Pi_\mathcal A(\mathtt a), \Pi_\mathcal B(\mathtt b)).
\end{alignat*}
By possibly translating $\mu$ so that boundaries of dyadic squares have zero measure, $\mu$-almost every $x\in\R^2$ has a unique pre-image under $\Pi_\mathcal A|_{\mathcal A_{\rm adm}^\N}$ and so we may define the measure $\tilde{\mu} := (\Pi_\mathcal A|_{\mathcal A_{\rm adm}^\N})^{-1} \mu$ on $\mathcal A_{\rm adm}^\N \subset \mathcal A^{\N}$. The measure $\bnu$ on $\Lambda^\N$ will be identified with the measure on $\mathcal B^\N$ which is the image of $\bar{\nu}$ through the map 
\begin{equation*}
\jo \mapsto \mathtt b,\qquad \jo|_{i_k} = \mathtt b|_k\ \text{for every}\ k\in\N,
\end{equation*}
also denoted by $\bnu$. Let $\eta := \tilde{\mu} \times \bnu$ and note that $\tilde{\Pi}\eta = \mu\times \nu$. If we let $*$ denote the map $\R^2\times \R^2 \to \R^2, (x,y) \mapsto x+y$, then $(*\circ \tilde{\Pi})\eta = \mu*\nu$. 

We will now construct a tree $Y_\beta^\N$, a tree morphism $g_\beta:\Sigma \to Y_\beta^\N$ and a faithful map $h_\beta: Y_\beta^\N \to \R^2$ such that the following diagram commutes:
\[
  \begin{tikzcd}
    \Sigma \arrow{r}{g_\beta} \arrow[swap]{dr}{*\circ \tilde{\Pi}} & Y_\beta^\N \arrow{d}{h_\beta} \\
     & \R^2
  \end{tikzcd}
\]
Let $Y_\beta = \lbrace (\frac{k}{2^{\beta}}, \frac{\ell}{2^{\beta}}):\ 0 \leq k, \ell \leq \lceil 2^{\beta}\rceil-1 \rbrace$, and associate to each $x \in Y_\beta$ the closed square $Q_x$ of side length $10 \cdot 2^{-\beta}<1$ and bottom left corner at $x$. Then $[-1,1]^2 \subseteq \bigcup_{x\in Y_\beta} Q_x$. Let the metric on $Y_\beta^\N$ be given by $d(\mathtt{x},\mathtt{y})=2^{-\beta|\mathtt{x}\wedge\mathtt{y}|}$ for $\mathtt{x},\mathtt{y}\in Y_\beta^\N$. Writing $L_x$ for the homothety which sends $[-1,1]^2$ onto $Q_x$, the map $h_\beta: Y_\beta^\N \to [-1,1]^2$ defined by $h_N(x_1 x_2 x_3 \ldots) = \lim_{n \to \infty} L_{x_1} \circ \cdots \circ L_{x_n}(0)$ is a faithful surjection.

We construct $g_\beta: \Sigma \to Y_\beta^\N$ iteratively. First observe that for any $(a,b)\in\mathcal A\times \mathcal B$, we have $\Pi_\mathcal A[a] + \Pi_\mathcal B[b] \subseteq Q_{x_1}$ for some $x_1\in Y_\beta$. This is because by the definitions of $\mathcal A$ and $\mathcal B$, $\diam(\Pi_\mathcal A([a])) \leq 2\cdot2^{-\lfloor \beta\rfloor}\leq 4\cdot2^{-\beta}$ and $\diam (\Pi_{\mathcal B}([b])) \leq \diam \psi_b(B(0,1)) \leq 2^{-\beta}$, and therefore 
\begin{equation*}
\diam (\Pi_\mathcal A[a] + \Pi_\mathcal B[b]) \leq 5\cdot2^{-\beta}.
\end{equation*}
Supposing that the sequence $x_1 \ldots x_n \in Y_N^n$ has been determined in such a way that $\Pi_\mathcal A[\mathtt a|_n] + \Pi_\mathcal B[\mathtt b|_n] \subseteq Q_{x_1\ldots x_{n}}$, we choose $x_{n+1}\in Y_N$ so that $\Pi_\mathcal A[\mathtt a|_{n+1}] + \Pi_\mathcal B[\mathtt b|_{n+1}] \subseteq Q_{x_1\ldots x_n x_{n+1}} =: L_{x_1\ldots x_n}(Q_{x_{n+1}})$. This is again possible since 
\begin{equation*}
    \diam(\Pi_\mathcal A[\mathtt a|_{n+1}] + \Pi_\mathcal B[\mathtt b|_{n+1}]) \leq 5\cdot 2^{-(n+1)\beta}
\end{equation*}
and $Q_{x_1\ldots x_n} \subset \bigcup_{x\in Y_\beta} Q_{x_1\ldots x_n x}$. Since $\diam(Q_{x_1 \ldots x_n}) \to 0$, we may let $g_\beta: \Sigma \to Y_\beta^\N$ be defined by $g_\beta(\mathtt a,\mathtt b) = x_1x_2\ldots$. It is evident from the construction that $g_\beta([\mathtt a|_n]\times [\mathtt b|_n]) \subseteq [x_1x_2\ldots x_n]$, that is, sets of the form $[\mathtt a|_n]\times [\mathtt b|_n]$ are mapped into cylinders of $Y_\beta^\N$. Moreover, $h_\beta \circ g_\beta = * \circ \tilde{\Pi}$. 

Now, since the $h_\beta$ is faithful, \cite[Proposition 5.3]{HochmanShmerkin2012} asserts that it essentially preserves entropy in the sense that
\begin{equation}\label{eq-firstplace}
|H_{(k+1)\beta}(g_\beta\eta_{[\mathtt a|_k] \times [\mathtt b|_k]}) - H_{(k+1)\beta}(\Pi_{\mathcal A}\tilde{\mu}_{[\mathtt a|_k]}*\Pi_\mathcal B\bnu_{[\mathtt b|_k]})| \leq O(1)
\end{equation}
for every $k\in\N$. Here and in the following, we use the ``big-$O$'' notation $A \leq O(B)$ to indicate that $A \leq C B$ for some universal constant $C> 0$ independent of $A$ and $B$. In particular, in \eqref{eq-firstplace} the constant $O(1)$ does not depend on $\beta$, $\mathtt a$ or $\mathtt b$. Let $\io \in \Gamma^\N$ and $\jo\in\Lambda^\N$ be the unique words (which exist for almost every $\mathtt a,\mathtt b$) such that $\Pi_{\mathcal A}(\mathtt a)= \Pi_\Gamma(\io)$ and $\mathtt b|_k = \jo|_{i_k}$ for $k\in\N$, and note that 
\begin{align*}
    &\Pi_\mathcal A\tilde{\mu}_{[\mathtt a|_k]} = \mu_{\Pi_\mathcal A([\mathtt a|_k])} = \mu_{D_{\lfloor k\beta\rfloor} (\Pi_\Gamma(\io))}, \\
    &\Pi_\mathcal B \bnu_{[\mathtt b|_k]} = \psi_{\jo|_{i_k}} \nu.
\end{align*}
Using the bound $|H_n(S_m^* \zeta) - H_{n+m}(\zeta_{B(0,2^{-m})})|\leq O(1)$ which holds for any $n,m>0$ and probability measure $\zeta$ on $B(0,1)$, and Lemma \ref{lemma-almostcontinuity}, it follows from \eqref{eq-firstplace} that for every $k$, for $\mathtt i,\mathtt j$ chosen as above,
\begin{equation}\label{eq-secondplace}
|H_{(k+1)\beta}(g_\beta\eta_{[\mathtt a|_k] \times [\mathtt b|_k]}) - H_{\beta}(\mu^{D_{k\beta}(\Pi_\Gamma(\io))} * \nu_{\jo|_{i_k}})| \leq O(1).
\end{equation}
Therefore, by the assumption i) and the local entropy averages for the measure $\eta$ on the tree $\Sigma$ \cite[Theorem 4.4]{HochmanShmerkin2012}, we have $\dim g_\beta \eta \geq \alpha$. Finally, taking $g_\beta\eta$ through the faithful map $h_\beta$ yields $\mu * \nu$ and distorts the dimension by at most $O(1/\beta)$, by \cite[Proposition 5.2]{HochmanShmerkin2012}, which shows that $\dim(\mu*\nu) \geq \alpha - O(1/\beta) \geq \alpha-\varepsilon$ as long as $\beta\geq N_0$ is large enough, which is what we wanted to show.

Suppose then that the condition i) holds for $\bmu\times\bnu$-almost every $(\io,\jo)$ and for some $\beta= \beta(\io,\jo,\varepsilon) \geq N_0$ depending on $\io$ and $\jo$. Let $\varepsilon' < \varepsilon$ be small enough so that 
\begin{equation}\label{eq-choiceofep'}
\dim (\mu*\nu) \geq \inf \lbrace \dim E\subseteq \R^2: \mu*\nu(E) \geq \varepsilon'\rbrace - \varepsilon.
\end{equation}
By Egorov's theorem, we find a set $E\subseteq \Gamma^\N\times\Lambda^\N$ and an integer $N_1\geq N_0$ such that $\bmu\times\bnu(E) \geq 1-\varepsilon'/2$ and the condition i) holds for every $(\io,\jo)\in E$ and every $\beta\geq N_1$. Arguing as above with the measure $\eta$ replaced by the measure $\eta' := ((\Pi_\mathcal A|_{\mathcal A_{\rm adm}^\N})^{-1}\circ\Pi_\Gamma, {\rm Id})(\bmu\times\bnu)_E$ on $\Sigma$, where 
\begin{equation*}
((\Pi_\mathcal A|_{\mathcal A_{\rm adm}^\N})^{-1} \circ \Pi_\Gamma, {\rm Id}): (\mathtt i, \mathtt j) \mapsto ((\Pi_\mathcal A|_{\mathcal A_{\rm adm}^\N})^{-1}(\Pi_\Gamma(\io)), \jo),
\end{equation*}
the proof proceeds identically up until the equation \eqref{eq-firstplace} where we used the identity $(*\circ\tilde{\Pi})(\eta_{[\mathtt a|_k]\times[\mathtt b|_k]}) = \Pi_\mathcal A\tilde{\mu}_{[\mathtt a|_k]}*\Pi_\mathcal B\bnu_{[\mathtt b|_k]}$. Unfortunately, this is no longer true if $\eta$ is replaced by $\eta'$ because of the restriction to $E$ involved in the definition, and instead of \eqref{eq-firstplace}, we arrive at
$$
|H_{(k+1)\beta}(g_\beta\eta'_{[\mathtt a|_k]\times[\mathtt b|_k]}) - H_{(k+1)\beta}((*\circ\tilde{\Pi})(\eta'_{[\mathtt a|_k]\times[\mathtt b|_k]}))| \leq O(1).
$$

However, since $\eta'\ll\eta$, it follows from an application of the Lebesgue-Besicovitch differentiation theorem (cf. \cite[Proposition 3.8]{Hochmanpreprint}) that 
$$
\lim_{k\to\infty} \sup|\eta'_{[\mathtt a|_k]\times[\mathtt b|_k]}(A) - \eta_{[\mathtt a|_k]\times[\mathtt b|_k]}(A)| =0
$$
for $\eta'$-almost every $(\mathtt a,\mathtt b)\in\Sigma$, where the supremum is taken over all Borel sets $A\subset \Sigma$. In particular, 
$$
\lim_{k\to\infty} |H_{(k+1)\beta}((*\circ\tilde{\Pi})(\eta'_{[\mathtt a|_k]\times[\mathtt b|_k]})) - H_{(k+1)\beta} ((*\circ\tilde{\Pi})(\eta_{[\mathtt a|_k]\times[\mathtt b|_k]}))| = 0
$$
where $(*\circ\tilde{\Pi})(\eta_{[\mathtt a|_k]\times[\mathtt b|_k]}) = \Pi_\mathcal A\tilde{\mu}_{[\mathtt a|_k]}*\Pi_\mathcal B\bnu_{[\mathtt b|_k]}$. Combining this with \eqref{eq-secondplace}, the assumption i) and the local entropy averages for the measure $\eta'$ on $\Sigma$ \cite[Theorem 4.4]{HochmanShmerkin2012}, we conclude that $\dim (*\circ\tilde{\Pi}) \eta' \geq \alpha-O(1/\beta) \geq \alpha - \varepsilon$ when $\beta$ is large enough, and so $\dim(\mu*\nu) \geq \alpha -2 \varepsilon$ by \eqref{eq-choiceofep'}.
\end{proof}

\section{Proof of Theorems \ref{theorem-main} and \ref{theorem-main-2}}\label{section4}

In this section, we state our key technical results, and apply them to prove Theorems \ref{theorem-main} and \ref{theorem-main-2}. We begin with some notation.

Recall that we defined $\RP$ as the collection of one-dimensional subspaces of $\R^2$. Through the identification $\RP \cong [0,\pi)$ we use $\theta \in [0,\pi)$ to denote both angles and lines (making that angle with the positive $x$-axis). Recall that $R_\theta$ denotes the ``shortest'' rotation which takes $\theta$ onto $0^\perp$ (the $y$-axis), and for $\theta = 0$, we choose the clockwise rotation. 

For almost every $\io \in \Gamma^\N$ and $\theta\in\RP$, we let $\mu_{\io, \theta}$ denote the probability measure supported on the $y$-axis, obtained by taking the conditional measure $\mu_{\pi_\theta \Pi(\io)}^\theta$ from the disintegration $\mu = \int \mu_x^\theta \,d\pi_{\theta}\mu(x) = \int \mu_{\pi_\theta \Pi(\io)}^\theta\,d\bmu(\io)$ supported on the line $\theta^\perp + \Pi(\io)$, translating it by $T_{\Pi(\io)}$ and finally rotating it by $R_\theta$. That is,
\begin{equation}\label{eq-muitheta}
\mu_{\io,\theta} = R_\theta T_{\Pi(\io)} \mu_{\pi_\theta \Pi(\io)}^\theta.
\end{equation}
For $t \geq 0$, write $\mu_{\io, \theta, t} = S_t^* \mu_{\io, \theta}$. The measures $\mu_{\io,\theta}$ are occasionally called \emph{slices} of $\mu$. 

Let now $\mu$ be a self-affine measure associated to an iterated function system $\Phi$ as in Theorem \ref{theorem-main}. The strong separation condition asserts the existence of a bounded open set $V \neq \emptyset$ whose closure is mapped into disjoint subsets of $V$ by the maps $\varphi_i$. For the sake of notational simplicity, we suppose that $V = {\rm int}(B(0,1))$; Since all the statements in the following are local in nature, everything works also for a general $V$ by restricting onto small balls centered in the point of interest. 

Let $\beta$ be a large real number so that $2^{-\beta} < \min \lbrace d(\varphi_i(B(0,1)), \varphi_j(B(0,1))):\ i,j\in\Gamma\rbrace$. For $\io\in\Gamma^\N$ and $k\in\N$, let $\ell_k = \ell_k(\io,\beta)\in\N$ be defined by
\begin{equation}\label{eq-lk}
\ell_k(\io) = \max \lbrace n:\ \alpha_1(A_{\io|_n}) \geq 2^{-(k-1)\beta} \rbrace
\end{equation}
and notice that for every $k$, 
$$
B(\Pi(\io), 2^{-k\beta}) \cap \Pi(\Gamma^\N) \subseteq \varphi_{\io|_{\ell_k}}(B(0,1)).
$$
Our main geometric observation is that for any self-affine measure $\mu$ with $\dim\mu > 1$, the measures $\mu^{D_{k\beta}(\Pi(\io))}$ have a fiber structure in the sense that $\pi_{\theta(\io)^\perp} \mu^{D_{k\beta}(\Pi(\io))}$ is very close to a slice of the original measure $\mu$, in a direction typical to the Furstenberg measure $\mu_F$. This is true also when $\dim \mu\leq 1$, but in this case the proof is slightly different. 

Throughout the paper, we adopt the convention that $-\mu$ denotes the push-forward of $\mu$ through the map $x\mapsto -x$. Let $\pi^2: \R^2\to 0^\perp$ denote the orthogonal projection onto the $y$-axis. If $\nu$ is a measure and $I$ is an interval on the $y$-axis with $\nu(I)>0$, we let $\nu^I := F \nu_I$, where $F$ denotes the unique homothety on the $y$-axis taking the closure of $I$ onto $\lbrace 0\rbrace\times[-1,1]$. We say that two iterated function systems $\Phi = \lbrace \varphi_i\rbrace_{i\in\Gamma}$ and $\Psi = \lbrace \psi_i\rbrace_{i\in\Gamma}$ are \emph{linearly conjugate} if there exists a non-singular linear map $A$ such that $\varphi_i = A \circ \psi_i \circ A^{-1}$. In this case we say that $\Phi$ is a conjugate of $\Psi$ through $A$. 

\begin{proposition}[Fiber structure]\label{prop-parempijono}
Let $\Phi=\lbrace \varphi_i(x)= A_ix +a_i\rbrace_{i\in\Gamma}$ be a self-affine IFS with irreducibility, domination and strong separation conditions. For any $\varepsilon>0$, the following holds after conjugating $\Phi$ through a linear map:

Let $\mu$ be a fully supported self-affine measure associated to $\Phi$ with $\dim\mu > 1$. For $\bmu$-almost every $\io\in\Gamma^\N$, any large enough integer $N$ and all $\theta$ in a set of positive $\mu_F$-measure, there exists a sequence of intervals $(I_k)_k$ of length at least $1/2$ on the $y$-axis and a set $\cN_\varepsilon\subseteq\N$ with $\liminf_{n\to\infty} \frac{\#(\cN_\varepsilon\cap[0,n])}{n} \geq 1-\varepsilon$ such that
$$
d_{\rm LP}(\pi^2\mu^{D_{k\beta}(\Pi(\io))},\ \left(\rho(\ell_k, (\io,\theta)) \mu_{\sigma^{\ell_k}\io,\, A_{\io|_{\ell_k}}^{-1}\theta,\, k\beta -1+ \log \alpha_1(\io|_{\ell_k})}\right)^{I_k}) < \varepsilon
$$
for every $k\in\cN_\varepsilon$, where $\ell_k$ is as in \eqref{eq-lk} and $\rho: \N \times (\Gamma^\N \times \RP) \to \lbrace -1,1\rbrace$ is a cocycle.
\end{proposition}
The reason for having to pass to a conjugate IFS is purely technical and it arises only because in applications it is convenient to have the conclusion stated for magnifications of $\mu$ along the standard dyadic squares; We leave it for the interested reader to verify that if $\pi^2$ is replaced by $\pi_{\theta(\io)^\perp}$ and one chooses to magnify along balls centered at $\Pi(\io)$ instead of dyadic squares, then no domination condition, no passing to a conjugate IFS nor the restriction to $I_k$ is required. Such a modification is not required in the proof of Theorem \ref{theorem-main}, however. The proof of the proposition is given in Section \ref{section5}.

Let us briefly compare Proposition \ref{prop-parempijono} with the existing results on the scenery of self-affine measures. First, the result of Ferguson et. al. \cite{FergusonFraserSahlsten2015} deals with measures on Bedford-McMullen carpets. In this case, the Furstenberg measure $\mu_F$ equals a point mass, and the defining matrices involve no reflections, whence the cocycle $\rho$ above may be replaced by the identity and the map $M$ is just the left shift on the first argument. Using the carpet structure, the authors also describe the fibers of $\mu^{D_{k\beta}(\Pi(\io))}$ with respect to $\pi^2$, which seems to be difficult in our setting. 

The result of Kempton \cite{Kempton2015}, on the other hand, deals with self-affine measures under the irreducibility and domination conditions. In addition, the methods of \cite{Kempton2015} assume that the orthogonal projection of $\mu$ in $\mu_F$-almost every direction is absolutely continuous, and that the defining affine contractions preserve orientation. The projection condition allows Kempton to deduce that each fiber of $\mu^{D_{k\beta}(\Pi(\io))}$ with respect to $\pi^2$ is just the Lebesgue measure, while assuming the defining contractions to be orientation-preserving allows the cocycle $\rho$ to be replaced with the identity. Finally, Proposition \ref{prop-parempijono} considers magnifications of $\mu$ along the dyadic squares instead of balls centered at $\Pi(\io)$, which brings in some additional technical considerations regarding the distribution of $\mu$ near the boundaries of dyadic squares.

\subsection{Estimating the local entropy averages}

Recall from Section \ref{section3} that we have to find a lower bound for either $\frac{1}{\beta}H_\beta(\mu^{D_{k\beta}(\Pi(\io))} * \nu_{\jo|_{i_k}})$ or $\frac{1}{\beta}H_\beta(\mu_{\io|_{i_k}} * \nu_{\jo|_{i_k}})$, for most $k$. Bounding the second quantity is easier and can be done without the domination condition, but the bound we obtain is weaker:

\begin{claim}\label{claim-easycase}
Let $\mu$ and $\nu$ be as in Theorem \ref{theorem-main-2}, and let $\varepsilon>0$. For $\bmu$-a.e. $\io\in\Gamma^\N$, $\bnu$-a.e. $\jo\in\Lambda^\N$ and any $N_0\in\N$, there exists a real number $\beta\geq N_0$ such that 
$$
\liminf_{n\to\infty} \frac{1}{n}\sum_{k=0}^{n-1} \frac{1}{\beta} H_\beta(\mu_{\io|_{i_k}} * \nu_{\jo|_{i_k}}) \geq \min\lbrace 1,\dim\mu\rbrace + \min \lbrace 1,\dim \nu\rbrace - \varepsilon.
$$
\end{claim}
The proof is given in Section \ref{section6}. Theorem \ref{theorem-main-2} is immediate from Claim \ref{claim-easycase}. 
\begin{proof}[Proof of Theorem \ref{theorem-main-2}] 
For any $\varepsilon>0$, applying Claim \ref{claim-easycase} and Theorem \ref{thm-localentropyaverages} yields that 
$$
\dim (\mu*\nu) \geq \min \lbrace 1,\dim \mu \rbrace + \min \lbrace 1,\dim\nu\rbrace-2\varepsilon
$$
which proves Theorem \ref{theorem-main-2}.
\end{proof}
From Theorem \ref{theorem-main-2} it readily follows that if $\mu$ and $\nu$ satisfy the assumptions of Theorem \ref{theorem-main} and
\begin{equation}\label{eq-counterassumption}
\dim (\mu*\nu) < \min \lbrace 2, \dim\mu + \dim\nu\rbrace,
\end{equation}
then $\dim \mu > 1>\dim\nu$. From now on, we suppose that $\dim\mu>1>\dim\nu$, and that $\Phi$ and $\Psi$ satisfy the domination condition. 

The strategy of the rest of the proof is to assume that the arithmetic conclusion of Theorem \ref{theorem-main} does not hold, and use this assumption to obtain a lower bound for $\frac{1}{\beta}H_\beta(\mu^{D_{k\beta}(\Pi(\io))} * \nu_{\jo|_{i_k}})$ which, when combined with Theorem \ref{thm-localentropyaverages}, results in a contradiction with \eqref{eq-counterassumption}. Applying Lemmas \ref{lemma-entropyofconvolution} and \ref{lemma-chainruleapplication} we see that 
\begin{align}\label{eq-entropycomputations}
&H_\beta(\mu^{D_{k\beta}(\Pi(\io))}*\nu_{\jo|_{i_k}}) \\
\geq\ &H_\beta( \pi^2 \mu^{D_{k\beta}(\Pi(\io))} * \pi^2 \nu_{\jo|_{i_k}}) + H_\beta( \mu^{D_{k\beta}(\Pi(\io))} | \pi^2) -O(1)\nonumber\\
=\ &H_\beta( \pi^2\mu^{D_{k\beta}(\Pi(\io))} * \pi^2 \nu_{\jo|_{i_k}}) +  H_\beta(\mu^{D_{k\beta}(\Pi(\io))}) - H_\beta(\pi^2 \mu^{D_{k\beta}(\Pi(\io))}) - O(1) \nonumber
\end{align}
for every $\beta$. Dividing both sides by $\beta$, the average of the terms $\frac{1}{\beta}H_\beta(\mu^{D_{k\beta}(\Pi(\io))})$ over $k$ is known to equal $\dim \mu$. Thus, Theorem \ref{theorem-main} follows from the following estimates together with the local entropy averages:

\begin{claim}\label{claim-1}
Let $\Phi$ be as in Theorem \ref{theorem-main}. For any $\varepsilon>0$, the following holds after conjugating $\Phi$ through the linear map given by Proposition \ref{prop-parempijono}: 

Let $\mu$ be a fully supported self-affine measure associated to $\Phi$ with $\dim\mu > 1$. For any $N_0\in\N$, there exists a real number $\beta\geq N_0$ such that for $\bmu$-a.e. $\io\in\Gamma^\N$, we have
$$
\limsup_{n\to\infty} \frac{1}{n}\sum_{k=0}^{n-1} \frac{1}{\beta} H_\beta(\pi^2 \mu^{D_{k\beta}(\Pi(\io))}) \leq \dim \mu - 1 + \varepsilon.
$$
\end{claim}

\begin{claim}\label{claim-2}
Let $\Phi$ and $\Psi$ be as in Theorem \ref{theorem-main}, and suppose that there exists $(i,j)\in\Gamma\times\Lambda$ such that $\frac{\log|\lambda_1(A_i)|}{\log |\lambda_2(B_j)|}\not\in\Q$. For any $\varepsilon>0$, the following holds after conjugating $\Phi$ and $\Psi$ through the linear map given by Proposition \ref{prop-parempijono}: 

Let $\mu$ and $\nu$ be fully supported self-affine measures associated to $\Phi$ and $\Psi$ with $\dim\mu > 1 > \dim\nu$. For any $N_0\in\N$ there exists a real number $\beta\geq N_0$ such that for $\bmu$-a.e. $\io\in\Gamma^\N$ and $\bnu$-a.e. $\jo\in\Lambda^\N$, we have
$$
\liminf_{n\to\infty}\frac{1}{n}\sum_{k=0}^{n-1} \frac{1}{\beta} H_\beta(\pi^2\mu^{D_{k\beta}(\Pi(\io))} * \pi^2 \nu_{\jo|_{i_k}}) \geq \min \lbrace 1, \dim \mu - 1 + \dim \nu \rbrace - \varepsilon.
$$
\end{claim} 
While the domination condition could be relaxed from Claim \ref{claim-1} by slightly modifying the statement, in the proof of Claim \ref{claim-2} it plays a crucial role. These claims are proved in Section \ref{section6}. We now show how to conclude the proof of Theorem \ref{theorem-main}.

\begin{proof}[Proof of Theorem \ref{theorem-main}]
Let $\Phi =\lbrace \varphi_i\rbrace_{i\in\Gamma}$ and $\Psi=\lbrace \psi_j\rbrace_{j\in\Lambda}$ be as in the hypothesis, and suppose that there exist fully supported self-affine measures $\mu$ and $\nu$ such that $\dim\mu\geq\dim\nu$ and $\dim(\mu*\nu) < \min\lbrace 2,\dim\mu + \dim\nu\rbrace$. By Theorem \ref{theorem-main-2} we have $\dim\mu>1>\dim\nu$. For a contradiction, suppose that 
$$
\lbrace |\lambda_1 (A_i)|:\ i\in\Gamma\rbrace \cup \lbrace |\lambda_2(B_j)|:\ j\in\Lambda\rbrace
$$
is not an arithmetic set. It is not difficult to see that there must then exist a pair $(i,j)\in\Gamma\times\Lambda$ such that $\frac{\log |\lambda_1(A_i)|}{\log|\lambda_2(B_j)|} \not\in\Q$. 

Let $\varepsilon >0$ be small, and conjugate $\Phi$ and $\Psi$ through a linear map as in Claims \ref{claim-1} and \ref{claim-2}. Since $\Phi$ and $\Psi$ are conjugated through the same map, it is easy to see that a pair $(i,j)\in\Gamma\times\Lambda$ and measures $\mu$ and $\nu$ as above also exist for the conjugate IFSs. By \cite[Lemma 4.3]{HochmanShmerkin2012}, we have
$$
\dim \mu = \liminf_{n\to\infty} \frac{1}{n}\sum_{k=0}^{n-1} \frac{1}{\beta} H_\beta(\mu^{D_{k\beta}(\Pi(\io))})
$$
for $\bmu$-almost every $\io\in\Gamma^\N$ and any $\beta>0$, so by \eqref{eq-entropycomputations} and Claims \ref{claim-1} and \ref{claim-2}, for any $N_0\in\N$ there exists $\beta\geq N_0$ such that
\begin{align*}
    &\liminf_{n\to\infty} \frac{1}{n} \sum_{k=0}^{n-1} \frac{1}{\beta} H_\beta(\mu^{D_{k\beta}(\Pi(\io))} * \nu_{\jo|_{i_k}}) \\
    \geq\ &\min \lbrace 1, \dim\mu-1 + \dim \nu \rbrace + \dim \mu - (\dim \mu-1) - 2\varepsilon - O(1/\beta) \\
    =\ & \min \lbrace 2, \dim \mu + \dim \nu\rbrace -3\varepsilon
\end{align*}
for $\bmu$-almost every $\io \in \Gamma^\N$ and $\bnu$-almost every $\jo\in\Lambda^\N$. It now follows from Theorem \ref{thm-localentropyaverages} that $\dim(\mu*\nu) \geq \min\lbrace 2,\dim\mu+\dim\nu\rbrace -4 \varepsilon$, which is a contradiction if $\varepsilon$ is small enough.
\end{proof}

\section{The scenery of the self-affine measure}\label{section5}

In this section, our aim is to prove Proposition \ref{prop-parempijono}. We assume throughout the section that $\mu$ is a self-affine measure associated to an iterated function system $\Phi = \lbrace \varphi_i(x) = A_i x + a_i\rbrace_{i\in\Gamma}$ satisfying the irreducibility, domination and strong separation conditions, and that $\dim\mu>1$. The arguments of this section were inspired by the work of Kempton \cite{Kempton2015} on a similar result for a more special class of self-affine measures, and some of our lemmas are analogous to those in \cite{Kempton2015}.

\subsection{Restrictions of $\mu$ on thin rectangles}

Let us begin with the heuristics of why magnifications of $\mu$ are related to slices of $\mu$. Let $\io\in\Gamma^\N$, let $B$ be a small ball centered at $\Pi(\io)$ and let $n$ be the largest integer so that $\varphi_{\io|_n}(B(0,1)) \supseteq B$. By the strong separation condition, we have $\mu_B = \varphi_{\io|_n} \varphi_{\io|_n}^{-1} \mu_B = \varphi_{\io|_n} \mu_{\varphi_{\io|_n}^{-1}B}$. Therefore, in order for us to understand the magnifications $\mu_B$, it suffices to understand the measures $\mu_{\varphi_{\io|_n}^{-1}B}$, the restrictions of $\mu$ on the ellipses $\varphi_{\io|_n}^{-1}B$ whose major semi-axes have length comparable to $1$ and are parallel to $\theta(A_{\io|_n}^{-1})$. Since Lemma \ref{lemma-suuntasuppeneeprelminary} asserts that the ellipses $\varphi_{\io|_n}^{-1}B$ get thinner and thinner as $n$ increases, the measure $\mu_{\varphi_{\io|_n}^{-1}B}$ should be close to a slice of $\mu$ in the direction $\theta(A_{\io|_n}^{-1})$, when $n$ is large. 

In a more rigorous approach, it is better to work with rectangles instead of ellipses. For $\io \in \Gamma^\N$, $\theta \in \RP$, $r_2\geq r_1 \geq 0$, write $Y_{\io, \theta, r_1, r_2}$ for the rectangle centered at $\Pi(\io)$ with sidelengths $2^{-r_1} \geq 2^{-r_2}$ and the longer side oriented in the direction $\theta$. For a rectangle $Y$ with center at $x$ and major side oriented in the direction $\theta$, write $H_Y$ for the map which translates $x$ to the origin and stretches $T_x Y$ onto $R_{\theta}^{-1}[-1,1]^2$. It follows from Lemma \ref{prop-generalslices} that the measures $H_{Y_{\io,\theta,r_1, r_2}} \mu_{Y_{\io,\theta,r_1,r_2}}$ have a fiber structure in the following sense.

\begin{lemma}\label{lemma-fiberstructure1}
For $\bmu\times\mu_F$-a.e. $(\io,\theta)\in\Gamma^\N\times\RP$ and any $\varepsilon,\delta,r>0$, there exists $t>0$ such that if $0\leq r_1\leq r$, $r_2\geq t$, $Q$ is a translate of $[0,1/2]^2$ that contains the origin and 
$$
R_\theta H_{Y_{\io,\theta,r_1,r_2}} \mu_{Y_{\io,\theta,r_1,r_2}}(Q) \geq \delta,
$$
then 
$$
d_{\rm LP}\left( \pi^2( (R_\theta H_{Y_{\io,\theta,r_1,r_2}} \mu_{Y_{\io,\theta,r_1,r_2}})_{Q}),\ (\mu_{\io,\theta,r_1})_{\pi^2 Q} \right) < \varepsilon.
$$
\end{lemma}

For the proof, we record the following elementary observations.

\begin{lemma}[Lemma 3.3 of \cite{BaranyKaenmakiPyoralaWu2023}]\label{claim-restriction}
For any $r,\delta>0$, the following holds for all small enough $\varepsilon\geq\varepsilon'>0$:

If $\mu$ and $\nu$ are probability measures on $[-1,1]^d$ with $d_{\rm LP}(\mu,\nu)<\varepsilon'$ and $B = B(x,r)$ is a closed ball with $\min\lbrace \mu(B),\nu(B)\rbrace\geq \delta$ and $\nu(B(x,r+\varepsilon'))\leq \nu(B(x,r-\varepsilon'))+\varepsilon$, then 
$$
d_{\rm LP}(\mu_B, \nu_B) < O(\varepsilon/\delta).
$$
\end{lemma}
\begin{lemma}\label{lemma-countable}
Let $\mu$ be a probability measure on $\R$. If $I\subseteq \R$ is a closed interval which contains $0$, then $\mu(x+I) > 0$ for $\mu$-almost every $x\in\R$. 
\end{lemma}
\begin{proof}
    The distance between any two points in $E := \lbrace x\in\R:\ \mu(x+I) = 0\rbrace \cap {\rm spt}\,\mu$ is at least $|I|/2>0$. In particular, $E$ is countable, and so $\mu(E) \leq \bigcup_{x\in E} \mu(x+I) = 0$.  
\end{proof}

\begin{proof}[Proof of Lemma \ref{lemma-fiberstructure1}]
Recall from \eqref{eq-muitheta} that $\mu_{\io,\theta} = R_\theta T_{\Pi(\io)}\mu_{\pi_\theta \Pi(\io)}^\theta$, where $\mu_{\pi_\theta \Pi(\io)}^\theta$ is a conditional measure of $\mu$ from the disintegration $\mu = \int \mu_y^\theta\,d\pi_\theta\mu(y) = \int \mu_{\pi_\theta \Pi(\io)}^\theta \,d\bmu(\io)$, and that $\mu_{\io,\theta,r} = S_r^*\mu_{\io,\theta}$. We claim that 
\begin{equation}\label{eq-infbound}
\inf \lbrace \mu_{\io,\theta,r_1}(\lbrace 0\rbrace \times [a,a+1/2]):\ -1/2\leq a \leq 0,\ 0\leq r_1\leq r \rbrace > 0
\end{equation}
for every $\theta\in\RP$ and $\bmu$-almost every $\io\in\Gamma^\N$. Indeed, for every $\theta\in\RP$ and $\pi_\theta\mu$-almost every $y\in\theta$, the measure $\mu_y^\theta$ exists and is supported on the line $y+\theta^\perp$. It follows from Lemma \ref{lemma-countable}, applied with $y+\theta^\perp$ in place of $\R$ and for the interval $R_\theta^{-1}(\lbrace0\rbrace\times [0,2^{-r}/4])\subseteq \theta^\perp$, that 
$$
\mu_{y}^\theta(x+R_\theta^{-1}(\lbrace0\rbrace\times [0,2^{-r}/4])) > 0
$$
for $\mu_{y}^\theta$-almost every $x\in y+\theta^\perp$. Since $\pi_\theta(x) = y$ for every $x\in y+\theta^\perp$, it follows that $\mu_{\Pi^{-1}(x),\theta}(\lbrace0\rbrace\times [0,2^{-r}/4]) = \mu_y^\theta(x+ R_\theta^{-1}(\lbrace0\rbrace\times [0,2^{-r}/4])) > 0$ for $\pi_\theta\mu$-almost every $y$ and $\mu_y^\theta$-almost every $x$. But since $\mu = \int \mu_{y}^\theta \,d\pi_\theta \mu(y)$, this means that 
$$
\mu_{\io,\theta, r_1}(\lbrace0\rbrace\times [0,1/4]) \geq \mu_{\io,\theta}(\lbrace0\rbrace\times [0,2^{-r}/4])> 0
$$
for $\bmu$-almost every $\io\in\Gamma^\N$ and every $0\leq r_1\leq r$. Analogously, $\mu_{\io,\theta,r_1}(\lbrace0\rbrace\times [-1/4, 0]) > 0$ for $\bmu$-almost every $\io$ and every $0\leq r_1\leq r$. Since the intervals of the form $\lbrace 0\rbrace \times [a,a+1/2]$ considered in \eqref{eq-infbound} always contain either $\lbrace 0\rbrace \times [-1/4,0]$ or $\lbrace 0\rbrace \times [0,1/4]$, \eqref{eq-infbound} follows. 

Thus, by \eqref{eq-infbound}, for a given $\varepsilon > 0$, there exists $c > 0$ and a set $A_\varepsilon$ with $\bmu\times \mu_F(A_\varepsilon) > 1-\varepsilon$ such that $\mu_{\io,\theta,r_1}(\lbrace0\rbrace\times[a,a+1/2]) \geq c$ for every $(\io,\theta) \in A_\varepsilon$, $-1/2\leq a\leq 0$ and $0\leq r_1\leq r$. Since $\bmu\times\mu_F(\bigcup_{n\in\N} A_{1/n}) = 1$, it suffices to prove the statement for $(\io,\theta)\in A_\varepsilon$, for a given $\varepsilon>0$.

Recall that $\pi^1$ denotes the orthogonal projection to the $x$-axis. Lemma \ref{prop-generalslices} translated to the language of this section states that if $t\geq 0$ is large enough, then for every $r_2\geq t$, 
$$
d_{\rm LP}\left( \pi^2 (R_\theta H_{Y_{\io,\theta,r_1,r_2}} \mu_{Y_{\io,\theta,r_1,r_2}})_{(\pi^1)^{-1} \pi^1(Q)},\ \mu_{\io,\theta,r_1}\right) < \varepsilon.
$$
See Figure \ref{kuva1}. Since $\dim\mu>1$, $\mu_{\io,\theta,r_1}$ is non-atomic for $\bmu\times\mu_F$-almost every $(\io,\theta)$ by Theorem \ref{thm-dimensionconservation}. Moreover, $(\pi^1)^{-1} \pi^1(Q)\cap (\pi^2)^{-1} \pi^2 (Q) = Q$ and $\mu_{\io,\theta,r_1}(\pi^2 Q) \geq c$, so Lemma \ref{claim-restriction} combined with the above asserts that for a possibly even larger $t\geq 0$, for every $r_2\geq t$ and every $0\leq r_1\leq r$,
$$
d_{\rm LP}\left( \pi^2 (R_\theta H_{Y_{\io,\theta,r_1,r_2}} \mu_{Y_{\io,\theta,r_1,r_2}})_{Q},\ (\mu_{\io,\theta,r_1})_{\pi^2 Q}\right) <\varepsilon
$$
which is what we wanted to prove.
\end{proof}

\begin{figure}[H]
\centering
\begin{tikzpicture}

\draw (-5, -2) node[right]{$\theta$} -- (-5, 2);
\draw (-3.5, 0) node{$\longleftarrow_{\pi^2}$};
\draw (-5.1, -0.5) node[left]{$\pi^2 Q$} -- (-4.9, -0.5);
\draw (-5.1, 1.5) -- (-4.9, 1.5);

\draw[fill=black] (0,0) node[above,right]{$(0,0)$} circle(0.05cm);
\draw (-0.5,-0.5) node[below, left]{$Q$} rectangle (1.5,1.5);
\draw (-2,-2) rectangle (2,2);
\draw (2,-2) node[above,right]{$H_{Y_{\io,\theta,r_1, r_2}}Y_{\io,\theta,r_1, r_2}$};
\draw[thick, dotted] (-0.5, -2) rectangle (1.5, 2);
\draw (1.5, -1) node[right]{$(\pi^1)^{-1}\pi^1(Q)$};

\end{tikzpicture}
\caption{The measure $H_{Y_{\io,\theta,r_1, r_2}}\mu_{Y_{\io,\theta,r_1, r_2}}$ projects close to a slice of $\mu$ under $\pi^2$, even when restricted onto $(\pi^1)^{-1}\pi_1(Q)$.}\label{kuva1}
\end{figure}
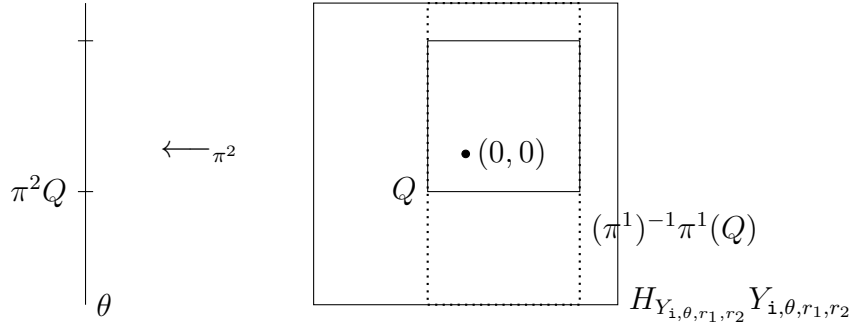

We will next relate the measures $\mu^{D_{k\beta}(\Pi(\io))}$ to the measures $H_{Y_{\io,\theta,r_1, r_2}} \mu_{Y_{\io,\theta,r_1,r_2}}$. To make statements more economic, we introduce some additional notation. 

\subsection{Magnifications of $\mu$}

For $\ifi \in \Gamma^*$, let $A_\ifi = U_\ifi D_\ifi V_\ifi^{-1}$ denote the singular value decomposition, where $U_\ifi, V_\ifi$ are orthogonal and $D_{\ifi} = {\rm diag}(\alpha_1(A_\ifi), \alpha_2(A_\ifi))$. Recall the definition of the sequence $(\ell_k)_{k\in\N}$ from \eqref{eq-lk}. Throughout the following, we will use the short-hand notation 
$$
Q_{\io, \theta, k} := Y_{\sigma^{\ell_k} \io,\, \theta,\, k\beta -2+ \log \alpha_1(\io|_{\ell_k}),\, k\beta-2+ \log\alpha_2(\io|_{\ell_k})}.
$$
Then the set $Q_{\io, \theta(A_{\io|_{\ell_k}}^{-1}), k}\subset \R^2$ is the smallest rectangle that contains the ellipse $\varphi_{\io|_{\ell_k}}^{-1}(B(\Pi(\io),2^{-k\beta+2}))$.
\begin{proposition}\label{prop-asymptotics}
For every $\io\in\Gamma^\N$, every large enough $\beta>0$ and every $\theta\in B(\theta(\io)^\perp, 1/5)$, there exists a sequence of non-singular linear maps $(L_{\io,\theta,k})_{k\in\N}$ such that
$$
S_{k\beta-1}^*T_{\Pi(\io)}\mu= S_1^*U_{\io|_{\ell_k}} V_{\io|_{\ell_k}}^{-1} L_{\io,\theta,k} H_{Q_{\io,A_{\io|_{\ell_k}}^{-1}\theta, k}}\mu_{Q_{\io, A_{\io|_{\ell_k}}^{-1}\theta, k}}
$$
for all large enough $k$.
\end{proposition}

An important point of the proposition is that the direction $\theta$ can be chosen arbitrarily from a large set. In proving this we require the following geometric lemma, analogous to \cite[Lemma 8.2]{Kempton2015}. Write $E_{\io, \theta, k} \subseteq Q_{\io,\theta, k}$ for the largest ellipse contained in $Q_{\io, \theta, k}$. Note that $E_{\io, \theta(A_{\io|_{\ell_k}}^{-1}), k} = \varphi_{\io|_{\ell_k}}^{-1}(B(\Pi(\io),2^{-k\beta+2}))$. 

\begin{lemma}[Figure \ref{ellipsi}]\label{lemma-geometriclemma}
For $\bmu$-almost every $\io\in\Gamma^\N$ and every $\theta\in B(\theta(\io)^\perp, 1/5)$, for every large enough $k$, 
$$
S_{-1} T_{\Pi(\sigma^{\ell_k}\io)}E_{\io, \theta(A_{\io|_{\ell_k}}^{-1}), k} \subseteq T_{\Pi(\sigma^{\ell_k}\io)}Q_{\io, A_{\io|_{\ell_k}}^{-1}\theta, k}.
$$
\end{lemma}

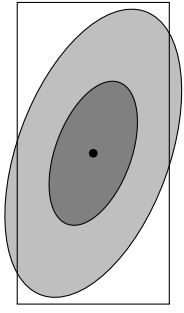
\begin{figure}[H]
\begin{tikzpicture}
\draw[rotate=-20, fill=lightgray] (0,0) ellipse (1cm and 2cm);
\draw (-1,-2) rectangle (1,2);
\draw[rotate=-20, fill=gray](0,0) ellipse (0.5cm and 1 cm);
\draw[fill=black] (0,0) circle (0.05cm);
\end{tikzpicture}
\caption{The ellipse $E_{\io, \theta(A_{\io|_{\ell_k}}^{-1}), k}$ (light gray) fits completely inside the rectangle $Q_{\io, A_{\io|_{\ell_k}}^{-1}\theta, k}$ after the axes are scaled by $1/2$.}\label{ellipsi}
\end{figure}

\begin{proof}
For any $\theta\in\RP$, $\io\in \Gamma^\N$ and $k\in \N$, we have 
\begin{equation}\label{eq-tangentti}
\tan d(A_{\io|_k}^{-1} \theta,\ \theta(A_{\io|_k}^{-1})) = \frac{\alpha_1(\io|_k)}{\alpha_2(\io|_k)} \tan d(\theta,\  \theta(A_{\io|_k})^\perp).
\end{equation}
See Figure \ref{kolmioita}.
\begin{figure}[H]

\begin{tikzpicture}
\draw(0,0) coordinate(origo) -- node[below]{$\ell$} (2,0) coordinate(1) -- node[right]{$1$} (2,1) coordinate(2) -- (0,0);
\draw(4,0.5) node{$A_{\io|_k}^{-1}\ \mapsto$};
\draw (6,0) coordinate(a) -- node[below]{$\alpha_2(\io|_k)^{-1}\ell$} (8.5,0) coordinate(b) -- node[right]{$\alpha_1(\io|_k)^{-1}$} (8.5, 2) coordinate(c) -- (6,0);
\pic [draw, "$\beta$", angle eccentricity=1.5] {angle = 1--origo--2};
\pic [draw, "$\beta'$", angle eccentricity=1.5] {angle = b--a--c};

\end{tikzpicture}
\caption{Here $\beta = d(\theta,\ \theta(A_{\io|_k})^\perp)$ which is taken to $\beta' = d(A_{\io|_k}^{-1} \theta,\ \theta(A_{\io|_k}^{-1}))$ by $A_{\io|_k}^{-1}$. If the side of length $1$ is parallel to $\theta(A_{\io|_k})^\perp$ (which is the major expanding direction of $A_{\io|_k}^{-1}$), then its image through $A_{\io|_k}^{-1}$ is of length $\alpha_1(\io|_k)^{-1}$.}\label{kolmioita}
\end{figure}
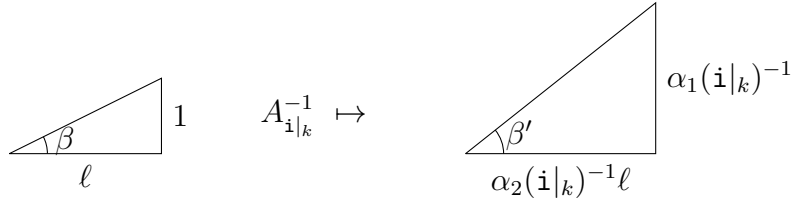

Let $\eta = d(A_{\io|_k}^{-1} \theta,\ \theta(A_{\io|_k}^{-1}))$ and define $X = [-\alpha_1 (\io|_k),\alpha_1 (\io|_k)] \times [-\alpha_2(\io|_k),\alpha_2(\io|_k)]$ and $Y = R_\eta^{-1}([-\tfrac{\alpha_1 (\io|_k)}{2}, \tfrac{\alpha_1 (\io|_k)}{2}] \times [-\tfrac{\alpha_2 (\io|_k)}{2}, \tfrac{\alpha_2 (\io|_k)}{2}])$. It is not difficult to see that $Y\subseteq X$ if we have $\sin \eta \cdot \tfrac{\alpha_2 (\io|_k)}{2} +\tfrac{\alpha_1 (\io|_k)}{2} \leq \alpha_1 (\io|_k)$, or equivalently,
$$
\sin \eta \leq \frac{\alpha_1 (\io|_k)}{\alpha_2(\io|_k)}.
$$
Since $\eta \to 0$ as $k\to\infty$ by Lemma \ref{cor-almostcontraction}, we have $\sin \eta \leq 2\tan \eta$ for large enough $k$. Therefore, for $\bmu$-almost every $\io\in\Gamma^\N$ and every $\theta\in B(\theta(\io)^\perp, 1/5)$, it follows from \eqref{eq-tangentti} and Lemma \ref{lemma-suuntasuppeneekaikkialla} that
$$
\sin \eta \leq \frac{2\alpha_1 (\io|_k)}{\alpha_2(\io|_k)} \tan d(\theta,\ \theta(A_{\io|_k})^\perp) \leq \frac{\alpha_1 (\io|_k)}{\alpha_2(\io|_k)} 2\tan(1/4) < \frac{\alpha_1 (\io|_k)}{\alpha_2(\io|_k)},
$$
for all large enough $k$.
\end{proof}

\begin{proof}[Proof of Proposition \ref{prop-asymptotics}]
Let $\beta>0$ be large enough so that 
$$
2^{-\beta}< \min\lbrace d(\varphi_i(B(0,1)),\ \varphi_j(B(0,1)):\ i,j\in\Gamma\rbrace,
$$
let $(\io, \theta)\in\Gamma^\N\times\RP$, let $k$ be large, and let $\ell_k$ be defined as in \eqref{eq-lk}. Write $B_{\io|_{\ell_k}} := U_{\io|_{\ell_k}} D^{-1}_{\io|_{\ell_k}} U_{\io|_{\ell_k}}^{-1}$. Recalling that $E_{\io, \theta(A_{\io|_{\ell_k}}^{-1}), k} = (\varphi_{\io|_{\ell_k}})^{-1} B(\Pi(\io), 2^{-k\beta+2})$, it is not difficult to see that $S_{k\beta-2} B_{\io|_{\ell_k}}^{-1}$ is the linear map which scales the origocentric ellipse $U_{\io|_{\ell_k}}V_{\io|_{\ell_k}}^{-1} T_{\varphi_{\io|_{\ell_k}}^{-1}\Pi(\io)} E_{\io, \theta(A_{\io|_{\ell_k}}^{-1}), k}$ onto $B(0,1)$ without rotating it. 

By the strong separation condition, with this notation we may write
\begin{align}\label{eq-ekalasku}
S_{k\beta-2}^*T_{\Pi(\io)} \mu &= S_{k\beta-2}^* T_{\Pi(\io)} \varphi_{\io|_{\ell_k}} \mu_{E_{\io, \theta(A_{\io|_{\ell_k}}^{-1}), k}}\nonumber\\
&= S_{k\beta-2}^* B_{\io|_{\ell_k}}^{-1} B_{\io|_{\ell_k}} A_{\io|_{\ell_k}} T_{\varphi_{\io|_{\ell_k}}^{-1} \Pi(\io)}\mu_{E_{\io, \theta(A_{\io|_{\ell_k}}^{-1}), k}} \nonumber\\
&= S_{k\beta-2}^* B_{\io|_{\ell_k}}^{-1}U_{\io|_{\ell_k}}V_{\io|_{\ell_k}}^{-1} T_{\varphi_{\io|_{\ell_k}}^{-1}\Pi(\io)} \mu_{E_{\io, \theta(A_{\io|_{\ell_k}}^{-1}), k}} \nonumber\\
&= U_{\io|_{\ell_k}} V_{\io|_{\ell_k}}^{-1} H_{Q_{\io, \theta(A_{\io|_{\ell_k}}^{-1}), k}} \mu_{E_{\io, \theta(A_{\io|_{\ell_k}}^{-1}), k}}.
\end{align}
by switching the order of scaling and rotation in the last equality. This is almost the representation we are seeking. The reason we are not yet content with this is that we would ultimately like to apply Lemma \ref{lemma-fiberstructure1} to say that the measure $R_{\theta(A_{\io|_{\ell_k}}^{-1})} H_{Q_{\io, \theta(A_{\io|_{\ell_k}}^{-1}), k}} \mu_{E_{\io, \theta(A_{\io|_{\ell_k}}^{-1}), k}}$ is close to $\mu_{\sigma^{\ell_k}\io, \theta(A_{\io|_{\ell_k}}^{-1}), t}$ for some $t\geq0$. However, we are able to say this only if $\theta(A_{\io|_{\ell_k}}^{-1})$ is replaced by some $\theta$ which is drawn randomly with respect to $\mu_F$. Although the sequence $(\theta(A_{\io|_{\ell_k}}^{-1}))_{k\in\N}$ does in fact equidistribute for $\mu_F$, for $\bmu$-almost every $\io\in\Gamma^\N$, it is only in the weak-$^*$ sense and the lack of any continuity for the function $(\io,\theta)\mapsto \mu_{\io,\theta}$ makes it difficult to say anything about the relationship of $R_{\theta(A_{\io|_{\ell_k}}^{-1})}H_{Q_{\io, \theta(A_{\io|_{\ell_k}}^{-1}), k}} \mu_{E_{\io, \theta(A_{\io|_{\ell_k}}^{-1}), k}}$ and $\mu_{\sigma^{\ell_k}\io, \theta(A_{\io|_{\ell_k}}^{-1})}$. Therefore, we devote the rest of the proof to the technical work of replacing the measure $H_{Q_{\io, \theta(A_{\io|_{\ell_k}}^{-1}), k}} \mu_{E_{\io, \theta(A_{\io|_{\ell_k}}^{-1}), k}}$ by $H_{Q_{\io, A_{\io|_{\ell_k}}^{-1}\theta, k}} \mu_{Q_{\io, A_{\io|_{\ell_k}}^{-1}\theta, k}}$ in \eqref{eq-ekalasku}, where $\theta$ can be drawn freely from a set of positive $\mu_F$-measure, so that we may eventually apply Lemma \ref{lemma-fiberstructure1}.

If we now let $\theta \in B(\theta(\io)^\perp, 1/5)$, then by Lemma \ref{lemma-geometriclemma} we have
\begin{equation}\label{eq-subsetC}
S_{-1} T_{\Pi(\sigma^{\ell_k}\io)}E_{\io, \theta(A_{\io|_{\ell_k}}^{-1}), k} \subseteq T_{\Pi(\sigma^{\ell_k}\io)}Q_{\io, A_{\io|_{\ell_k}}^{-1}\theta, k}
\end{equation}
for every large enough $k$. Using the general fact that $\mu_X = (\mu_Y)_X$ whenever $X\subseteq Y\subseteq \R^2$, we obtain from \eqref{eq-subsetC} that
$$
\left(T_{\Pi(\sigma^{\ell_k}\io)}\mu_{E_{\io, \theta(A_{\io|_{\ell_k}}^{-1}), k}}\right)_{S_{-1} T_{\Pi(\sigma^{\ell_k}\io)}E_{\io, \theta(A_{\io|_{\ell_k}}^{-1}), k}} = \left(T_{\Pi(\sigma^{\ell_k}\io)}\mu_{Q_{\io, A_{\io|_{\ell_k}}^{-1}\theta, k}}\right)_{S_{-1} T_{\Pi(\sigma^{\ell_k}\io)}E_{\io, \theta(A_{\io|_{\ell_k}}^{-1}), k}}.
$$
Pushing both measures forward under $H_{Q_{\io, \theta(A_{\io|_{\ell_k}}^{-1}), k}}\circ T_{\Pi(\sigma^{\ell_k}\io)}^{-1}$ and using the identity $B(0,1) = H_{Q_{\io, \theta(A_{\io|_{\ell_k}}^{-1}), k}}T_{\Pi(\sigma^{\ell_k}\io)}^{-1} T_{\Pi(\sigma^{\ell_k}\io)}E_{\io, \theta(A_{\io|_{\ell_k}}^{-1}), k}$ which is immediate from the definitions of $H$ and $E$, we obtain 
\begin{equation}\label{eq-submeasureC}
\left(H_{Q_{\io, \theta(A_{\io|_{\ell_k}}^{-1}), k}} \mu_{E_{\io, \theta(A_{\io|_{\ell_k}}^{-1}), k}}\right)_{B(0,1/2)} = \left(H_{Q_{\io, \theta(A_{\io|_{\ell_k}}^{-1}), k}} \mu_{Q_{\io, A_{\io|_{\ell_k}}^{-1}\theta, k}}\right)_{B(0,1/2)}.
\end{equation}
See Figure \ref{suunnanvaihto}.

\begin{figure}[H]

\begin{tikzpicture}
\draw[rotate=20] (-1,2) node[right]{} rectangle (1,-2) node[right]{$Q_{\io, A_{\io|_{\ell_k}}^{-1}\theta, k}$};
\draw[fill=lightgray](0,0) ellipse (0.5cm and 1 cm);
\draw[thick, dotted] (0,0) ellipse (1cm and 2cm) node at (2, 0.5) {$E_{\io, \theta(A_{\io|_{\ell_k}}^{-1}), k}$};
\end{tikzpicture}
\begin{tikzpicture}
\draw[white] (0,2) node[black]{$H_{\ldots}\,\rightarrow$}--(0.5,0);
\end{tikzpicture}
\resizebox {7cm} {\height}{
\begin{tikzpicture}
\draw[rotate=20] (-1,2) node[right]{} rectangle (1,-2);
\draw[fill=lightgray](0,0) ellipse (0.5cm and 1 cm);
\draw[thick, dotted] (0,0) ellipse (1cm and 2cm) ;
\end{tikzpicture}
}
\caption{The precise ``location'' of the line $A_{\io|_{\ell_k}}^{-1}\theta$ in $\RP$ is easier to control than that of $\theta(A_{\io|_{\ell_k}}^{-1})$ even though the lines are very close to each other. Because of this, we want to consider restrictions of $\mu$ on $Q_{\io,A_{\io|_{\ell_k}}^{-1}\theta,k}$ instead of on $E_{\io, \theta(A_{\io|_{\ell_k}}^{-1}), k}$.}\label{suunnanvaihto}
\end{figure}
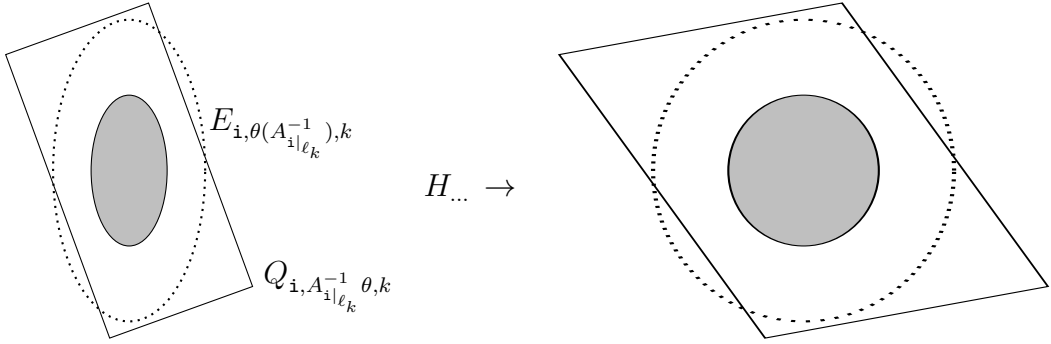

Writing
$$
L_{\io, \theta,k} := H_{Q_{\io, \theta(A_{\io|_{\ell_k}}^{-1}), k}} H_{Q_{\io, A_{\io|_{\ell_k}}^{-1}\theta, k}}^{-1}
$$
and combining \eqref{eq-submeasureC} with \eqref{eq-ekalasku}, we obtain
$$
S_{k\beta-1}^*T_{\Pi(\io)} \mu = S_{1}^*U_{\io|_{\ell_k}} V_{\io|_{\ell_k}}^{-1}L_{\io,\theta,k} H_{Q_{\io, A_{\io|_{\ell_k}}^{-1}\theta, k}} \mu_{Q_{\io, A^{-1}_{\io|_{\ell_k}}\theta, k}}
$$
for all large enough $k$. This completes the proof. 
\end{proof}

Note that because $D_{k\beta}(\Pi(\io))\subseteq B(\Pi(\io), 2^{-k\beta+1})$, it follows that
\begin{equation}\label{eq-corofProp51}
\mu^{D_{k\beta}(\Pi(\io))} = (S_{k\beta-1}^*T_{\Pi(\io)} \mu)_Q = (S_1^* U_{\io|_{\ell_k}} V_{\io|_{\ell_k}}^{-1}L_{\io,\theta,k} H_{Q_{\io, A_{\io|_{\ell_k}}^{-1}\theta, k}} \mu_{Q_{\io, A^{-1}_{\io|_{\ell_k}}\theta, k}})_{Q}
\end{equation}
where $Q$ is a closed square of side length $2^{k\beta -\lfloor k\beta\rfloor}$ that contains the origin. If the map $U_{\io|_{\ell_k}} V_{\io|_{\ell_k}}^{-1} L_{\io,\theta,k}$ were just the rotation $R_{A_{\io|_{\ell_k}}^{-1}\theta}$, then all that would be left to obtain the fiber structure of $\mu^{D_{k\beta}(\Pi(\io))}$ were to combine \eqref{eq-corofProp51} and Lemma \ref{lemma-fiberstructure1}; See Figure \ref{hankalatilanne}. 
\begin{figure}[H]
\begin{tikzpicture}
    \draw[rotate=10] (0,0) -- (3,0) -- (4,3) node[right]{$S_1 U_{\io|_{\ell_k}} V_{\io|_{\ell_k}}^{-1} L_{\io,\theta,k} H_{Q_{\io,A_{\io|_{\ell_k}}^{-1}\theta, k}}Q_{\io, A_{\io|_{\ell_k}}^{-1}\theta, k}$} -- (1,3) -- (0,0);
    \draw (1,1.25) rectangle (2.5,2.75) node[right]{$Q$};
    \draw[rotate=10, fill=black] (2, 1.5) node[above]{$(0,0)$} circle (0.05cm);
\end{tikzpicture}
\caption{For us to be able to directly apply Lemma \ref{lemma-fiberstructure1}, the above parallelogram would have to be a square similar to $Q$.}\label{hankalatilanne}
\end{figure}
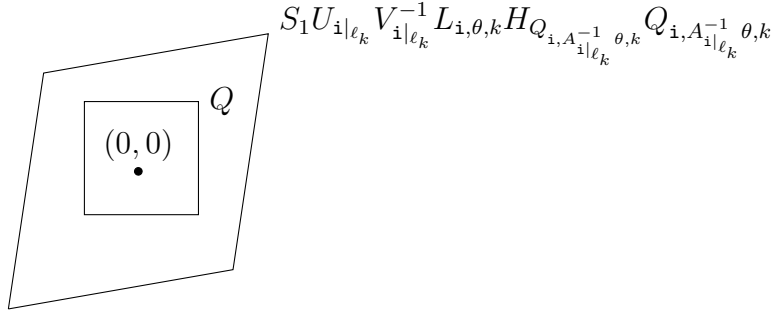 
Unfortunately, because of the involvement of the map $L_{\io,\theta,k}$, the map $U_{\io|_{\ell_k}} V_{\io|_{\ell_k}}^{-1} L_{\io,\theta,k}$ is not merely a rotation, and some technical work is required to deal with this additional distortion. We begin by breaking the map $U_{\io|_n} V_{\io|_n}^{-1}$ into three components: The reflection, and two rotations which do not reflect. 

\subsection{The distortion $U_{\io|_{\ell_k}} V_{\io|_{\ell_k}}^{-1} L_{\io,\theta,k}$}

For a linear map $A: \R^2\to\R^2$ and $\theta\in\RP$, let $A|_\theta:\ \theta\to\R^2$ denote the restriction of $A$ onto $\theta$. For every $\theta \in \RP$, let $e_\theta$ denote the unit vector of $\theta$ with non-negative $y$-coordinate. Let $O \subset \Gamma^\N \times \RP$ be the open set of those $(\io, \theta)$ for which $\langle e_{A_{i_0}^{-1}\theta},\, A_{i_0}^{-1}e_\theta\rangle <0$. Here $\langle \cdot,\, \cdot \rangle$ denotes the Euclidean inner product. Recall that $M$ denotes the map $(\io,\theta) \mapsto (\sigma \io, A_{i_0}^{-1}\theta)$. Define the function $\rho:\ \N\times (\Gamma \times \RP) \to \lbrace -1,1\rbrace$, 

\begin{equation}\label{eq-cocycle}
\rho(n, (\io, \theta)) = \prod_{k=1}^n (-1)^{\textbf{1}_{O}(M^k(\io, \theta))}
\end{equation}
where $\textbf{1}_O$ denotes the indicator of $O$. The map $\rho$ captures the reflections done by $A_{\io|_n}^{-1}$ on the line $\theta$, or by $A_{\io|_n}$ on the line $A_{\io|n}^{-1}\theta$. Indeed, for any $x \in \theta$, we may decompose $A_{\io|_n}^{-1}$ as
\begin{equation}\label{eq-roomaar}
A_{\io|_n}^{-1} x = \Vert A_{\io|_n}^{-1}x \Vert R_{A_{\io|_n}^{-1}\theta}^{-1} \rho(n, (\io, \theta)) R_{\theta} x.
\end{equation}
In other words, first rotate $x$ to the $y$-axis, apply the possible reflections, rotate the $y$-axis onto the line $A_{\io|_n}^{-1}\theta$, and finally scale. The map $\rho$ is easily seen to satisfy the \emph{cocycle equation}
$$
\rho(n+k, (\io, \theta)) = \rho(n, M^k(\io, \theta)) \rho(k, (\io, \theta))
$$
for every $n,k\in\N$.

\begin{lemma}\label{lemma-peilausulos}
Let $\io\in \Gamma^\N$ and let $U_{\io|_n} D_{\io|_n} V^{-1}_{\io|_n}$ be the singular value decomposition of $A_{\io|_n}$. For any $\theta\in\RP\setminus \lbrace\theta(\io)\rbrace$, we have
$$
\lim_{n\to \infty}\Vert U_{\io|_n} V_{\io|_n}^{-1}|_{A_{\io|_n}^{-1}\theta} - R_{\theta(\io)^\perp}^{-1} \rho(n, (\io,\theta)) R_{A_{\io|_n}^{-1}\theta}|_{A^{-1}_{\io|_n}\theta} \Vert =0.
$$
\end{lemma}

\begin{proof}

Let $B_{\io|_n} := U_{\io|_n} D_{\io|_n}^{-1} U_{\io|_n}^{-1}$ and note that
$$
A_{\io|_n} = B_{\io|_n}^{-1} B_{\io|_n} A_{\io|_n} = B_{\io|_n}^{-1} U_{\io|_n} V_{\io|_n}^{-1}.
$$
In particular, $U_{\io|_n} V_{\io|_n}^{-1} = B_{\io|_n} A_{\io|_n}$. 

Now, by \eqref{eq-roomaar}, 
\begin{equation}\label{eq-Ain}
A_{\io|_n}|_{A^{-1}_{\io|_n}\theta} = \Vert A_{\io|_n}|_{A^{-1}_{\io|_n}\theta} \Vert R_\theta^{-1} \rho(n, (\io, \theta)) R_{A_{\io|_n}^{-1}\theta}|_{A^{-1}_{\io|_n}\theta}.
\end{equation}
On the other hand, it follows from Lemma \ref{lemma-A^n} and some basic geometry that
$$
\left\Vert B_{\io|_n}|_\theta - (-1)^{k_n}\Vert A_{\io|_n}^{-1}|_{\theta}\Vert R_{\theta(\io)^\perp}^{-1} R_\theta|_\theta\right\Vert\to 0
$$
as $n\to\infty$, for some sequence $(k_n)_n\in \lbrace 1,2\rbrace^\N$. Since $B_{\io|_n}$ is positive definite and $\theta\neq\theta(\io)$, the sequence $(-1)^{k_n}$ has to be eventually constant. By absorbing the eventual value of this sequence to the definition of $\rho(n,(\io,\theta))$, we may without loss of generality assume that 
\begin{equation}\label{eq-B}
\lim_{n\to\infty} \Vert A_{\io|_n}^{-1}|_\theta\Vert^{-1}B_{\io|_n}|_\theta = R_{\theta(\io)^\perp}^{-1}R_\theta|_\theta.
\end{equation}
Thus, combining \eqref{eq-Ain} and \eqref{eq-B}, we have
$$
\lim_{n\to\infty} \Vert B_{\io|_n} A_{\io|_n} |_{A^{-1}_{\io|_n}\theta} - R_{\theta(\io)^\perp}^{-1} \rho(n, (\io, \theta)) R_{A_{\io|_n}^{-1}\theta}|_{A^{-1}_{\io|_n}\theta}\Vert=0
$$
which is what we wanted to show.
\end{proof}

It remains to study the behaviour of linear map $L_{\io,\theta,k}$ from the statement of Proposition \ref{prop-asymptotics}, as $k\to\infty$. The content of the following lemma is that $U_{\io|_{\ell_k}} V_{\io|_{\ell_k}}^{-1}L_{\io,\theta,k}$ takes the square $H_{Q_{\io,A_{\io|_{\ell_k}}^{-1}\theta, k}}Q_{\io,A_{\io|_{\ell_k}}^{-1}\theta, k} = R_{A_{\io|_{\ell_k}}^{-1}\theta}^{-1}[-1,1]^2$ onto a parallelogram of bounded eccentricity and one side in direction $\theta(\io)$.

\begin{lemma}\label{lemma-suunnikas}
For every $\io\in\Gamma^\N$ and $\theta\in\RP\setminus\lbrace\theta(\io)\rbrace$, 
$$
\Vert U_{\io|_{\ell_k}} V_{\io|_{\ell_k}}^{-1}L_{\io,\theta,k}- R_{\theta(\io)^\perp}^{-1} \rho(\ell_k,(\io,\theta)) F_{\theta,\io}^{k} R_{A_{\io|_{\ell_k}}^{-1}\theta}\Vert \to 0
$$
as $k\to\infty$, where $F_{\theta, \io}^k= \begin{bmatrix} a_k &  b_k\cdot \tan d(\theta,\ \theta(\io)^\perp) \\ 0 & 1 \end{bmatrix}$ for some $a_k,b_k\in\lbrace -1,1\rbrace$.
\end{lemma}

\begin{proof}
Recall that $L_{\io,\theta,k} = H_{Q_{\io, \theta(A_{\io|_{\ell_k}}^{-1}), k}} H_{Q_{\io, A_{\io|_{\ell_k}}^{-1}\theta, k}}^{-1}$. As the map $L_{\io,\theta,k}$ acts on the set $R_{A_{\io|_{\ell_k}}^{-1}\theta}^{-1} [-1,1]^2$, it is first taken to the thin rectangle $H_{Q_{\io, A_{\io|_{\ell_k}}^{-1}\theta, k}}^{-1}R_{A_{\io|_{\ell_k}}^{-1}\theta}^{-1} [-1,1]^2$ (see Figure \ref{suunnikas}), and then stretched onto a parallelogram by $H_{Q_{\io, \theta(A_{\io|_{\ell_k}}^{-1}),k}}$.
\begin{figure}[H]
    \begin{tikzpicture}[scale=0.75]
    \draw[rotate=5, fill=lightgray] (-0.3,-2) rectangle (0.3,2);
    \draw (-0.3,2) rectangle (0.3,-2);
    \draw (0,-3);
    \end{tikzpicture}
    \begin{tikzpicture}
\draw[white] (0,2) node[black]{$H_{Q_{\io, \theta(A_{\io|_{\ell_k}}^{-1}), k}}\,\rightarrow$}--(0.5,0);
\end{tikzpicture}
    \resizebox{5 cm}{\width}{
    \begin{tikzpicture}[scale=3.75]
    \draw[rotate=5, fill=lightgray] (-0.3,-2) rectangle (0.3,2);
    \draw (-0.3,2) rectangle (0.3,-2);
    \draw (0,-3);
    \end{tikzpicture}
    }
    \caption{The set $H_{Q_{\io, \theta(A_{\io|_{\ell_k}}^{-1}), k}}^{-1}R_{\theta(A_{\io|_{\ell_k}}^{-1})}^{-1} [-1,1]^2$ is depicted as the thin white rectangle, while the set $H_{Q_{\io, A_{\io|_{\ell_k}}^{-1}\theta, k}}^{-1}R_{A_{\io|_{\ell_k}}^{-1}\theta}^{-1} [-1,1]^2$ is depicted as the thin gray rectangle.}\label{suunnikas}
\end{figure}
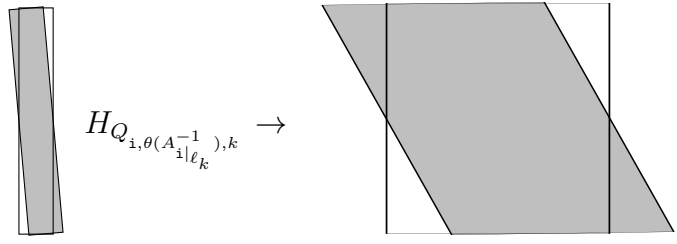
The angle between the two thin rectangles in Figure \ref{suunnikas} is $d(\theta(A_{\io|_{\ell_k}}^{-1}),\  A_{\io|_{\ell_k}}^{-1}\theta)$, so by \eqref{eq-tangentti} and basic geometry, the Hausdorff distance between the parallelograms $L_{\io,\theta,k}R_{A_{\io|_{\ell_k}}^{-1}\theta}^{-1} [-1,1]^2$ and $R_{\theta(A_{\io|_{\ell_k}}^{-1})}^{-1}  F_{\theta,\io}^{k}[-1,1]^2$ tends to $0$ as $k\to\infty$, where in the definition of $F^k_{\theta,\io}$ we have $a_k=1$ and $b_k\in \lbrace -1,1\rbrace$ depends on the order of $\theta(A_{\io|_{\ell_k}}^{-1})$ and $A_{\io|_{\ell_k}}^{-1}\theta$. In particular, 
$$
\Vert L_{\io,\theta,k} - R_{\theta(A_{\io|_{\ell_k}}^{-1})}^{-1}  F_{\theta,\io}^{k}R_{A_{\io|_{\ell_k}}^{-1}\theta} \Vert \to 0 
$$
as $k\to\infty$. It follows from Lemma \ref{cor-almostcontraction} that $d(\theta(A_{\io|_{\ell_k}}^{-1}),\  A_{\io|_{\ell_k}}^{-1}\theta) \to 0$, so using Lemma \ref{lemma-peilausulos} and incorporating the possible reflection for the line $\theta(\io)$ as the value of $a_k$ in the definition of $F_{\theta,\io}^k$ completes the proof.
\end{proof}

\subsection{Proof of Proposition \ref{prop-parempijono}}

We will now explain how to combine Proposition \ref{prop-asymptotics} and Lemmas \ref{lemma-fiberstructure1} and \ref{lemma-suunnikas} to prove Proposition \ref{prop-parempijono} and obtain the fiber structure of $\mu^{D_{k\beta}(\Pi(\io))}$. In \eqref{eq-corofProp51} we noted that by Proposition \ref{prop-asymptotics}, the dyadic maginifactions of $\mu$ have the form
\begin{equation*}
\mu^{D_{k\beta}(\Pi(\io))} = (S_{k\beta-1}^* T_{\Pi(\io)}\mu)_{Q} = (S_1^*U_{\io|_{\ell_k}} V_{\io|_{\ell_k}}^{-1} L_{\io,\theta,k} H_{Q_{\io,A_{\io|_{\ell_k}}^{-1}\theta, k}}\mu_{Q_{\io, A_{\io|_{\ell_k}}^{-1}\theta, k}})_Q
\end{equation*}
where $Q$ is a square of side length $2^{k\beta-\lfloor k\beta\rfloor}$ that contains the origin. Therefore, what stands in the way of using Lemma \ref{lemma-fiberstructure1} to relate $\mu^{D_{k\beta}(\Pi(\io))}$ to a slice of $\mu$ is the presence of the distortion $U_{\io|_{\ell_k}} V_{\io|_{\ell_k}}^{-1} L_{\io,\theta,k}$.

However, because of our freedom in choosing the direction $\theta$, this issue can be resolved by replacing $\Phi$ by a conjugate IFS so that $0^\perp \in \theta^-(\Gamma^\N)$ and $\theta(\Gamma^\N)\subseteq B(0, \varepsilon)$ for some small $\varepsilon>0$. For the conjugated IFS, the map $\pi^2 U_{\io|_{\ell_k}} V_{\io|_{\ell_k}}^{-1} L_{\io,\theta,k}$ will be within distance $\varepsilon$ of the map $\pi^2 R_{A_{\io|_{\ell_k}}^{-1}\theta}$ for all $\theta$ in a set of positive $\mu_F$-measure. Recall that we are assuming the domination condition. 
\begin{claim}\label{claim-conjugate}
    For any $\varepsilon>0$, there exists a non-singular linear map $A_\varepsilon:\R^2\to\R^2$ such that $0^\perp \in A_\varepsilon(\theta^-(\Gamma^\N))$ and $A_\varepsilon(\theta(\Gamma^\N))\subseteq B(0,\varepsilon)$. 
\end{claim}
\begin{proof}[Proof of Claim]
    Note first that the domination condition ensures that $\theta^-(\Gamma^\N) \cap \theta(\Gamma^\N) = \emptyset$, in particular $d(\theta^-(\Gamma^\N), \theta(\Gamma^\N)) > 0$ since the sets are compact. Indeed, if $\mathcal{C}$ is the strongly invariant multicone of $\lbrace A_i\rbrace_{i\in\Gamma}$, then $\RP \setminus {\rm int}(\mathcal{C})$ is the strongly invariant multicone for $\lbrace A_i^{-1}\rbrace_{i\in\Gamma}$ and so it follows from Lemma \ref{lemma-suuntasuppeneeprelminary} that $\theta(\Gamma^\N) \subseteq {\rm int}(\mathcal{C})$ and $\theta^-(\Gamma^\N) \subseteq \RP \setminus {\rm int}(\mathcal{C})$. Now, pick any $\theta_0\in \theta^-(\Gamma^\N)$, and set $A_\varepsilon = {\rm diag}(1/(d(\theta_0, \theta(\Gamma^\N))\varepsilon),\ 1) R_{\theta_0}$.
\end{proof}

We are now ready to prove Proposition \ref{prop-parempijono}.

\begin{proof}[Proof of Proposition \ref{prop-parempijono}]
Let $\Phi = \lbrace \varphi_i\rbrace_{i\in\Gamma}$ be an affine iterated function system satisfying the irreducibility, domination and strong separation conditions. Let $\varepsilon>0$, let $\varepsilon_1>0$ be small with respect to $\varepsilon$ and let $\varepsilon_2>0$ be small with respect to $\varepsilon_1$. We will later see how small these numbers have to be chosen. Let $A_{\varepsilon_2}$ be the linear map obtained by applying Claim \ref{claim-conjugate} with parameter $\varepsilon_2$. We continue to denote the conjugated IFS $\lbrace A_{\varepsilon_2} \circ \varphi_i\circ A_{\varepsilon_2}^{-1}\rbrace_{i\in\Gamma}$ by $\Phi$. Let $\mu$ be a self-affine measure associated to $\Phi$, and denote by $\bmu$ the associated Bernoulli measure on $\Gamma^\N$ and by $\mu_F$ the associated Furstenberg measure on $\RP$.

Let $\beta>0$ be large enough as in Proposition \ref{prop-asymptotics}, and let the sequence $(\ell_k)_{k\in\N} = (\ell_k(\io,\beta))_{k\in\N}$ be defined as in \eqref{eq-lk}. Recall that the IFS $\Phi$ was conjugated so that $\theta(\Gamma^\N)\subseteq B(0,\varepsilon_2)$ and $0^\perp \in {\rm spt}\,\mu_F = \theta^-(\Gamma^\N)$. In particular, $\mu_F(B(0^\perp, \varepsilon_2)) > 0$ and for every $\io\in\Gamma^\N$ we have $B(0^\perp, \varepsilon_2) \subseteq B(\theta(\io)^\perp, 2\varepsilon_2)$.

Now, for $\bmu$-almost every $\io\in \Gamma^\N$ and $\mu_F$-almost every $\theta\in B(0^\perp, \varepsilon_2)$, it follows from Proposition \ref{prop-asymptotics} and Lemma \ref{lemma-suunnikas} that
\begin{align}\label{eq-prop43}
d_{\rm LP}(S_{k\beta-1}^*&T_{\Pi(\io)}\mu,\nonumber\\
& S_1^* \rho(\ell_k,(\io,\theta)) F_{\theta,\io}^{k} R_{A_{\io|_{\ell_k}}^{-1}\theta}H_{Q_{\io, A_{\io|_{\ell_k}}^{-1}\theta,k}} \mu_{Q_{\io, A_{\io|_{\ell_k}}^{-1}\theta,k}})
\to 0
\end{align}
as $k\to\infty$, where $F_{\theta,\io}^k$ is as in Lemma \ref{lemma-suunnikas}. Since $\theta\in B(\theta(\io)^\perp,2\varepsilon_2)$, the map $F_{\theta,\io}^k$ is within distance $2\varepsilon_2$ of ${\rm diag}(a_k, 1)$. In particular, from \eqref{eq-prop43} it follows that
\begin{align}\label{eq-prop43-2}
d_{\rm LP}(\pi^2 S_{k\beta-1}^*&T_{\Pi(\io)}\mu,\nonumber\\
&  \pi^2 S_1^* \rho(\ell_k,(\io,\theta)) R_{A_{\io|_{\ell_k}}^{-1}\theta}H_{Q_{\io, A_{\io|_{\ell_k}}^{-1}\theta,k}} \mu_{Q_{\io, A_{\io|_{\ell_k}}^{-1}\theta,k}})
< 3\varepsilon_2.
\end{align}
for every large enough $k$. Now, it would be relatively easy to conclude from Lemma \ref{lemma-fiberstructure1} that 
$$
\dlp (\pi^2 S_{k\beta-1}^*T_{\Pi(\io)}\mu,\ \rho(\ell_k,(\io,\theta)) \mu_{\sigma^{\ell_k}\io, A_{\io|_{\ell_k}}^{-1}\theta, k\beta-1+\log \alpha_1(\io|_{\ell_k})}) < 3\varepsilon_2
$$
for most $k$. However, we are seeking a version of the above for $\pi^2 \mu^{D_{k\beta}(\Pi(\io))}$ which is why some technical work is still required. In the following, let 
$$
\io^\alpha := 2^{\lfloor\alpha\rfloor}\Pi(\io) \mod 1.
$$
and note that for any $k\in\N$, 
\begin{equation}\label{eq-dyadicscaling}
S_{k\beta}T_{\Pi(\io)} D_{\lfloor k\beta\rfloor} (\Pi(\io)) = S_{k\beta - \lfloor k\beta\rfloor} T_{\io^{k\beta}}([0,1]^2).
\end{equation}
Since $D_{k\beta}(\Pi(\io))\subseteq B(\Pi(\io), 2^{-k\beta+1})$, it follows from \eqref{eq-dyadicscaling} that
\begin{equation}\label{eq-restrictionobservation}
\mu^{D_{k\beta}(\Pi(\io))} = (S_{k\beta-1}^* T_{\Pi(\io)} \mu)_{S_{-1+k\beta-\lfloor k\beta\rfloor}T_{\io^{k\beta}}([0,1]^2)}
\end{equation}
for every $k$. Now, we would like to return to \eqref{eq-prop43} and apply the restriction $(\cdot)_{S_{-1+k\beta-\lfloor k\beta\rfloor}T_{\io^{k\beta}}([0,1]^2)}$ to both measures while retaining a bound for the distance between the measures. According to Lemma \ref{claim-restriction}, this is possible as long as the square $S_{-1}T_{\io^{k\beta}}([0,1]^2)$ has large $S_{k\beta-1}^*T_{\Pi(\io)}\mu$-measure, and the neighbourhood of its boundary has small measure. In Claims \ref{claim-lowerboundformeasure} and \ref{claim-upperboundforboundary}, we perform these standard verifications.

\begin{claim}\label{claim-lowerboundformeasure}
    There exists $\delta>0$ such that for $\bmu$-almost every $\io\in\Gamma^\N$, there exists a set $\cN_\varepsilon^1\subseteq \N$ such that $\liminf_{n\to\infty} \frac{\#(\cN_\varepsilon^1 \cap [0,n])}{n} \geq 1-\varepsilon/2$ and for every $k\in \cN_\varepsilon^1$, 
    $$
S_{k\beta-1}^*T_{\Pi(\io)}\mu(S_{-1+k\beta-\lfloor k\beta\rfloor}T_{\io^{k\beta}}([0,1]^2)) \geq \delta.
    $$
\end{claim}
\begin{proof}[Proof of Claim]
By an ergodic theorem of Bourgain \cite[Theorem 2]{Bourgain1989} and strong mixing of the map $x\mapsto 2x\mod 1$, the sequence $(2^{\lfloor k\beta\rfloor} x\mod 1)_{k\in\N}$ equidistributes for the Lebesgue measure on $[0,1]^2$ for almost every $x$. Therefore By Fubini's theorem, by applying a random translation to $\mu$ (which does not affect the dimension of $\mu*\nu$), we may suppose that for $\bar{\mu}$-almost every $\io$, the sequence $(\io^{k\beta})_{k\in\N}$ equidistributes for the Lebesgue measure on $[0,1]^2$. In particular, for a small enough $\delta_0>0$ there exists $\cN_\varepsilon'\subseteq \N$ with $\liminf_{n\to\infty} \frac{\#(\cN'_\varepsilon\cap[0,n])}{n} \geq 1-\varepsilon/6$ such that 
\begin{equation*}
B(\io^{k\beta}, 4\delta_0) \subseteq [0,1]^2,\ \text{in particular,}\ B(0, \delta_0)\subseteq S_{-1+k\beta-\lfloor k\beta\rfloor} T_{\io^{k\beta}}([0,1]^2)
\end{equation*}
for every $k\in \cN_\varepsilon'$. On the other hand, by \cite[Proposition 1.19]{Hochmanpreprint},
$$
\lim_{n\to\infty} \frac{1}{n}\sum_{k=1}^n \frac{\log S_{k\beta-1}^* T_{\Pi(\io)} \mu(B(0, \delta_0))}{\log \delta_0}\leq \overline{\dim}_{\rm loc} \mu(\Pi(\io))<\infty
$$
for $\bar{\mu}$-a.e. $\io$, so there exists a $\delta>0$ and a set $\cN_\varepsilon''\subseteq \N$ such that $\liminf_{n\to\infty} \frac{\#(\cN_\varepsilon'' \cap[0,n])}{n} \geq 1-\varepsilon/6$ and 
\begin{equation*}
S_{k\beta-1}^*T_{\Pi(\io)} \mu(B(0,\delta_0)) \geq \delta > 0
\end{equation*}
for every $k\in \cN_\varepsilon''$. If we now set $\cN_\varepsilon^1 = \cN_\varepsilon' \cap \cN_\varepsilon''$, then for every $k\in \cN_\varepsilon^1$, we have
\begin{equation*}
S_{k\beta-1}^*T_{\Pi(\io)}\mu(S_{-1+k\beta-\lfloor k\beta\rfloor}T_{\io^{k\beta}}([0,1]^2)) \geq \delta
\end{equation*}
which concludes the proof of the claim.
\end{proof}
Next we will ensure that the neighbourhood of the boundary of $S_{-1+k\beta-\lfloor k\beta\rfloor}T_{\io^{k\beta}}([0,1]^2)$ has small measure for most $k$, which was another requirement of Lemma \ref{claim-restriction}. 
\begin{claim}\label{claim-upperboundforboundary}
    For $\bmu$-almost every $\io\in\Gamma^\N$, there exists a set $\cN_\varepsilon^2\subseteq \N$ such that $\lim_{n\to\infty}\frac{\#(\cN_\varepsilon^2\cap[0,n])}{n} \geq 1-\varepsilon/3$ and for every $k\in\cN_\varepsilon^2$, we have 
$$
S_{k\beta-1}^*T_{\Pi(\io)}\mu(S_{-1+k\beta-\lfloor k\beta\rfloor}T_{\io^{k\beta}}([0,1]^2)^{\varepsilon_2}) \leq S_{k\beta-1}^*T_{\Pi(\io)}\mu(S_{-1+k\beta-\lfloor k\beta\rfloor}T_{\io^{k\beta}}([0,1]^2)) + \varepsilon_1.
$$
\end{claim}
\begin{proof}[Proof of Claim]
For $s,t>0$ let $A_{s,t}\subseteq \mathcal{P}(\R^2)$ denote the closed set of probability measures giving mass at least $t$ to an $s$-neighbourhood of a line within distance at least $\delta_0$ and at most $1/2$ from the origin, that is,
$$
A_{s,t} = \lbrace \nu \in \cP(\R^2):\ \text{there exists a line}\ \ell\ \text{with}\ \delta_0<d(\ell, 0)\leq 1/2\ \text{and}\ \nu(\ell^s) \geq t\rbrace.
$$
If $\varepsilon_2$ was chosen small enough, then for $\bmu$-almost every $\io\in\Gamma^\N$ we have
$$
\limsup_{n\to\infty} \frac{\#\lbrace 0\leq k \leq n:\ S_{k\beta-1}^*T_{\Pi(\io)}\mu\in A_{\varepsilon_2,\varepsilon_1} \rbrace}{n} \leq \varepsilon/3.
$$
Indeed, if this was not the case, then for any $n\in\N$ we could find a weak-$^*$ accumulation point $P_n$ of $(\frac{1}{\ell}\sum_{k=1}^\ell \delta_{S_{k\beta-1}^*T_{\Pi(\io)}\mu})_{\ell\in \N}$ such that $P_n(A_{1/n,\varepsilon_1})\geq \varepsilon/4$. In particular, any accumulation point $P$ of the sequence $(P_n)_{n\in\N}$ would have
\begin{equation}\label{eq-lowerboundcontradiction}
P\left( \bigcap_n A_{1/n,\varepsilon_1}\right) = \lim_{n\to\infty} P(A_{1/n,\varepsilon_1}) \geq \lim_{n\to\infty} \limsup_{m\to\infty} P_m(A_{1/n,\varepsilon_1}) \geq \varepsilon/4.
\end{equation}
For every $n$, the measure $\int_0^\beta S_r^* P_n\,dr$ is a fractal distribution by \cite[Proposition 5.5]{Hochmanpreprint}. Since the function $\nu\mapsto \int_0^\beta S_r^*\nu \,dr$ is continuous and the space of fractal distributions is closed by \cite[Theorem 3.1]{KaenmakiSahlstenShmerkin2015}, also $P' = \int_0^\beta S_r^* P \,dr$ is a fractal distribution. From \eqref{eq-lowerboundcontradiction} it follows that $P'(\bigcap_{n} A_{1/n,\varepsilon_1/2}) > 0$, contradicting Lemma \ref{lemma-nolines0}. Therefore, if we define for every $\io\in\Gamma^\N$ the set $\cN_\varepsilon^2 = \lbrace k\in\N:\ S_{k\beta-1}^*T_{\Pi(\io)}\mu\in A_{\varepsilon_2, \varepsilon_1}\rbrace$, then for $\bmu$-almost every $\io$ we have $\liminf_{n\to\infty} \frac{\#(\cN_\varepsilon^2 \cap [0,n])}{n} \geq 1-\varepsilon/3$ as long as $\varepsilon_2$ is small enough. Clearly, the statement of the claim holds for every $k\in\cN_\varepsilon^2$.
\end{proof}

Let us now resume the main line of the proof. For every $k\in\cN_\varepsilon^1\cap\cN_\varepsilon^2$, it was verified in Claims \ref{claim-lowerboundformeasure} and \ref{claim-upperboundforboundary} that the measures in \eqref{eq-prop43} satisfy the conditions Lemma \ref{claim-restriction}, and so we have 
\begin{align}\label{eq-prop43-3}
&d_{\rm LP}\big(\mu^{D_{k\beta}(\Pi(\io))},\nonumber\\
&\qquad (S_1^*\rho(\ell_k,(\io,\theta)) {\rm diag}(a_k, 1) R_{A_{\io|_{\ell_k}}^{-1}\theta}H_{Q_{\io, A_{\io|_{\ell_k}}^{-1}\theta,k}} \mu_{Q_{\io, A_{\io|_{\ell_k}}^{-1}\theta,k}})_{S_{-1+k\beta-\lfloor k\beta\rfloor} T_{\io^{k\beta}}([0,1]^2)}\big) \nonumber\\
&\qquad\qquad < O(\varepsilon_1/\delta)
\end{align}
as long as $\varepsilon_2$ is small enough with respect to $\varepsilon_1$. All that remains to be done is to apply Lemma \ref{lemma-fiberstructure1}. Let $\varepsilon_3>0$ be small enough with respect to $\varepsilon>0$ and $\beta$. Let $X\subseteq \Gamma^\N\times\RP$ be the set given by Lemma \ref{lemma-fiberstructure1} and Egorov's theorem, such that $\bmu\times\mu_F(X)>1-\varepsilon_3$ and for all large enough $r_2\geq r_2(\varepsilon_3)$, every $0\leq r_1\leq O(\beta)$, every $(\io,\theta)\in X$ and any square $Q$ of side length $2^{k\beta-\lfloor k\beta\rfloor}$ containing the origin and with $R_\theta H_{Y_{\io,\theta,r_1,r_2}} \mu_{Y_{\io,\theta,r_1,r_2}}(Q)\geq \delta$, we have
\begin{equation}\label{eq-defofN4}
d_{\rm LP}\left( \pi^2(R_\theta H_{Y_{\io,\theta,r_1,r_2}} \mu_{Y_{\io,\theta,r_1,r_2}})_{Q},\ (\mu_{\io,\theta,r_1})_{\pi^2 Q}\right) < \varepsilon/2.
\end{equation}
For every $(\io,\theta)$ we let $\cN_\varepsilon^3\subseteq \N$ be the set of those $k$ for which $M^{\ell_k}(\io,\theta) \in X$. For $\bar{\mu}\times \mu_F$-almost every $(\io,\theta)$, we have $\liminf_{n\to\infty} \frac{\#\lbrace 0\leq k \leq n:\ M^k(\io,\theta) \in X\rbrace}{n} \geq 1-\varepsilon_3$ by Birkhoff's ergodic theorem, and since it is not difficult to verify from the definition of $\ell_k$ (in \eqref{eq-lk}) that $\ell_k \leq O_\beta(k)$, we have $\liminf_{n\to\infty} \frac{\#(\cN_\varepsilon^3\cap[0,n])}{n}\geq1-\varepsilon/3$ if $\varepsilon_3$ was chosen small enough.

Finally, setting $\cN_\varepsilon = \cN_\varepsilon^1\cap\cN_\varepsilon^2\cap\cN_\varepsilon^3$ and combining \eqref{eq-prop43-3} with \eqref{eq-defofN4}, we see that
$$
d_{\rm LP}\left(\pi^2\mu^{D_{k\beta}(\Pi(\io))},\left( \rho(\ell_k, (\io,\theta)) \mu_{\sigma^{\ell_k}\io,\,A_{\io_{\ell_k}}^{-1}\theta,\, k\beta -1+ \log \alpha_1(\io|_{\ell_k})}\right)_{I_k} \right) < \varepsilon/2 + O(\varepsilon_1/\delta) < \varepsilon,
$$
for every $k\in \cN_\varepsilon$, when $I_k:= \pi^2S_{-1+k\beta-\lfloor k\beta\rfloor}T_{\io^{k\beta}}([0,1]^2)$. This is what we set out to prove.
\end{proof}

\section{Bounding the local entropy averages}\label{section6}
In this section, we prove Claims \ref{claim-easycase}, \ref{claim-1} and \ref{claim-2}. In order to do this, we need to understand the dynamics of the sequences $(\mu_{\io|_{i_k}})_{k\in\N}$, $(\nu_{\jo|_{i_k}})_{k\in\N}$ and $(\pi^2 \mu^{D_{k\beta}(\Pi(\io))})_{k\in\N}$. 

\subsection{Auxiliary suspension flows}
Let $\mu$ and $\nu$ be self-affine measures associated to iterated function systems $\Phi = \lbrace \varphi_i(x) = A_i x + a_i\rbrace_{i\in\Gamma}$ and $\Psi = \lbrace \psi_j(x) = B_j x + b_j \rbrace_{j\in\Lambda}$ satisfying the total irreduciblity and hyperbolicity conditions. We require a cocycle acting on $\Lambda^\N\times\RP$ similar to $\rho$ of the previous section, which we also denote by $\rho$ since it plays an identical role: Let $O\subseteq \Lambda^\N \times \RP$ denote the open set of those $(\jo,\theta)$ for which $\langle e_{B_{j_0}^*\theta},\, B_{j_0}^* e_\theta\rangle <0$, where $e_\theta$ denotes the unit vector of $\theta$ with non-negative $y$-coordinate, and set
\begin{equation}\label{eq-roopsi}
\rho(k,(\jo,\theta)) = \prod_{\ell=1}^k (-1)^{\chi[O](M_*^\ell(\jo,\theta))}.
\end{equation}
Recall that $M_*:\Lambda^\N\times\RP\to\Lambda^\N\times\RP$ denotes the function $(\jo,\theta)\mapsto (\sigma\jo, B_{j_0}^*\theta)$.

Define the Hölder continuous functions 
\begin{align*}
f: \Gamma^\Z \times \RP \to \R,&\ (\io, \theta) \mapsto -\log \Vert A_{i_0} |_{A_{i_0}^{-1}\theta}\Vert\\
g: \Lambda^\Z \times \RP \to \R,&\ (\jo,\theta) \mapsto -\log \Vert B_{j_0}^* |_\theta\Vert
\end{align*}
and the sets 
\begin{align*}
Z_\Phi &:= \lbrace (\io, \theta, u,t):\ \io\in\Gamma^\Z, \theta\in \RP, u\in\lbrace-1,1\rbrace,0 \leq t \leq f(\io, \theta) \rbrace, \\
Z_\Psi &:= \lbrace (\jo,\theta,u,t):\ \jo \in \Lambda^\Z, \theta \in \RP, u\in\lbrace-1,1\rbrace, 0\leq t \leq g(\jo,\theta) \rbrace,
\end{align*}
both equipped with the identifications $(\io, \theta,u, f(\io, \theta)) = (M(\io, \theta),\rho(1,(\io,\theta))u, 0)$ and $(\jo,\theta,u,g(\jo,\theta)) = (M_*(\jo,\theta), \rho(1,(\jo,\theta))u,0)$.

Let $\mathcal{T}_s$ denote the flow induced by the positive reals on both $Z_\Phi$ and $Z_\Psi$, given by
$$
\mathcal{T}_s: (\ko,\theta,u, t) \mapsto (\ko,\theta,u, t+s)
$$
for every $s\geq 0$ and $(\ko,\theta,u, t)\in Z_\Phi \cup Z_\Psi$. 

Let $\gamma = \delta_{-1}/2 + \delta_1/2$. Denote by $\lambda_\Phi$ the measure $\bmu \times \mu_F\times \gamma \times\mathcal{L}$ restricted and normalized on $Z_\Phi$, and let $\lambda_\Psi$ denote the normalized restriction of $\bnu\times \nu_F^* \times \gamma\times \mathcal{L}$ on $Z_\Psi$. 

Denote by $\mathcal{Z}_\Phi$ and $\mathcal{Z}_\Psi$ the suspensions $(Z_\Phi, \lambda_\Phi, \mathcal{T}_s)$ and $(Z_\Psi, \lambda_\Psi, \mathcal{T}_s)$. It is not difficult to see that they are measure-preserving. Since the skew-product maps 
\begin{align*}
(\io,\theta,u) &\mapsto (M(\io,\theta), \rho(1,(\io,\theta)) u), \\
(\jo,\theta,u) &\mapsto (M_*(\jo,\theta), \rho(1,(\jo,\theta)) u)
\end{align*}
are $\bmu\times\mu_F\times\gamma$- and $\bnu\times\nu_F^*\times\gamma$-ergodic, respectively, by \cite[Corollary 5.4]{Schmidt1977}, it is standard that these suspensions are also ergodic. 

Let $\mathcal{Z}_\Phi' = \pi^{1,2,4} \mathcal{Z}_\Phi$ and $\mathcal{Z}_\Psi' = \pi^{1,2,4}\mathcal{Z}_\Psi$. Here $\pi^{1,2,4}$ denotes the projection onto the first, second and fourth coordinate. The systems $\mathcal{Z}_\Phi'$ and $\mathcal{Z}_\Psi'$ are suspension flows of $\Gamma^\Z\times\RP$ and $\Lambda^\Z\times \RP$ over $f$ and $g$, respectively, and preserve the measures $\lambda_\Phi' := \bmu\times\mu_F\times \mathcal{L}$ and $\lambda_\Psi' := \bnu\times\nu_F^*\times\mathcal{L}$, restricted and normalized on the sets $Z_\Phi' = \pi^{1,2,4}Z_\Phi$ and $Z_\Psi' = \pi^{1,2,4}Z_\Psi$. 

The key observation we require on the dynamics of $\mathcal{Z}_\Phi'\times \mathcal{Z}_\Psi'$ is the following. If one were able to prove this proposition without requiring the domination condition, then it is very likely that with little work the assumption of domination could be removed from Theorem \ref{theorem-main} as well. 

\begin{proposition}\label{claim-ergodicproduct}
If the systems $\Phi$ and $\Psi$ satisfy the domination condition and there exists a pair $(i,j)\in \Gamma\times\Lambda$ with
$$
\frac{\log |\lambda_1(A_i)|}{\log |\lambda_2(B_j)|} \not\in\Q,
$$
then for any $N_0\in\N$, the flow $\mathcal{Z}_\Phi' \times \mathcal{Z}_\Psi'$ is ergodic under the discrete-time map $\mathcal{T}_{\beta}$ for some real number\footnote{In an earlier version of the paper, it was claimed that $\beta$ can be taken to be an integer. However, the proof of this claim contained an error, and the applications of Proposition \ref{claim-ergodicproduct} have now been modified to accommodate for a non-integer $\beta$.} $\beta \geq N_0$. 
\end{proposition}

Suppose now that $\Phi$ and $\Psi$ satisfy the domination condition. Let $f':\Gamma^\Z \to (0,+\infty)$ denote the map $\io\mapsto f(\io^+, \theta^-(\io^-))$ which is well-defined and Hölder continuous by Lemma \ref{lemma-suuntasuppeneekaikkialla}, and let 
$$
Z_\Phi'' = \lbrace (\io, t):\ \io \in \Gamma^\Z, 0 \leq t \leq f'(\io)\rbrace.
$$
Let $\lambda''_\Phi$ denote the measure $\bmu \times \mathcal{L}$ restricted and normalized on $Z_\Phi''$, and write $\mathcal{Z}''_\Phi = (Z_\Phi'', \cT_s, \lambda''_\Phi)$ for the suspension of $\Gamma^\Z$ over $f'$. Similarly, let $\mathcal{Z}_\Psi''$ be the suspension of $\Lambda^\Z$ over $g'(\jo) := g(\jo^+, \theta^*(\jo^-))$.

We will show that the flow $\mathcal{Z}_\Phi'' \times \mathcal{Z}_\Psi''$ is ergodic, which will immediately imply ergodicity of $\mathcal{Z}_\Phi' \times\mathcal{Z}_\Psi'$ since the latter system is a factor of the former. Borrowing an idea of Bowen \cite{Bowen1975book}, let $r: \Gamma^\Z\to \Gamma^\Z$ denote the function which replaces all the negative coordinates of $\ldots i_{-1};i_0i_1\ldots$ by $i_0$. Then $f'$ is cohomologous to $h:= f'\circ r$, that is, 
\begin{equation}\label{eq-cohomologous}
f' = h + u - u \circ \sigma,
\end{equation}
where $u$ is the continuous function defined by $u(\io) = \sum_{k=0}^\infty (f'(\sigma^k\io) - f'(\sigma^k r(\io)))$. The sum converges since $f'$ is Hölder. In particular, the flow $\mathcal{Z}_\Phi''$ is conjugate to the suspension of $\Gamma^\Z$ over $h$, denoted by $\mathcal{Z}_\Phi^h = (Z_\Phi^h, \mathcal{T}_s, \lambda_\Phi^h)$, and the advantage of this is that the function $h$ depends only on the positive coordinates of $\Gamma^\Z$.  

\begin{lemma}\label{lemma-eigenvalues1}
If $\Phi$ satisfies the domination condition, then the eigenvalues of the flow $\mathcal{Z}_\Phi''$ are contained in the set $$\bigcap_{i \in \Gamma}(\log|\lambda_1(A_i)|)^{-1} \Q.$$ 
\end{lemma}

\begin{proof}
Since $\mathcal{Z}_\Phi''$ is conjugate to $\mathcal{Z}_\Phi^h$ and conjugate flows have the same eigenvalues, it suffices to prove the statement for $\mathcal{Z}_\Phi^h$. Let $\beta \neq 0$ be an eigenvalue of $\mathcal{Z}_\Phi^h$. By Lemma \ref{lemma-samatominaisarvot}, there exists $n\in\Z$ such that $n\beta$ is an eigenvalue of $(\mathcal{Z}_\Phi^h)^+$. 

Let $i \in \Gamma$, and let $\io_0 = \ldots iii \ldots \in \Gamma^\Z$. Now, $|\lambda_1(A_i)|^{-1} = |\lambda_2(A_i^{-1})| = \Vert A_i^{-1} |_{\theta(\io_0^-)} \Vert = \Vert A_i |_{\theta^-(\io_0^-)} \Vert^{-1}$. In particular, 
$$
f'(\io_0) = -\log \Vert A_i |_{A_i^{-1}\theta^-(\io_0^-)}\Vert =-\log|\lambda_1(A_i)|.
$$
From \eqref{eq-cohomologous} it is clear that also $h(\io_0) = -\log |\lambda_1(A_i)|$. 

Since $n\beta$ is an eigenvalue of $(\mathcal{Z}_\Phi^h)^+$, Proposition \ref{preliminary-weakmixing} asserts that there exists a \emph{continuous} eigenfunction $\phi$ for $n\beta$, so we have
\begin{equation}\label{eq-continuouseigenfunction}
\phi(\mathcal{T}_s(\io, t)) = e(n\beta s)\phi(\io, t)
\end{equation}
for \emph{every} $(\io, t) \in (Z_\Phi^h)^+$ and $s \geq 0$. Since $(\mathcal{Z}_\Phi^h)^+$ is ergodic, $|\phi|$ is constant everywhere, so we may let $\psi: (Z_\Phi^h)^+ \to \R$ be the real-valued function defined by $\phi(\io, t) = |\phi|e(\psi(\io,t))$ for every $(\io,t)\in (Z_\Phi^h)^+$. We obtain from \eqref{eq-continuouseigenfunction} that
$$
\psi(\mathcal{T}_s(\io,t)) = n\beta s + \psi(\io,t) + m(\io, t)
$$
for some integer-valued function $m: (Z_\Phi^h)^+ \to \Z$. Inserting the value $s = h(\io)$, we obtain
$$
\psi(\sigma \io, t) = n \beta h(\io) + \psi(\io, t) + m(\io, t)
$$
which is equivalent to
$$
h(\io) = (n\beta)^{-1} m(\io, t) + (n\beta)^{-1} \psi(\io,t) - (n\beta)^{-1} \psi(\sigma \io, t).
$$
In particular, $h(\io_0) = (n\beta)^{-1}m(\io_0, t) \in \beta^{-1} \Q$ since $\sigma(\io_0) = \io_0$. However, we saw above that $h(\io_0) = -\log|\lambda_1(A_i)|$, whence it follows that $\beta \in (\log |\lambda_1(A_i)|)^{-1}\Q$.
\end{proof}

\begin{lemma}\label{lemma-eigenvalues2}
If $\Psi$ satisfies the domination condition, then the eigenvalues of $\mathcal{Z}_\Psi''$ are contained in the set 
$$
\bigcap_{j \in \Lambda} (\log |\lambda_2(B_j)|)^{-1}\Q
$$
\end{lemma}

\begin{proof}
Let $\jo_0 = \ldots jjj\ldots$. Then $\jo_0$ is a fixed point for $\sigma$, and
$$
g'(\jo_0) = -\log \Vert B_j^*|_{\theta^*(\jo_0^-)}\Vert = -\log |\lambda_2(B_j^*)| = -\log |\lambda_2(B_j)|
$$
and by replacing $\mathcal{Z}_\Psi''$ with a conjugate flow, we can proceed exactly as in the proof of the previous claim. 
\end{proof}

\begin{proof}[Proof of Proposition \ref{claim-ergodicproduct}]
By Lemmas \ref{lemma-eigenvalues1} and \ref{lemma-eigenvalues2} and the assumption of Proposition \ref{claim-ergodicproduct}, the flows $\mathcal{Z}_\Phi''$ and $\mathcal{Z}_\Psi''$ have no common eigenvalues other than $0$. Thus by Proposition \ref{preliminary-ergodicproduct}, the product $\mathcal{Z}_\Phi''\times\mathcal{Z}_\Psi''$ is ergodic. Since it has at most countably many eigenvalues by Lemma \ref{preliminary-countablymanyeigenvalues}, for any $N_0\in\N$ there exists a real number $\beta>N_0$ such that it is also ergodic under the discrete-time map $\mathcal{T}_\beta$. 

Now, since we have the domination assumption in place, the system $\mathcal{Z}_\Phi'\times\mathcal{Z}_\Psi'$ is a factor of the system $\mathcal{Z}_\Phi''\times\mathcal{Z}_\Psi''$ through the map 
$$
(\io,t,\jo,s) \mapsto (\io^+, \theta^-(\io^-),t,\jo^+,\theta^*(\jo^-), s).
$$ 
Since factor maps preserve ergodicity, this proves the statement.
\end{proof}

\subsection{Dynamics of the sequence $(\mu^{D_{k\beta}(\Pi(\io))})_{k\in\N}$}\label{subsec-magnificationsofmu}
Let $\Phi= \lbrace \varphi_i(x) = A_i x + a_i\rbrace_{i\in\Gamma}$ be an affine iterated function system satisfying the irreducibility, domination and strong separation conditions, let $\mu$ be a self-affine measure associated to $\Phi$ and let $\mathcal{Z}_\Phi$ be as in the previous section. The reason for requiring domination and strong separation in this section (even though they were not required in the construction of $\mathcal{Z}_\Phi$) is that they allow us to use Proposition \ref{prop-parempijono} to relate $(\pi^2\mu^{D_{k\beta}(\Pi(\io))})_{k\in \N}$ to an orbit of a point in $\mathcal{Z}_\Phi$.

\begin{lemma}\label{lemma-geomlemma}
For $\bmu$-almost every $\io\in\Gamma^\N$ and every $\theta\neq\theta(\io)$, there exists a constant $C(\io,\theta) > 0$ such that 
$$
\lim_{n\to\infty} \frac{\Vert A_{\io|_n} |_{A_{\io|_n}^{-1}\theta}\Vert}{\alpha_1(A_{\io|_n})} = 2^{C(\io,\theta)}.
$$
\end{lemma}

\begin{proof}
Since $\Vert A_{\io|_n}|_{A_{\io|_n}^{-1}\theta}\Vert = |A_{\io|_n}B(0,1) \cap \theta|$, we have
\begin{align*}
\frac{\Vert A_{\io|_n} |_{A_{\io|_n}^{-1}\theta}\Vert}{\alpha_1(A_{\io|_n})} &=  \frac{|A_{\io|_n}B(0,1) \cap \theta|}{|A_{\io|_n}B(0,1)\cap \theta(A_{\io|_n})^\perp|}  \\
&= \frac{|A_{\io|_n}B(0,1) \cap \theta|}{|A_{\io|_n}B(0,1)\cap \theta(\io)^\perp|} \frac{|A_{\io|_n}B(0,1)\cap \theta(\io)^\perp|}{|A_{\io|_n}B(0,1)\cap \theta(A_{\io|_n})^\perp|}.
\end{align*}
Here, for $\bmu$-almost every $\io$, we have $\lim_{n\to\infty} \frac{|A_{\io|_n}B(0,1)\cap \theta(\io)^\perp|}{|A_{\io|_n}B(0,1)\cap \theta(A_{\io|_n})^\perp|} = 1$ by Lemma \ref{lemma-suuntasuppeneekaikkialla} and $\lim_{n\to\infty} \frac{|A_{\io|_n}B(0,1) \cap \theta|}{|A_{\io|_n}B(0,1)\cap \theta(\io)^\perp|} = (\cos d(\theta, \theta(\io)^\perp))^{-1}$ by Lemma \ref{lemma-suuntasuppeneekaikkialla} and basic geometry. This completes the proof with $C(\io,\theta) = -\log (\cos d(\theta, \theta(\io)^\perp))$. 
\end{proof}

In particular, for $\bmu\times\mu_F$-almost every $(\io,\theta)\in\Gamma^\N\times\RP$,
\begin{equation}\label{eq-geomeq}
\dlp(\mu_{M^{\ell_k}(\io,\theta),k\beta+\log\alpha_1(\io|_{\ell_k})},\ \mu_{M^{\ell_k}(\io,\theta),k\beta+\log\Vert A_{\io|_{\ell_k}}|_{A_{\io|_{\ell_k}}^{-1}\theta}\Vert - C(\io,\theta)})\to 0
\end{equation}
as $k\to\infty$, where the sequence $(\ell_k)_{k\in\N} = (\ell_k(\io,\beta))_{k\in\N}$ is as in \eqref{eq-lk}. Note that for large enough $k$,  
\begin{align*}
    \ell_k &= \max \lbrace n:\ \alpha_1(A_{\io|_n}) \geq 2^{-(k-1)\beta} \rbrace\\
    &= \max \lbrace n:\ \Vert A_{\io|_n} |_{A_{\io|_n}^{-1}\theta} \Vert^{-1} \leq 2^{(k-1)\beta - C(\io,\theta)} \rbrace \\
    &= \max \lbrace n:\ -\log \Vert A_{\io|_n} |_{A_{\io|_n}^{-1} \theta}\Vert \leq (k-1)\beta - C(\io,\theta) \rbrace.
\end{align*}
Combining this with the elementary observation that
$$
\Vert A_{i_1} A_{i_2} |_{A_{i_2}^{-1} A_{i_1}^{-1}\theta}\Vert = \Vert A_{i_1} |_{A_{i_1}^{-1} \theta}\Vert \cdot \Vert A_{i_2} |_{A_{i_2}^{-1} A_{i_1}^{-1}\theta}\Vert
$$
we see that in $Z_\Phi$ we have the identity
\begin{align}\label{eq-kN}
&(\io,\theta,u, (k-1)\beta - C(\io,\theta)) \nonumber \\
=\ &(M^{\ell_k}(\io,\theta),\rho(\ell_k,(\io, \theta))u, (k-1)\beta - C(\io,\theta) + \log \Vert A_{\io_{\ell_k}}|_{A_{\io|_{\ell_k}}^{-1}\theta}\Vert )
\end{align}
for every $(\io,\theta,u)\in\Gamma^\Z \times \RP\times \lbrace -1,1\rbrace$. 

Let $F$ denote the map $Z_\Phi  \to \mathcal{P}(\R^2)$, 
\begin{equation}\label{eq-definitionF}
(\io,\theta,u,t)\mapsto 2u\mu_{\io,\theta,t+\beta-1},
\end{equation}
for every $(\io,\theta,u,t)$ with $0\leq t< f(\io,\theta)$, where $2\mu(\cdot) := \mu(2^{-1}\cdot)$. 

By \eqref{eq-kN}, we have
\begin{align}\label{eq-Fmeasure}
&F(\io,\theta,{\rm Id}, (k-1)\beta-C(\io,\theta)) \nonumber\\
=\ &2\rho(\ell_k,(\io, \theta)) \mu_{M^{\ell_k}(\io,\theta),k\beta-1+C(\io,\theta) + \log \Vert A_{\io_{\ell_k}}|_{A_{\io|_{\ell_k}}^{-1}\theta}\Vert}.
\end{align}
Let $\varepsilon>0$. If we now replace $\Phi$ by a linearly conjugate system as in Proposition \ref{prop-parempijono} and $\mu$ is any fully supported self-affine measure associated to $\Phi$, then by equations \eqref{eq-geomeq} and \eqref{eq-Fmeasure}, the following holds:

For $\bmu$-a.e. $\io\in\Gamma^\N$ and all $\theta$ in a set of positive $\mu_F$-measure, there exists a sequence of intervals $(I_k)_k$ with $|I_k|= 2^{k\beta - \lfloor k\beta\rfloor}$ and a set $\cN_\varepsilon \subseteq \N$ such that $\liminf_{n\to\infty} \frac{\#(\cN_\varepsilon\cap[0,n])}{n} \geq 1-\varepsilon$ and
\begin{equation}\label{eq-slice}
d_{\rm LP}(\pi^2 \mu^{D_{k\beta}(\Pi(\io))},\ (F(\io,\theta,1, (k-1)\beta- C(\io,\theta)))_{I_k}) <\varepsilon
\end{equation}
for every $k\in\cN_\varepsilon$. 

\subsection{Dynamics of the sequence $(\nu_{\jo|_{i_k}})_{k\in\N}$}
Let $\Psi = \lbrace \psi_j(x) = B_j x+b_j\rbrace_{j\in\Lambda}$ be an affine iterated function system satisfying the total irreducibility and hyperbolicity conditions. We remark that in this subsection, no domination or any separation condition is required. Similarly as in the previous section, our aim is to relate the sequence $(\nu_{\jo|_{i_k}})_{k\in\N}$ to an orbit of a point in $\mathcal{Z}_\Psi$. In fact, it turns out to be more useful to study the sequence $(\pi_\theta \nu_{\jo|_{i_k}})_{k\in\N}$, for every $\theta\in\RP$. As we will demonstrate in this section, this sequence is related to $Z_\Phi$ through the function $G: Z_\Psi \to \mathcal{P}(\R^2)$,
\begin{equation}\label{eq-definitionG}
(\jo, \theta,v, t) \mapsto S_{t}^* vR_\theta \pi_\theta T_{\Pi(\jo)}\nu.
\end{equation}
For any $\theta\in\RP$, $x\in\R^2$ and matrix $B$, it follows from the identity $\langle Bx,\, e_\theta\rangle = \langle x,\, B^* e_\theta\rangle$ that
\begin{equation}\label{eq-notdifficult}
\pi_\theta(Bx) = \pm R_\theta^{-1}\Vert B^*|_\theta\Vert R_{B^*\theta} \pi_{B^*\theta}(x),
\end{equation}
were the sign is negative if and only if $\langle e_{B^*\theta},\, B^* e_\theta \rangle < 0$, recalling that $e_\theta$ denotes the unit vector of $\theta$ with non-negative $y$-coordinate. In particular, 
\begin{equation}\label{eq-lengthofpitheta}
|\pi_{\theta}(B_{\jo|_k} B(0,1))| = \Vert B_{\jo|_k}^* |_\theta\Vert
= \Vert B_{j_0}^* |_\theta\Vert  \cdot \ldots \cdot \Vert B_{j_k}^* |_{B_{j_{k-1}}^* \ldots B_{j_0}^* \theta}\Vert
\end{equation}
and since $T_{\Pi(\jo)}\circ \varphi_{\jo|_k} = B_{\jo|_k} \circ T_{\Pi(\sigma^k \jo)}$, we have
\begin{align}\label{eq-representationofnu}
\pi_\theta \nu_{\jo|_{k}} &= \pi_\theta S_{k\beta}^* T_{\Pi(\jo)} \psi_{\jo|_{k}}\nu \nonumber\\
&= R_{\theta}^{-1}\rho(k, (\jo,\theta))S^*_{k\beta+\log \Vert B_{\jo|_{k}}^*|_\theta\Vert} R_{B_{\jo|_{k}}^*\theta}\pi_{B_{\jo|_{k}}^*\theta} T_{\Pi(\sigma^{k}\jo)}\nu,
\end{align}
where $\rho$ denotes the cocycle that captures the sign of the right-hand side of \eqref{eq-notdifficult} and was defined in \eqref{eq-roopsi}. From \eqref{eq-lengthofpitheta} and Lemma \ref{lemma-suuntasuppeneekaikkialla} it readily follows that
\begin{equation}\label{eq-limit*}
    \lim_{n\to\infty} \frac{\Vert B_{\jo|_n}\Vert}{\Vert B_{\jo|_n}^*|\theta\Vert} = \cos d(\theta, \theta(\jo))
\end{equation}
for $\bnu$-almost every $\jo\in\Lambda^\N$ and $\theta\neq \theta(\jo)^\perp$. In particular, recalling that 
\begin{align*}
i_k = i_k(\jo,\beta) &= \min\lbrace n\in\N:\ \Vert B_{\jo|_n}\Vert \leq 2^{-k\beta}\rbrace \\
&=\min\lbrace n\in\N:\ \Vert B_{\jo|_n}^*|\theta\Vert \leq 2^{-k\beta + \log \cos( d(\theta,\theta(\jo)))}\rbrace,
\end{align*} 
for every large enough $k$, it follows from \eqref{eq-representationofnu} and \eqref{eq-limit*} that
\begin{equation}\label{eq-propertyofG}
\lim_{k\to\infty} \dlp(\pi_\theta\nu_{\jo|_{i_k}},\ R_\theta^{-1} G(\mathcal{T}(\jo, \theta, 1, k\beta-\log \cos d(\theta,\theta(\jo))))) = 0
\end{equation}
for $\bnu$-almost every $\jo\in\Lambda^\N$ and every $\theta\neq \theta(\jo)^\perp$. 

\subsection{Proof of Claims \ref{claim-easycase}, \ref{claim-1} and \ref{claim-2}}

We now proceed to prove Claims \ref{claim-easycase}, \ref{claim-1} and \ref{claim-2} using \eqref{eq-slice}, \eqref{eq-propertyofG} and Proposition \ref{claim-ergodicproduct}. Recall that \eqref{eq-slice} holds when $\Phi$ satisfies the irreducibility, domination and strong separation conditions, while \eqref{eq-propertyofG} holds when $\Psi$ satisfies the total irreducibility and hyperbolicity conditions. 

\begin{proof}[Proof of Claim \ref{claim-easycase}]
    Let $\mu$ and $\nu$ be as in the statement of the claim. We stress that in this proof, only total irreducibility, hyperbolicity and exponential separation is assumed of both measures. It will be convenient to replace the function $\eta\mapsto H_\beta(\eta)$ by the weak-$^*$ continuous substitute $\tilde{H}_\beta: \eta\mapsto\int_0^1 H_\beta(\delta_x*\eta)\,dx$. It follows from Lemma \ref{lemma-almostcontinuity} that $|\tilde{H}_\beta(\eta) - H_\beta(\eta)| \leq O(1)$, whence any bounds for $\frac{1}{\beta}\tilde{H}_\beta$ hold also for $\frac{1}{\beta}H_\beta$, up to an additive error of $O(1/\beta)$.
    
    Recall that $\mu_{\io|_{i_k}} = S_{k\beta}^* T_{\Pi(\io)} \varphi_{\io|_{i_k}}\mu$ and $\nu_{\jo|_{i_k}} = S_{k\beta}^* T_{\Pi(\jo)} \psi_{\jo|_{i_k}} \nu$ for some $\beta>0$ specified later. By Lemma \ref{lemma-suuntasuppeneeprelminary}, any accumulation points of $S_{k\beta} T_{\Pi(\io)} \varphi_{\io|_{i_k}}$ and $S_{k\beta} T_{\Pi(\jo)} \psi_{\jo|_{i_k}}$ as $k\to\infty$ are affine maps of rank $1$, with ranges $\theta(\io)$ and $\theta(\jo)$, respectively. In particular, for $\bmu\times\bnu$-almost every $(\io,\jo)$,
    \begin{equation}\label{eq-convolutionproduct}
    \lim_{k\to\infty} \dlp(\mu_{\io|_{i_k}}*\nu_{\jo|_{i_k}},\ \pi_{\theta(\io)}\mu_{\io|_{i_k}}*\pi_{\theta(\jo)}\nu_{\jo|_{i_k}}) = 0
    \end{equation}

    Recall the definition of the function $G$ from \eqref{eq-definitionG}, and note that by Theorem \ref{theorem-selfaffineprojections}, $\dim G(\jo,\theta,v,t) = \min\lbrace 1,\dim\nu\rbrace$ for $\lambda_\Psi$-almost every $(\jo,\theta,v,t)$. This is the only part in the proof where the exponential separation is used. Since the flows $\mathcal{Z}_\Phi$ and $\mathcal{Z}_\Psi$ have at most countably many eigenvalues by Lemma \ref{preliminary-countablymanyeigenvalues}, it follows from Lemma \ref{preliminary-eigenvalueergodic} that there exist arbitrarily large real numbers $\beta>0$ such that both $\mathcal{Z}_\Phi$ and $\mathcal{Z}_\Psi$ are ergodic under the map $\mathcal{T}_\beta$. Let $\varepsilon>0$ and $N_0\in \N$. Combining Egorov's theorem with the well-known fact that $\dim \eta \leq \liminf_{\beta\to\infty} \frac{1}{\beta}H_\beta(\eta)$ for any probability measure $\eta$, Birkhoff's ergodic theorem applied for $\mathcal{T}_\beta$ for some $\beta\geq N_0$ asserts that for $\lambda_\Psi$-a.e. $(\jo,\theta,v,t)$, 
    $$
    \lim_{n\to\infty} \frac{1}{n}\sum_{k=1}^n \frac{1}{\beta} \tilde{H}_\beta(G(\mathcal{T}_{k\beta}(\jo,\theta,v,t))) > \min \lbrace 1,\dim \nu\rbrace - \varepsilon/2.
    $$
    Fix this $\beta$ for the rest of the proof. By continuity of $\tilde{H}_\beta$, this actually holds for \emph{every} $0\leq t<g(\jo,\theta)$. Moreover, we can replace $(\jo,\theta,v,t)$ by $(\jo,\theta(\jo), 1,0)$ since the orbits of $\theta$ and $\theta(\jo)$ under $\cT_{\beta}$ are asymptotic by Lemma \ref{cor-almostcontraction}. Thus, recalling \eqref{eq-propertyofG}, we have for $\bnu$-almost every $\jo\in\Lambda^\N$ that
    \begin{equation}\label{eq-entropyofnu}
    \lim_{n\to\infty} \frac{1}{n}\sum_{k=1}^n \frac{1}{\beta} \tilde{H}_\beta(\pi_{\theta(\jo)}\nu_{\jo|_{i_k}}) > \min \lbrace 1,\dim \nu\rbrace - \varepsilon/2.
    \end{equation}
    Repeating the exact same argument for $\mu$, we also find that for $\bmu$-almost every $\io\in\Gamma^\N$, 
    \begin{equation}\label{eq-entropyofmu}
    \lim_{n\to\infty} \frac{1}{n}\sum_{k=1}^n \frac{1}{\beta} \tilde{H}_\beta(\pi_{\theta(\io)}\mu_{\io|_{i_k}}) > \min \lbrace 1, \dim\mu\rbrace - \varepsilon/2
    \end{equation}
   
    Now we wish to estimate the entropy of $\mu_{\io|_{i_k}}*\nu_{\jo|_{i_k}}$ using \eqref{eq-convolutionproduct} and the estimates \eqref{eq-entropyofmu} and \eqref{eq-entropyofnu}. However, for that we need a lower bound for the angle between $\theta(\io)$ and $\theta(\jo)$. 
    
    By Lemma \ref{lemma-differentangles} and Fubini, $d(\theta(\io), \theta(\jo)) > 0$ for $\bmu\times\bnu$-almost every $(\io,\jo)$. For every $n\in\N$, let $X_n = \lbrace (\io,\jo)\in\Gamma^\N\times\Lambda^\N:\ d(\theta(\io),\theta(\jo))\geq 1/n\rbrace$, and note that $\bmu\times\bnu(\bigcup_{n\in\N} X_n) = 1$. Fix now an integer $n$, and note that for every $(\io,\jo)\in X_n$, $\pi_{\theta(\jo)^\perp}|_{\theta(\io)} = R_{\theta(\jo)^\perp}^{-1} R_{\theta(\io)} \cos d(\theta(\io), \theta(\jo)^\perp)\pi_{\theta(\io)}$ where $\cos d(\theta(\io), \theta(\jo)^\perp)\geq 1/2n$. It follows from Lemmas \ref{lemma-chainruleapplication} and \ref{lemma-entropyofconvolution} that 
    \begin{align*}
    &\tilde{H}_\beta(\pi_{\theta(\io)}\mu_{\io|_{i_k}}*\pi_{\theta(\jo)}\nu_{\jo|_{i_k}}) \\
    \geq\ &\tilde{H}_\beta(S^*_{\log \cos d(\theta(\io), \theta(\jo)^\perp)}\pi_{\theta(\io)} \mu_{\io|_{i_k}}) + \tilde{H}_\beta(\pi_{\theta(\io)}\mu_{\io|_{i_k}}*\pi_{\theta(\jo)}\nu_{\jo|_{i_k}} |\pi_{\theta(\jo)^\perp}) - O(1) \\
    \geq\ &\tilde{H}_{\beta+\log\cos d(\theta(\io), \theta(\jo)^\perp)}(\pi_{\theta(\io)}\mu_{\io|_{i_k}}) + \tilde{H}_\beta(\pi_{\theta(\jo)}\nu_{\jo|_{i_k}}|\pi_{\theta(\jo)^\perp}) - O(1) \\
    =\ &\tilde{H}_{\beta+\log\cos d(\theta(\io), \theta(\jo)^\perp)}(\pi_{\theta(\io)}\mu_{\io|_{i_k}}) + \tilde{H}_\beta(\pi_{\theta(\jo)}\nu_{\jo|_{i_k}}) - O(1)
    \end{align*}
    for every $(\io,\jo)\in X_n$ and $k\in\N$, where the last equality follows from the fact that $\pi_{\theta(\jo)}\nu_{\jo|_{i_k}}$ is supported on the line $\theta(\jo)$. Dividing by $\beta$ and using \eqref{eq-entropyofmu}, \eqref{eq-entropyofnu}, \eqref{eq-convolutionproduct} and the fact that $\cos d(\theta(\io), \theta(\jo)^\perp)\geq 1/2n$ we obtain that for $\bmu\times\bnu$-a.e. $(\io,\jo)\in X_n$,
    \begin{align*}
    \lim_{n\to\infty} \frac{1}{n}\sum_{k=1}^n \frac{1}{\beta}\tilde{H}_\beta(\mu_{\io|_{i_k}}*\nu_{\jo|_{i_k}}) \geq \min \lbrace 1,\dim\mu \rbrace + \min \lbrace 1,\dim\nu\rbrace - \varepsilon,
    \end{align*}
    as long as $\beta$ is large enough with respect to $\varepsilon$ and $n$. Since $\bmu\times\bnu(\bigcup_{n\in\N} X_n) = 1$, this completes the proof. 
\end{proof}

\begin{proof}[Proof of Claim \ref{claim-1}]

Let $\varepsilon>0$, and let $\Phi$ denote the conjugated IFS as in the statement of the claim. Let $\mu$ be a fully supported self-affine measure associated to $\Phi$ with $\dim\mu>1$. Once again, it will be convenient to replace $H_\beta$ by the continuous substitute $\tilde{H}_\beta$ defined in the proof of Claim \ref{claim-easycase}. For $\bmu\times\mu_F$ a.e. $(\io, \theta)$ we have
$$
\lim_{\beta \to \infty} \frac{1}{\beta} \tilde{H}_\beta(\mu_{\io,\theta}) = \dim \mu_{\io,\theta} = \dim \mu - 1
$$
by Theorem \ref{thm-dimensionconservation}.

Recall the definition of the function $F: Z_\Phi \to \cP(\R^2)$ from \eqref{eq-definitionF}, and that $F': Z_\Phi' \to \cP(\R^2)$ is defined by $F'(\io,\theta,t) = (\io,\theta,1,t)$. Since $\mu_{\io,\theta}$ is a.s. exact dimensional and non-atomic, by a slight modification of the proof of Lemma \ref{lemma-continuityofentropy} there exists $N_0\in\N$ and a set $S\subseteq Z_\Phi'$ with $\lambda_\Phi'(S)\geq 1-\varepsilon$ such that 
\begin{equation}\label{eq-restrictiondim}
\frac{1}{\beta} \tilde{H}_\beta((F'(\io,\theta, t))_I) \leq \dim\mu - 1 + \varepsilon
\end{equation}
for every $(\io,\theta,t)\in S$, every $\beta\geq N_0$ and every interval $I$ on the $y$-axis such that $0\in I$ and $|I|\geq1$. Since the function $t\mapsto \frac{1}{\beta} \tilde{H}_\beta((F'(\io,\theta, t))_I)$ is continuous for $\bmu\times\mu_F$-almost every $(\io,\theta)$, \eqref{eq-restrictiondim} in fact holds for \emph{every} $0 \leq t \leq f(\io,\theta)$. 

By Proposition \ref{claim-ergodicproduct}, we may choose $\beta$ so that $\mathcal{Z}_\Phi'$ is ergodic under the discrete-time map $\mathcal{T}_\beta$. Birkhoff's ergodic theorem asserts that for $\lambda_\Phi'$-almost every $(\io,\theta,t)$, 
$$
\limsup_{n\to\infty} \frac{1}{n}  \sum_{k=1}^n \frac{1}{\beta} \tilde{H}_\beta((F'(\mathcal{T}_{(k-1)\beta}(\io,\theta, t)))_{I_k}) \leq \dim \mu - 1 + \varepsilon
$$
for any intervals $(I_k)_{k=1}^n$ on the $y$-axis such that $0\in I_k$ and $|I_k|\geq 1$. As reasoned above, this in fact holds for every $t\geq0$. Moreover, since the functions $F'$ and $F$ differ only by a reflection which does not affect entropy, for $\bmu\times\mu_F$-almost every $(\io,\theta)$ we have
\begin{align*}
    &\frac{1}{n}\sum_{k=m}^{n-1} \frac{1}{\beta} \tilde{H}_\beta(F(\io,\theta,1,(k-1)\beta- C(\io,\theta))_{I_k})  \\
    =\ &\frac{1}{n}\sum_{k=m}^{n-1} \frac{1}{\beta} \tilde{H}_\beta((F'(\io,\theta,(k-1)\beta-C(\io,\theta)))_{I_k}) \\
     \leq\ &\dim\mu-1+\varepsilon
\end{align*}
for large enough $n$, where $C(\io,\theta)$ is the constant of Lemma \ref{lemma-geomlemma} and $m\in\N$ is such that $(m-1)\beta - C(\io,\theta) \geq 0$. Finally, using \eqref{eq-slice} and fixing the choice of $(I_k)_k$, we obtain
$$
\limsup_{n\to\infty} \frac{1}{n}\sum_{k=0}^{n-1} \frac{1}{\beta}\tilde{H}_\beta(\pi^2\mu^{D_{k\beta}(\Pi(\io))}) \leq \dim\mu-1+2\varepsilon
$$
for $\bmu$-almost every $\io\in\Gamma^\N$.
\end{proof}

\begin{proof}[Proof of Claim \ref{claim-2}]

Let $\varepsilon>0$, and denote by $\Phi$ and $\Psi$ the conjugated IFSs as in the statement. Let $\mu$ and $\nu$ be fully supported self-affine measures associated to $\Phi$ and $\Psi$. Let $\tilde{H}_\beta$ denote the continuous substitute of $H_\beta$ as in the proof of Claim \ref{claim-easycase}. Recall the definitions of the functions $F: Z_\Phi \to \cP(\R^2)$ and $G: Z_\Psi \to \cP(\R^2)$ from \eqref{eq-definitionF} and \eqref{eq-definitionG}. By Marstrand's projection theorem, if $\tau$ and $\kappa$ are any exact-dimensional Borel measures on the line, then for Lebesgue-almost every $t\geq 0$ we have
$$
\dim(S_t \tau * \kappa) = \min \lbrace 1,\dim \tau + \dim \kappa\rbrace.
$$
Combining this with Theorems \ref{theorem-selfaffineprojections} and \ref{thm-dimensionconservation}, we find that
\begin{equation}\label{eq-convolutiondim}
    \dim(F(\io,\theta_1,u,t_1) * G(\jo,\theta_2,v,t_2)) = \min \lbrace 1, \dim \mu - 1 +\dim \nu \rbrace
\end{equation}
for $\lambda_\Phi \times\lambda_\Psi$-almost every $(\io,\theta_1,u,t_1,\jo,\theta_2,v,t_2)$. Since $u$ and $v$ are drawn from the uniform measure on $\lbrace -1,1\rbrace$, \eqref{eq-convolutiondim} holds for $\lambda_\Phi'\times\lambda_\Psi'$-almost every $(\io,\theta_1,t_1,\jo,\theta_2,t_2)$ and every $(u,v)\in\lbrace -1,1\rbrace^2$

Now, for $\bmu\times\mu_F$-a.e. $(\io,\theta)$, any interval that contains the origin has positive $\mu_{\io,\theta}$-measure. Therefore, for any $\varepsilon> 0$, applying Lemma \ref{lemma-continuityofentropy}, Egorov's theorem and \eqref{eq-convolutiondim} gives a set $S \subseteq Z_\Phi'\times Z_\Psi'$ with $\lambda_\Phi' \times \lambda_\Psi'(S) > 1-\varepsilon$ and an integer $N_0$ such that 
\begin{align*}
    \frac{1}{\beta}\tilde{H}_\beta(F(\io,\theta_1,u,t_1)_{I} * G(\jo,\theta_2, v,t_2)) 
    &\geq\ \dim(F(\io,\theta_1,u,t_1) * G(\jo,\theta_2,v,t_2))-\varepsilon \\
    &\geq\ \min \lbrace 1, \dim\mu - 1 + \dim\nu \rbrace - \varepsilon 
\end{align*}
for every $\beta\geq N_0$, $(\io,\theta_1,t_1,\jo,\theta_2,t_2) \in S$, $(u,v)\in\lbrace -1,1\rbrace^2$ and any interval $I$ such that $0\in I$ and $|I|\geq 1$. Tt follows from Lemma \ref{lemma-almostcontinuity} that also
$$
\frac{1}{\beta}\tilde{H}_\beta(F(\io,\theta_1,u,t_1)_{I} * G(\jo,\theta_2,v, t_2)) \geq \min \lbrace 1, \dim\mu - 1 + \dim\nu \rbrace - \varepsilon
$$
for every $\beta\geq N_0$, $(\io,\theta_1,t_1,\jo,\theta_2,t_2) \in S$, $(u,v)\in\lbrace-1,1\rbrace^2$ and any interval $I$ such that $0\in I$ and $|I|\geq 1$. 

Proposition \ref{claim-ergodicproduct} asserts that the number $\beta$ can be chosen to that $\mathcal{Z}_\Phi' \times \mathcal{Z}_\Psi'$ is ergodic under $\mathcal{T}_\beta$. It follows from Birkhoff's ergodic theorem that for $\lambda_\Phi'\times\lambda_\Psi'$-almost every $(\io,\theta_1,t_1,\jo,\theta_2,t_2)$, $\lim_{n\to\infty} \frac{1}{n}\sum_{k=1}^n \textbf{1}_S(\mathcal{T}_{k\beta}(\io,\theta_1,t_1,\jo,\theta_2,t_2)) \geq 1-\varepsilon$ and consequently, 
\begin{align}\label{eq-convolutionlowerbound}
&\liminf_{n\to\infty} \frac{1}{n}\sum_{k=1}^n\frac{1}{\beta}\tilde{H}_\beta(F(\mathcal{T}_{k\beta}(\io,\theta_1,1,t_1))_{I_k} * G(\mathcal{T}_{k\beta}(\jo,\theta_2,1,t_2))) \nonumber\\
\geq\ &\min \lbrace 1, \dim\mu - 1 + \dim\nu \rbrace - 2\varepsilon.
\end{align}
Arguing as in the proofs of Claims \ref{claim-easycase} and \ref{claim-1} and using continuity of convolution, we see that the above holds for every $0\leq t_1 \leq f(\io,\theta_1)$ and $0\leq t_2\leq g(\io,\theta_2)$.

Note that $G(\jo,\theta_2,1,t_2)$ is continuous in $\theta_2$, uniformly over $(\jo,\theta_2,t_2)\in Z_\Psi$, and that for $\bnu\times\nu_F^*$-almost every $(\jo,\theta)\in\Lambda^\N\times\RP$ we have $d(B_{\jo|_k}^*\theta_2,\ B_{\jo|_k}^* 0^\perp) \to 0$ as $k\to \infty$, by Lemma \ref{cor-almostcontraction}. Therefore, we may replace $\theta_2$ by $0^\perp$ so that \eqref{eq-convolutionlowerbound} still holds. Finally, applying \eqref{eq-slice} and \eqref{eq-propertyofG} to \eqref{eq-convolutionlowerbound} with proper choices of $t_1$ and $t_2$, we obtain
$$
\liminf_{n\to\infty} \frac{1}{n}\sum_{k=0}^{n-1} \frac{1}{\beta} \tilde{H}_\beta( \pi^2 \mu^{D_{k\beta}(\Pi(\io))} * \pi^2 \nu_{\jo|_{i_k}}) \geq \min \lbrace 1, \dim \mu - 1+ \dim \nu\rbrace - 6\varepsilon
$$
for $\bmu$-almost every $\io\in\Gamma^\N$ and $\bnu$-almost every $\jo\in\Lambda^\N$.
\end{proof}

\bibliographystyle{abbrv}
\bibliography{Bibliography}

\begin{thebibliography}{10}

\bibitem{Barany2015}
B.~B{\'a}r{\'a}ny.
\newblock On the {L}edrappier-{Y}oung formula for self-affine measures.
\newblock {\em Math. Proc. Cambridge Philos. Soc.}, 159(3):405--432, 2015.

\bibitem{BaranyHochmanRapaport2019}
B.~B\'{a}r\'{a}ny, M.~Hochman, and A.~Rapaport.
\newblock Hausdorff dimension of planar self-affine sets and measures.
\newblock {\em Invent. Math.}, 216(3):601--659, 2019.

\bibitem{BaranyKaenmaki2017}
B.~B\'{a}r\'{a}ny and A.~K\"{a}enm\"{a}ki.
\newblock Ledrappier-{Y}oung formula and exact dimensionality of self-affine
  measures.
\newblock {\em Adv. Math.}, 318:88--129, 2017.

\bibitem{BaranyKaenmakiMorris2020}
B.~B\'{a}r\'{a}ny, A.~K\"{a}enm\"{a}ki, and I.~D. Morris.
\newblock Domination, almost additivity, and thermodynamic formalism for planar
  matrix cocycles.
\newblock {\em Israel J. Math.}, 239(1):173--214, 2020.

\bibitem{BaranyKaenmakiPyoralaWu2023}
B.~B\'{a}r\'{a}ny, A.~K\"{a}enm\"{a}ki, A.~Py\"{o}r\"{a}l\"{a}, and M.~Wu.
\newblock Scaling limits of self-conformal measures.
\newblock 2023.
\newblock Preprint, available at arXiv:2308.11399.

\bibitem{BaranyKaenmakiRossi2021}
B.~B\'{a}r\'{a}ny, A.~K\"{a}enm\"{a}ki, and E.~Rossi.
\newblock Assouad dimension of planar self-affine sets.
\newblock {\em Trans. Amer. Math. Soc.}, 374(2):1297--1326, 2021.

\bibitem{BaranyKaenmakiYu2021}
B.~B\'{a}r\'{a}ny, A.~K\"{a}enm\"{a}ki, and H.~Yu.
\newblock Finer geometry of planar self-affine sets.
\newblock 2021.
\newblock Preprint, available at arXiv:2107.00983.

\bibitem{BenoistQuint2016}
Y.~Benoist and J.-F. Quint.
\newblock {\em Random walks on reductive groups}, volume~62 of {\em Ergebnisse
  der Mathematik und ihrer Grenzgebiete. 3. Folge. A Series of Modern Surveys
  in Mathematics [Results in Mathematics and Related Areas. 3rd Series. A
  Series of Modern Surveys in Mathematics]}.
\newblock Springer, Cham, 2016.

\bibitem{BochiGourmelon2009}
J.~Bochi and N.~Gourmelon.
\newblock Some characterizations of domination.
\newblock {\em Math. Z.}, 263(1):221--231, 2009.

\bibitem{BougerolLacroix1985}
P.~Bougerol and J.~Lacroix.
\newblock {\em Products of random matrices with applications to
  {S}chr\"{o}dinger operators}, volume~8 of {\em Progress in Probability and
  Statistics}.
\newblock Birkh\"{a}user Boston, Inc., Boston, MA, 1985.

\bibitem{Bourgain1989}
J.~Bourgain.
\newblock Pointwise ergodic theorems for arithmetic sets.
\newblock {\em Inst. Hautes \'Etudes Sci. Publ. Math.}, (69):5--45, 1989.
\newblock With an appendix by the author, Harry Furstenberg, Yitzhak Katznelson
  and Donald S. Ornstein.

\bibitem{Bowen1975book}
R.~Bowen.
\newblock {\em Equilibrium states and the ergodic theory of {A}nosov
  diffeomorphisms}.
\newblock Lecture Notes in Mathematics, Vol. 470. Springer-Verlag, Berlin-New
  York, 1975.

\bibitem{BruceJin2022}
C.~Bruce and X.~Jin.
\newblock Furstenberg sumset conjecture and {M}andelbrot percolations.
\newblock 2022.
\newblock Preprint, available at arXiv:2211.16410.

\bibitem{CoverThomas2006}
T.~M. Cover and J.~A. Thomas.
\newblock {\em Elements of information theory}.
\newblock Wiley-Interscience [John Wiley \& Sons], Hoboken, NJ, second edition,
  2006.

\bibitem{Davies1995}
E.~B. Davies.
\newblock {\em Spectral theory and differential operators}, volume~42 of {\em
  Cambridge Studies in Advanced Mathematics}.
\newblock Cambridge University Press, Cambridge, 1995.

\bibitem{EdekoGerlachKuhner2019}
N.~Edeko, M.~Gerlach, and V.~K\"{u}hner.
\newblock Measure-preserving semiflows and one-parameter {K}oopman semigroups.
\newblock {\em Semigroup Forum}, 98(1):48--63, 2019.

\bibitem{Falconer1990}
K.~J. Falconer.
\newblock {\em Fractal geometry}.
\newblock John Wiley \& Sons, Ltd., Chichester, 1990.
\newblock Mathematical foundations and applications.

\bibitem{FanLauRao2002}
A.-H. Fan, K.-S. Lau, and H.~Rao.
\newblock Relationships between different dimensions of a measure.
\newblock {\em Monatsh. Math.}, 135(3):191--201, 2002.

\bibitem{Feng2023}
D.-J. Feng.
\newblock Dimension of invariant measures for affine iterated function systems.
\newblock {\em Duke Math. J.}, 172(4):701--774, 2023.

\bibitem{FergusonFraserSahlsten2015}
A.~Ferguson, J.~M. Fraser, and T.~Sahlsten.
\newblock Scaling scenery of {$(\times m,\times n)$} invariant measures.
\newblock {\em Adv. Math.}, 268:564--602, 2015.

\bibitem{Fraser2022}
J.~M. Fraser.
\newblock The {F}ourier spectrum and sumset type problems.
\newblock {\em Math. Ann.}, 2024.

\bibitem{Furstenberg1963}
H.~Furstenberg.
\newblock Noncommuting random products.
\newblock {\em Trans. Amer. Math. Soc.}, 108:377--428, 1963.

\bibitem{Hall2013}
B.~C. Hall.
\newblock {\em Quantum theory for mathematicians}, volume 267 of {\em Graduate
  Texts in Mathematics}.
\newblock Springer, New York, 2013.

\bibitem{Hochmanpreprint}
M.~Hochman.
\newblock Dynamics on fractals and fractal distributions.
\newblock 2010.
\newblock Preprint, available at arXiv:1008.3731.

\bibitem{Hochman2012}
M.~Hochman.
\newblock Geometric rigidity of {$\times m$} invariant measures.
\newblock {\em J. Eur. Math. Soc. (JEMS)}, 14(5):1539--1563, 2012.

\bibitem{Hochmanlecturenotes}
M.~Hochman.
\newblock Notes on {E}rgodic {T}heory. {H}ebrew {U}niversity of {J}erusalem.
\newblock 2013.
\newblock Available at
  https://math.huji.ac.il/~mhochman/courses/ergodic-theory-2012/notes.final.pdf.

\bibitem{Hochman2014}
M.~Hochman.
\newblock On self-similar sets with overlaps and inverse theorems for entropy.
\newblock {\em Ann. of Math. (2)}, 180(2):773--822, 2014.

\bibitem{Hochman2015}
M.~Hochman.
\newblock On self-similar sets with overlaps and inverse theorems for entropy
  in {$\R^d$}.
\newblock 2015.
\newblock Preprint, available at arXiv:1503.09043.

\bibitem{HochmanRapaport2021}
M.~Hochman and A.~Rapaport.
\newblock Hausdorff dimension of planar self-affine sets and measures with
  overlaps.
\newblock {\em J. Eur. Math. Soc. (JEMS)}, 2021.
\newblock To appear, available at arXiv:1904.09812.

\bibitem{HochmanShmerkin2012}
M.~Hochman and P.~Shmerkin.
\newblock Local entropy averages and projections of fractal measures.
\newblock {\em Ann. of Math. (2)}, 175(3):1001--1059, 2012.

\bibitem{KaenmakiKoivusaloRossi2017}
A.~K\"{a}enm\"{a}ki, H.~Koivusalo, and E.~Rossi.
\newblock Self-affine sets with fibred tangents.
\newblock {\em Ergodic Theory Dynam. Systems}, 37(6):1915--1934, 2017.

\bibitem{KaenmakiSahlstenShmerkin2015}
A.~K\"{a}enm\"{a}ki, T.~Sahlsten, and P.~Shmerkin.
\newblock Structure of distributions generated by the scenery flow.
\newblock {\em J. Lond. Math. Soc. (2)}, 91(2):464--494, 2015.

\bibitem{Kempton2015}
T.~Kempton.
\newblock The scenery flow for self-affine measures.
\newblock 2015.
\newblock Preprint, available at arXiv:1505.01663.

\bibitem{Moreira1998}
C.~G. T. d.~A. Moreira.
\newblock Sums of regular {C}antor sets, dynamics and applications to number
  theory.
\newblock volume~37, pages 55--63. 1998.
\newblock International Conference on Dimension and Dynamics (Miskolc, 1998).

\bibitem{NazarovPeresShmerkin2012}
F.~Nazarov, Y.~Peres, and P.~Shmerkin.
\newblock Convolutions of {C}antor measures without resonance.
\newblock {\em Israel J. Math.}, 187:93--116, 2012.

\bibitem{ParryPollicott1990}
W.~Parry and M.~Pollicott.
\newblock Zeta functions and the periodic orbit structure of hyperbolic
  dynamics.
\newblock {\em Ast\'{e}risque}, (187-188):268, 1990.

\bibitem{PeresShmerkin2009}
Y.~Peres and P.~Shmerkin.
\newblock Resonance between {C}antor sets.
\newblock {\em Ergodic Theory Dynam. Systems}, 29(1):201--221, 2009.

\bibitem{Rossi2021}
E.~Rossi.
\newblock Visible part of dominated self-affine sets in the plane.
\newblock {\em Ann. Fenn. Math.}, 46(2):1089--1103, 2021.

\bibitem{RossiShmerkin2020}
E.~Rossi and P.~Shmerkin.
\newblock On measures that improve {$L^q$} dimension under convolution.
\newblock {\em Rev. Mat. Iberoam.}, 36(7):2217--2236, 2020.

\bibitem{Schmidt1977}
K.~Schmidt.
\newblock {\em Cocycles on ergodic transformation groups}.
\newblock Macmillan Lectures in Mathematics, Vol. 1. Macmillan Co. of India,
  Ltd., Delhi, 1977.

\bibitem{Walters1982}
P.~Walters.
\newblock {\em An introduction to ergodic theory}, volume~79 of {\em Graduate
  Texts in Mathematics}.
\newblock Springer-Verlag, New York, 1982.

\end{thebibliography}

\end{document}